\documentclass[11pt]{amsart}
\usepackage{amsmath,amsfonts,amssymb,mathrsfs}
\usepackage{amssymb,mathrsfs,graphicx,subcaption}
\usepackage{mathtools}
\usepackage{bbm}
\usepackage{color}
\usepackage{smartdiagram}
\usepackage{ulem}
\usepackage{cancel}
\usepackage{pgfplots}

\graphicspath{{./figure/}}

\title[Hydrodynamic limit of a coupled 
Cucker-Smale system]{Hydrodynamic limit of a coupled  
Cucker-Smale system with strong and weak internal variable relaxation}

\author[Kim]{Jeongho Kim}
\address[Jeongho Kim]{\newline Institute of New Media and Communications\newline Seoul National University, Seoul 08826, Korea}
\email{jhkim206@snu.ac.kr}

\author[Poyato]{David Poyato}
\address[David Poyato]{\newline Institut Camille Jordan (ICJ), UMR 5208 CNRS \& Universit\'e Claude Bernard Lyon 1, 69100 Villeurbanne, France, and Research Unit ``Modeling Nature'' (MNat), Universidad de Granada, Granada, 18071, Spain}
\email{poyato@math.univ-lyon1.fr}

\author[Soler]{Juan Soler}
\address[Juan Soler]{Departamento de Mathem\'atica Aplicada and Research Unit ``Modeling Nature"(MNat), Universidad de Granada, Granada, 18071, Spain}
\email{jsoler@ugr.es}

\newtheorem{theorem}{Theorem}[section]
\newtheorem{lemma}{Lemma}[section]
\newtheorem{corollary}{Corollary}[section]
\newtheorem{proposition}{Proposition}[section]
\newtheorem{remark}{Remark}[section]

\newtheorem{definition}{Definition}[section]

\newcommand{\bbr}{\mathbb R}

\newcommand{\e}{\varepsilon}

\DeclareMathOperator{\supp}{supp}
\DeclareMathOperator{\diam}{diam}

\begin{document}
	
	\date{\today}
	
	
	\thanks{\textbf{Acknowledgment.} This project has received funding from the European Research Council (ERC) under the European Union’s Horizon 2020 research and innovation programme (grant agreement No 639638), the MECD (Spain) research grant FPU14/06304, the MINECO-Feder (Spain) research grant number RTI2018-098850-B-I00, the Junta de Andalucia (Spain) Projects PY18-RT-2422 \& A-FQM-311-UGR18 (D.P, J.S)}
	
	\begin{abstract}
	In this paper, we present the hydrodynamic limit of a multiscale system describing the dynamics of two populations of agents with alignment interactions and the effect of an internal variable. It consists of a kinetic equation coupled with an Euler-type equation inspired by the thermomechanical Cucker--Smale (TCS) model. We propose a novel drag force for the fluid-particle interaction reminiscent of Stokes' law. Whilst the macroscopic species is regarded as a self-organized background fluid that affects the kinetic species, the latter is assumed sparse and does not affect the macroscopic dynamics. We propose two hyperbolic scalings, in terms of a strong and weak relaxation regime of the internal variable towards the background population. Under each regime, we prove the rigorous hydrodynamic limit towards a coupled system composed of two Euler-type equations. Inertial effects of momentum and internal variable in the kinetic species disappear for strong relaxation, whereas a nontrivial dynamics for the internal variable appears for weak relaxation. Our analysis covers both the case of Lipschitz and weakly singular influence functions.
	\end{abstract}
	
	\maketitle

	%
	%
	\section{Introduction}\label{sec:1}
	\setcounter{equation}{0}

	Uncovering the mechanisms responsible for the cooperation in a group of agents is a topic of great interest and current relevance. These processes involve qualitatively different emergent phenomena such as reaching consensus without centralized control or the formation of cooperative clusters within a collective group. Further examples of collective dynamics are the emergence of common languages in primitive societies, the collective migration of animal populations, the biochemical interactions leading to the activation of target genes, or the way social populations reach a consensus of social opinion.

	Cooperation occurs when the interacting agents exchange information about their state. As active particles, not only mechanical state is involved, but also their microstate, defined by internal activation variables or behavioral variables \cite{BHO20,C15,HLe09,HR17,LRS20}. The final state manifests a rich appearance of consensus patterns that can lead to a process of aggregation or clustering induced by motility. 
	In this paper, we are interested in a Cucker-Smale-type flocking model that contains the effect of some internal variables.

	To include a microstate in the modeling of the Cucker-Smale system, we were inspired by how the internal variable associated with temperature acts in the recently proposed thermomechanical Cucker-Smale model \cite{HR17}. In addition, we consider a multiscale meso-macro coupled system composed of two types of populations. They respectively represent the ensemble of a gas of interacting particles at the kinetic/mesoscopic scale and a background population at the fluid/macroscopic scale, both including the effect of alignment interactions and the internal or behavioral variable. We account for two types of interactions: one is given by ``local" mean-field self-interactions within the same population, and the other determines ``global'' mean-field cross-interactions between populations. Further previous fluid-particle systems have been proposed in the literature in different settings, see e.g. \cite{BCHK12,BDGM09,CHJK19,CHJK20,GJV04-1,GJV04-2,H98}. For simplicity, we assume that the population described by the background fluid is much larger than that of the particle ensemble. This implies that whilst the macroscopic species self-organizes and affects the kinetic species, the latter can be considered sparse and does not affect the background dynamics. Of course, a non-negligible effect of particles over the fluid could be considered if both populations were comparable in size. The novelty of the newly introduced multiscale model is the systematic derivation of the interacting meso-macro model from the micro-micro interacting particle model, using the appropriate weight assumption for the mean-field scaling.

	To make these ideas more concrete, let us recall some collective behavior models of self-propelled individuals that have been studied extensively for the last decades. The main interest in those systems is the emergence of aggregation, swarming, and flocking of agents. Among them, the Cucker-Smale (C-S) model is one of the most well-known systems in flocking dynamics \cite{CS07}. It can be described as a system of ODEs governing the dynamics of the position and velocity pair $(x_i,v_i)\in \bbr^{2d}$ of the $i$-th agent using Newton's law:
	\begin{align}
	\begin{aligned}\label{A-0}
	&\frac{dx_i}{dt}=v_i,\quad i=1,2,\ldots, N,\quad t>0,\\
	&\frac{dv_i}{dt}=\frac{\kappa}{N}\sum_{j=1}^N \phi(x_i-x_j)(v_j-v_i).
	\end{aligned}
	\end{align}
	
	Here, $\kappa$ is a positive constant called coupling strength, and $\phi:\bbr^d\to\bbr_+$ is called an influence function, which is a general radially symmetric function, and non-increasing with respect to the radius of the argument, providing the spatial dependency of communication between agents. When the number of particles is large enough, the dynamics of C-S system \eqref{A-0} can be described effectively from the mesoscopic and macroscopic points of view. Let $f=f(t,x,v)$ be the one-particle density function on the state space $\bbr^{2d}$ at time $t$. Then, the evolution of $f$ is governed by the following Vlasov-type equation \cite{HT08}:
	\begin{align}
	\begin{aligned} \label{A-0-1}
	&\partial_t f+v\cdot\nabla_xf +\nabla_v\cdot(F[f]f)=0,\quad (t,x,v)\in\mathbb{R}_+\times \bbr^{2d},\\
	&F[f](t,x,v):=\kappa\int_{\bbr^{2d}}\phi(x-x_*)(v_*-v)f(t,x_*,v_*)\,dx_*\,dv_*.
	\end{aligned}
	\end{align}
	
	In the same literature, the macroscopic (or hydrodynamic) description of C-S model was also derived from the kinetic equation \eqref{A-0-1} by taking the following macroscopic variables:
	\[\rho(t,x):=\int_{\bbr^d} f(t,x,v)\,dv,\quad (\rho u)(t,x):=\int_{\bbr^d}vf(t,x,v)\,dv,\]
	and using formal mono-kinetic ansatz $f(t,x,v)=\rho(t,x)\delta_{u(t,x)}(v)$:
	\begin{align}
	\begin{aligned}\label{A-0-2}
	&\partial_t \rho +\nabla\cdot(\rho u)=0,\quad (t,x)\in\bbr_+\times\bbr^d,\\
	&\partial_t(\rho u)+ \nabla\cdot(\rho u\otimes u) =\int_{\bbr^d} \phi(x-x_*)(u(t,x_*)-u(t,x))\rho(t,x)\rho(t,x_*)\,dx_*.
	\end{aligned}
	\end{align}
	
	Due to its possible application in engineering, in particular, controlling UAVs or spaceships \cite{PEG09}, the C-S model has been extensively studied from various point of view. For example, emergent behaviors \cite{CFRT10, HT08, HL09}, mean-field limit \cite{HL09, MP18}, hydrodynamic limit \cite{FK19,KMT15,PS17} of the C-S model have been studied. For details, we refer to the survey paper \cite{CHL16} on the C-S model and the references therein.
	
	This type of interaction between agents is reminiscent of the dynamics of opinion, whose development was a pioneer in this field. Hegselmann and Krause \cite{Kra,H-K} introduced a nonlinear opinion dynamics model in which each individual's opinion is influenced by other close opinions. 
	In this context, the expected emergent phenomenon is the asymptotic formation of one or several clusters of individuals with very similar opinions. In recent years, the study of this model has once again attracted attention in the scientific community, see for example \cite{DMPW,JM,L,WCB} and the references therein. The C-S model \eqref{A-0} can be considered as an extension of the Krause model in which consensus (here named flocking) means that the spatial distances between individuals remain bounded globally in time and their velocities converge asymptotically towards a collective one. An extension of the emergence of flocking behavior to the case of nonsymmetric communication weights has been proposed by Motsch and Tadmor \cite{M-T}.

	However, the C-S model and its variants only describe the dynamics of mechanical variables, such as the position and velocity of agents, and the internal variables of agents are disregarded. Including internal or behavioral variables in the swarm dynamics, that can control or measure the degree of activation of individuals or clusters of individuals towards consensus, is of great importance. 
		In this sense, variables such as excitation levels, temperature, emotional affinity, spin, phase, or biochemical signaling in cell motility, among others, may have a role in this context. Among many possible ways to include the effect of the internal variable, the C-S model was recently generalized by Ha and Ruggeri \cite{HR17} so that it incorporates an additional internal variable, usually referred to as \textit{temperature}. They used general conservation laws for a mixture of Euler systems, and using the entropy principle and Galilean invariance to find the production terms, they derived the following thermomechanical Cucker-Smale (TCS) model:
	\begin{align}
	\begin{aligned}\label{A-1}
	&\frac{dx_i}{dt}=v_i, \quad i=1,2,\cdots, N,\quad t>0,\\
	&\frac{dv_i}{dt}=\frac{\kappa}{N}\sum_{j=1}^N\phi(x_i-x_j)\left(\frac{v_j}{\theta_j}-\frac{v_i}{\theta_i}\right),\\
	&\frac{d\theta_i}{dt}=\frac{\nu}{N}\sum_{j=1}^N\zeta(x_i-x_j)\left(\frac{1}{\theta_i}-\frac{1}{\theta_j}\right).
	\end{aligned}
	\end{align}
	
		Note that, when all the internal variables are identical, i.e., $\theta_i\equiv \theta_0$, then the TCS model \eqref{A-1} reduces to the C-S model \eqref{A-0}. Then, it is straightforward to derive the mesoscopic equation for the TCS model, analogous to \eqref{A-0-1} for the C-S model. Precisely, it is given by following Vlasov-type equation:
	
	\begin{align}
	\begin{aligned}\label{eq-kinetic-TCS}
	&\partial_t f+v\cdot \nabla_x f +\nabla_v\cdot (F[f]f)+\partial_\theta(G[f]f) = 0,\quad (t,x,v,\theta)\in\mathbb{R}_+\times \bbr^{2d}\times\bbr_+,\\
	&F[f](t,x,v,\theta):=\kappa\int_{\bbr^{2d}\times\bbr_+}\phi(x-x_*)\left(\frac{v_*}{\theta_*}-\frac{v}{\theta}\right)f(t,x_*,v_*,\theta_*)\,dx_*\,dv_*\,d\theta_*,\\
	&G[f](t,x,v,\theta):=\nu\int_{\bbr^{2d}\times\bbr_+}\zeta(x-x_*)\left(\frac{1}{\theta}-\frac{1}{\theta_*}\right)f(t,x_*,v_*,\theta_*)\,dx_*\,dv_*\,d\theta_*.
	\end{aligned}
	\end{align}
	
	The kinetic description of the TCS ensemble \eqref{eq-kinetic-TCS} has been analyzed recently in the literature. To name a few results: uniform-in-time mean-field limit \cite{HKMRZ19}, well-posedness and asymptotic behavior of the particle-fluid coupled systems \cite{CHJK19,CHJK20,KZ20}, interaction with chemotactic movement \cite{HKZ20} or propagation of mono-kinetic solutions \cite{KK20} have been studied. We also refer \cite[Section 4]{A19} for an overview of the TCS model.
	
	When the TCS ensemble interacts with the surrounding environment  or is affected by further self-interaction forces, there will be additional effects on the dynamics \eqref{eq-kinetic-TCS}. In this paper, we consider the case when this ensemble $f$ is affected by a given background fluid species ($\bar{\rho},\bar{u},\bar{e}$), which also obeys the hydrodynamic TCS equations \cite{HKMRZ18} obtained through the mono-kinetic closure of \eqref{eq-kinetic-TCS}, and self-interaction attractive/repulsive aggregation forces. Precisely, we consider the following coupled meso-macro multiscale system: 
	
	\begin{align}
	\begin{aligned}\label{eq-kinetic-TCS-couple}
	\partial_t f&+v\cdot \nabla_x f+\nabla_v\cdot (F[f]\,f+H[f]\,f+F_c[\bar \rho,\bar u,\bar e]f)+\partial_\theta(G[f]\,f+G_c[\bar \rho,\bar e]\,f)=0,\\
	&H[f](t,x,v,\theta):= -\int_{\bbr^{2d}\times\bbr_+}\nabla W(x-x_*)f(t,x_*,v_*,\theta_*)\,dx_*\,dv_*\,d\theta_*,\\
	&F_c[\bar\rho,\bar u,\bar e](t,x,v,\theta):=\int_{\mathbb{R}^d}\left(\frac{\bar u(t,x_*)}{\bar e(t,x_*)}-\frac{v}{\theta}\right)\bar \rho(t,x_*)\,dx_*,\\
		&G_c[\bar \rho,\bar e](t,x,v,\theta):=\int_{\mathbb{R}^d}\left(\frac{1}{\theta}-\frac{1}{\bar e(t,x_*)}\right)\bar\rho(t,x_*)\,dx_*.\\
		\end{aligned}
	\end{align}
	\begin{align}\label{eq-hydro-TCS-couple}
	\begin{aligned}
	\partial_t \bar \rho +\nabla\cdot (\bar \rho\, \bar u)&=0,\\
	\partial_t(\bar \rho\,\bar u)+\nabla\cdot (\bar \rho\,\bar u\otimes \bar u)&=\int_{\mathbb{R}^d}\phi(x-x_*)\left(\frac{\bar u(t,x_*)}{\bar  e(t,x_*)}-\frac{\bar u(t,x)}{\bar e(t,x)}\right)\bar \rho(t,x)\,\bar \rho(t,x_*)\,dx_*,\\
	\partial_t(\bar \rho\,\bar e)+\nabla\cdot (\bar \rho\,\bar e\,\bar u)&=\int_{\mathbb{R}^d}\zeta(x-x_*)\left(\frac{1}{\bar e(t,x)}-\frac{1}{\bar e(t,x_*)}\right)\bar \rho(t,x)\,\bar \rho(t,x_*)\,dx_*,
	\end{aligned}
	\end{align}
where $W:\mathbb{R}^d\longrightarrow \mathbb{R}$ is an aggregation potential and coupling strengths $\kappa$ and $\nu$ have been normalized to $1$ for simplicity. We note that in the coupled multiscale system \eqref{eq-kinetic-TCS-couple}-\eqref{eq-hydro-TCS-couple}, the dynamics of the ensemble $f$ is affected by the velocity and internal variable alignment self-interactions $F[f]$, $G[f]$, the aggregation self-interactions $H[f]$ and the cross-interactions $F_c[\bar{\rho},\bar{u},\bar{e}]$, $G_c[\bar{\rho},\bar{e}]$ with the background fluid species; while the background fluid $(\bar{\rho},\bar{u},\bar{e})$ only interacts with itself, and is not affected by the kinetic ensemble as a consequence of our sparseness assumption on the kinetic species. We also note that the governing equation \eqref{eq-hydro-TCS-couple} is nothing but the hydrodynamic description of the TCS model, which was introduced and studied in \cite{HKMRZ18}. We refer to Section \ref{sec:derivation} for the detailed procedure to derive the meso-macro multiscale system \eqref{eq-kinetic-TCS-couple}-\eqref{eq-hydro-TCS-couple} from the micro-micro system of interacting particles. In particular, we justify the specific Stokes-type drag forces $F_c$ and $G_c$ for the fluid-particle interaction.  Also, we mention \cite{CD11,CD14,FST16,LLE08} for similar models with coupled alignment and aggregation effects.

	The main interest of this paper is to derive some hydrodynamic approximation of the kinetic equation \eqref{eq-kinetic-TCS-couple} through an appropriate scaling limit. We consider two types of limiting processes: strong and weak relaxation of the internal variable. More precisely, we first consider the following scaling for \eqref{eq-kinetic-TCS-couple}:
	
	\begin{align}
	\begin{aligned}\label{A-4}
	\partial_t f_\e+v\cdot \nabla_x f_\e&+\frac{1}{\e}\nabla_v\cdot (F[f_\e]\,f_\e+H[f_\e]f_\e+F_c[\bar \rho,\bar u,\bar e]f_\e)\\
	&+\frac{1}{\e}\partial_\theta(G[f_\e]\,f_\e+G_c[\bar \rho,\bar e]\,f_\e)=0.
	\end{aligned}
	\end{align}
	We refer to Section \ref{SS-dimensional-analysis} for the derivation of the scaled equation \eqref{A-4} under appropriate assumptions on the scales of the parameters. In this scaling, the system undergoes both strong alignment $\frac{1}{\varepsilon}G[f_\e]$ and strong relaxation $\frac{1}{\varepsilon}G_c[\bar{\rho},\bar{e}]$ of the internal variable. Thus, the internal variable is expected to converge very fast toward a common value $\theta^\infty$, which is given by an averaged internal variable for the background fluid:
	\begin{align}
	\theta^\infty(t)&:=\left(\int_{\mathbb{R}^d}\frac{\bar\rho(t,x)}{\bar e(t,x)}\,dx\right)^{-1}.\label{E-background-mean-theta}
	\end{align}
	Then, our goal is to show that the system \eqref{A-4} asymptotically converges to the following macroscopic equation as $\e\to0$:
	\begin{align}
	\begin{aligned}\label{A-5}
	&\partial_t \rho +\nabla\cdot(\rho u)=0, & (t,x)\in \mathbb{R}_+\times \mathbb{R}^d,\\
	&u-u^\infty(t)+\theta^\infty(t)\nabla W*\rho=\phi*(\rho u)-(\phi*\rho), & (t,x)\in \mathbb{R}_+\times \mathbb{R}^d,
	\end{aligned}
	\end{align}
	where $u^\infty$ is defined as an averaged velocity for the background fluid:
	\begin{equation}
	u^\infty(t):=\theta^\infty(t)\int_{\mathbb{R}^d}\frac{\bar \rho(t,x)\bar u(t,x)}{\bar e(t,x)}\,dx.\label{E-background-mean-velocity}
	\end{equation}
	This is reminiscent of the vanishing inertia limiting system obtained in \cite{PS17} (see also \cite{FST16}) as hyperbolic hydrodynamic limit of the Cucker-Smale model. Indeed, the forcing term is now determined by the background values $\theta^\infty(t)$ and $u^\infty(t)$ of internal variable and velocity in the background fluid, and the aggregation force $\nabla W*\rho$. 
	
	On the other hand, the other scaling that we consider is the case when strong alignment is assumed for the internal variable, but the relaxation of the internal variable is not strong. Thus, in this regime, we consider the following alternative scaled equation:
	\begin{align}
	\begin{aligned}\label{A-6}
	\partial_t f_\e+v\cdot \nabla_x f_\e&+\frac{1}{\e}\nabla_v\cdot (F[f_\e]\,f_\e+H[f_\e]f_\e+F_c[\bar \rho,\bar u,\bar e]f_\e)\\
	&+\partial_\theta\left(\frac{1}{\e}G[f_\e]\,f_\e+G_c[\bar \rho,\bar e]\,f_\e\right)=0.
	\end{aligned}
	\end{align}
	Again, we refer Section to \ref{SS-dimensional-analysis} for deriving the scaled equation \eqref{A-6}. In this case, we expect that the internal variable will become homogeneous in space due to the strong alignment term $\frac{1}{\e}G[f_\e]f_\e$, but it slowly relaxes toward the average background internal variable $\theta^\infty(t)$. The precise asymptotic macroscopic equation as $\varepsilon\rightarrow 0$ will be 
	\begin{align}
	\begin{aligned}\label{A-7}
	&\partial_t \rho +\nabla\cdot(\rho u)=0, & (t,x)\in \mathbb{R}_+\times \mathbb{R}^d,\\
	&u-\frac{\theta(t)}{\theta^\infty(t)}u^\infty(t)+\theta(t)\nabla W*\rho=\phi*(\rho u)-(\phi*\rho), & (t,x)\in \mathbb{R}_+\times \mathbb{R}^d,
	\end{aligned}
	\end{align}
	where $\theta(t)$ is now given by the following relaxation ODE converging to $\theta^\infty(t)$:
	\[\dot{\theta}(t)=\frac{1}{\theta(t)}-\frac{1}{\theta^\infty(t)}.\]
	
	In this paper, we will work initially with bounded influence functions $\phi,\zeta$ for alignment interactions and bounded potential forces $-\nabla W$ for aggregation interactions. Specifically, we assume that:
	\begin{align}\label{H-hypothesis-phi-zeta-W}
	\begin{aligned}
	&\phi,\zeta\mbox{ are radially symmetric, non-increasing, }\phi(0)=\zeta(0)=1,\\
	&\phi,\zeta\in W^{1,\infty}(\mathbb{R}^d)\quad \mbox{and}\quad \nabla W\in W^{1,\infty}(\mathbb{R}^d,\mathbb{R}^d).
	\end{aligned}
	\end{align}
Under those regularity conditions \eqref{H-hypothesis-phi-zeta-W}, we can apply the results in \cite{CCR11,KMT13} to derive well-posedness of weak solutions to \eqref{eq-kinetic-TCS-couple}. 
Also, the fluid species \eqref{eq-hydro-TCS-couple} is expected to be well posed as we explain in Remark \ref{eq-hydro-TCS-couple}. However, we note that it is also possible to consider the case of non-local alignment interactions $F[f]$ and $G[f]$ whose influence functions $\phi$ and $\zeta$ are replaced by scaled versions $\phi_\e$ and $\zeta_\e$ that become weakly singular at the origin as $\e\to0$, in the spirit of \cite{PS17}. Similarly, we can also consider non-local aggregation interactions $H[f]$ whose potential $W$ is replaced by scaled version $W_\e$ that becomes weakly singular at the origin as $\e\to0$. We propose these weakly singular regimes in Section \ref{SSS-scaling-singular} and we address some of the corresponding hydrodynamic limits in Sections \ref{subsec:3.5} and \ref{subsec:4.5}.
	
	The rest of the paper is organized as follows. In Section 2, we present a formal derivation of the meso-macro multiscale system \eqref{eq-kinetic-TCS-couple}-\eqref{eq-hydro-TCS-couple} and appropriate dimensional analysis to derive the scaled systems \eqref{A-4} and \eqref{A-6}. Section 3 provides the rigorous hydrodynamic limit in the strong relaxation regime. More precisely we will derive the limiting macroscopic equation \eqref{A-5} from \eqref{A-4} after obtaining the estimate for the support of internal variable and velocity moments. In Section 4, we prove a similar result on the hydrodynamic limit for the weak relaxation regime for \eqref{A-6} toward \eqref{A-7}. Section 5 present the numerical simulations for the limiting macroscopic systems \eqref{A-5} and \eqref{A-7}. Finally, in Appendix \ref{Appendix-LB} we recall some necessary duality representation of weak Lebesgue--Bochner spaces that will be useful throughout the paper.\\
	
	\noindent {\bf Notation}. Throughout the paper, we use the following conventions: 
	
	\noindent $\bullet$ {\it Variables}: We use abbreviated variables $z=(x,v,\theta)\in \mathbb{R}^d\times \mathbb{R}^d\times \mathbb{R}_+$ and $dz=dx\,dv\,d\theta$.
	
	\noindent $\bullet$ {\it Function spaces}: For any open subset $\Omega\subset \bbr^{k}$, $\mathcal{M}(\Omega)$ denotes the Banach space of finite Radon measures on $\Omega$ endowed with the total variation norm. $\mathcal{P}(\Omega)$ is the metric space consisting of probability measures endowed with narrow topology. $C_c(\Omega)$, $C_0(\Omega)$ and $C_b(\Omega)$ respectively represent the Banach spaces of continuous functions with compact support, continuous functions vanishing at $\partial \Omega$ and bounded continuous functions, endowed with the uniform norm. Finally, for any $p\in [1,\infty]$, $L^p(\Omega)$ and $W^{1,p}(\Omega)$ denote the usual Lebesgue and Sobolev spaces, whilst $L^p(0,T;X)$ stands for the Lebesgue--Bochner space associated with any $T>0$ and any Banach space $X$.
	
	\noindent $\bullet$ {\it Partial supports}: Given any probability measure $f\in \mathcal{P}(\mathbb{R}^{2d}\times \mathbb{R}_+)$, we define the partial support with respect to $x$, $v$ and $\theta$ respectively as the projections $\supp_x f:=\pi_x(\supp f)\subset \mathbb{R}^d$, $\supp_v f:=\pi_v(\supp f)\subset \mathbb{R}^d$ and $\supp_{\theta}f:=\pi_\theta(\supp f)\subset \mathbb{R}_+$, where $\pi_x$, $\pi_v$ and $\pi_\theta$ stand for the standard projections of $z$ into the corresponding variables $x$, $v$ and $\theta$.

	\section{Formal derivation of the meso-macro system and scaling}\label{sec:derivation}
	\setcounter{equation}{0}

	Our goal in this section is to introduce a formal derivation of the multiscale meso-macro system \eqref{eq-kinetic-TCS-couple}-\eqref{eq-hydro-TCS-couple} of interest in this paper. That will be a consequence of appropriate chained scaling limits (mean-field and hydrodynamic limits) on a two-species coupled system as depicted in Figure \ref{fig:multiscale-systems}. More specifically, we will pay special attention to the choice of cross-interactions between both species in a compatible way with TCS dynamics. For simplicity, we shall justify the simpler case in the absence of noise, 
	although obvious modifications could be applied for a more general system under the effect of a thermal bath. 
	
	\begin{figure}
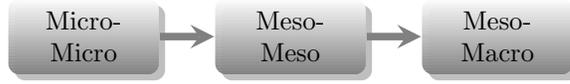

		\centering
		\smartdiagramset{border color=none,set color list={gray,gray,gray,gray},back arrow disabled=true}
		\smartdiagram[flow diagram:horizontal]{Micro-Micro,Meso-Meso,Meso-Macro}
		\caption{Chain of multiscale systems}
		\label{fig:multiscale-systems}
	\end{figure}
	
	\subsection{A micro-micro system with self and cross interactions}
	Our starting point is a micro-micro model for the interactions of agents in two distinguished microscopic species. More specifically, we will consider $N=N_1+N_2$ agents distributed into two different species with $N_1$ and $N_2$ agents for each species. We advance here that the main scaling assumption is that the first species has considerably fewer agents than the second one, i.e.,
	$$N_1\ll N_2.$$
	Then, the dynamics we will propose is based on the heuristic idea that the small species should essentially not affect the large one. Conversely, the small species should be both affected by agents in the small species itself and the agents in the large species. To describe our micro-micro system, we denote positions, velocities, and internal variables of each species by
	\begin{align*}
	z_i&=(x_i,v_i,\theta_i)\in\bbr^d\times\bbr^d\times\bbr_+, & i=1,\ldots,N_1,\\
	\bar z_k&=(\bar x_k,\bar v_k,\bar \theta_k)\in\bbr^d\times\bbr^d\times\bbr_+, & k=1,\ldots,N_2.
	\end{align*}
	We will assume that the whole group of $N_1+N_2$ agents interact weakly all-to-all through TCS-type interactions. In addition, we suppose that agents in the first species also interact weakly all-to-all via aggregation attractive/repulsive interactions. Specifically, by dividing into self-interactions for agents in the same species and cross-interactions for agents in different species, we can propose the following system of coupled ODEs:
	
	\begin{align}\label{E-micro-micro-1st}
	\begin{aligned}
	\frac{dx_i}{dt}&=v_i,\\
	m_1\frac{dv_i}{dt}&=w_1\frac{\kappa_1}{N}\sum_{j=1}^{N_1}\phi_1(x_i-x_j)\left(\frac{v_j}{\theta_j}-\frac{v_i}{\theta_i}\right)-w_1\frac{\kappa_a}{N_1}\sum_{j=1}^{N_1}\nabla W_1(x_i-x_j)\\
	&\quad +\frac{\kappa_c}{N}\sum_{l=1}^{N_2}\phi_c(x_i-\bar x_l)\left(\frac{\bar v_l}{\bar \theta_l}-\frac{v_i}{\theta_i}\right),\\
	\frac{d\theta_i}{dt}&=w_1\frac{\nu_1}{N}\sum_{j=1}^{N_1}\zeta_1(x_i-x_j)\left(\frac{1}{\theta_i}-\frac{1}{\theta_j}\right)+\frac{\nu_c}{N}\sum_{l=1}^{N_2}\zeta_c(x_i-\bar x_l)\left(\frac{1}{\theta_i}-\frac{1}{\bar \theta_l}\right),
	\end{aligned}
	\end{align}
	
	\begin{align}\label{E-micro-micro-2nd}
	\begin{aligned}
	\frac{d\bar x_k}{dt}&=\bar v_k,\\
	m_2\frac{d\bar v_k}{dt}&=w_2\frac{\kappa_2}{N}\sum_{l=1}^{N_2}\phi_2(\bar x_k-\bar x_l)\left(\frac{\bar v_l}{\bar \theta_l}-\frac{\bar v_k}{\bar \theta_k}\right)+\frac{\kappa_c}{N}\sum_{j=1}^{N_1}\phi_c(\bar x_k-x_j)\left(\frac{v_j}{\theta_j}-\frac{\bar v_k}{\bar \theta_k}\right),\\
	\frac{d\bar \theta_k}{dt}&=w_2 \frac{\nu_2}{N}\sum_{l=1}^{N_2}\zeta_m(\bar x_k-\bar x_l)\left(\frac{1}{\bar \theta_k}-\frac{1}{\bar \theta_l}\right)+\frac{\nu_c}{N}\sum_{j=1}^{N_1}\zeta_c(\bar x_k-x_j)\left(\frac{1}{\bar \theta_k}-\frac{1}{\theta_j}\right),
	\end{aligned}
	\end{align}
	for every $i=1,\ldots,N_1$ and $k=1,\ldots,N_2$. Here, $\kappa_1,\nu_1\in \mathbb{R}_+$ are coupling strength coefficients for the  self-interactions between agents in the first species. We remark that those interactions are of TCS-type and are regulated by influence functions $\phi_1,\zeta_1$. Similarly, for the second microscopic species, $\kappa_2,\nu_2\in \mathbb{R}_+$ are the coupling strength for their  self-interactions, whilst $\phi_2,\zeta_2$ stand for the associated influence functions. We couple both systems via TCS-type cross-interactions between agents in two different species, that are modulated by the coupling strength coefficients $\kappa_c,\nu_c\in \mathbb{R}_+$ and the influence functions $\phi_c$ and $\zeta_c$. Finally, $\kappa_a$ indicates the coupling strength for the aggregation self-interactions between agents in the first species. We remark that both self-interactions within the same species and cross-interactions between different ones are symmetric and compatible with Newton's laws of motion. The weights $w_1=w_1(N_1,N_2)$ and $w_2=w_2(N_1,N_2)$ in \eqref{E-micro-micro-1st}--\eqref{E-micro-micro-2nd} are dimensionless coefficients depending only on the amount of agents $N_1$ and $N_2$. They allow modifying the weak mean-field all-to-all interactions by an appropriate scaling so that stronger self-interactions might be considered. Specifically, we will assume here that weights are given by
	\begin{equation}\label{E-weights}
	w_1=\frac{N}{N_1}\quad \mbox{ and }\quad w_2=\frac{N}{N_2}.
	\end{equation}

	\begin{remark}[A weighted mean-field scaling]
		Notice that the usual mean-field scaling has been avoided in \eqref{E-micro-micro-1st}-\eqref{E-micro-micro-2nd} by introducing the scaling weights \eqref{E-weights}. The reason behind this correction relies on a typical constraint of the mean-field scale with regards to groups with considerably different sizes. Indeed, let us assume for a moment that $w_1=w_2=1$ so that the interactions are reduced to weak mean-field all-to-all interactions. In such a case, equations $\eqref{E-micro-micro-1st}_2$-$\eqref{E-micro-micro-2nd}_2$ would take the form
		\begin{align*}
		m_1\frac{dv_i}{dt}&=\frac{N_1}{N}\frac{\kappa_1}{N_1}\sum_{j=1}^{N_1}\phi_1(x_i-x_j)\left(\frac{v_j}{\theta_j}-\frac{v_i}{\theta_i}\right)-\frac{N_1}{N}\frac{\kappa_a}{N_1}\sum_{j=1}^{N_1}\nabla W_1(x_i-x_j)\\
		&\quad+\frac{N_2}{N}\frac{\kappa_c}{N_2}\sum_{l=1}^{N_2}\phi_c(x_i-\bar x_l)\left(\frac{\bar v_l}{\bar \theta_l}-\frac{v_i}{\theta_i}\right),\\
		m_2\frac{d\bar v_k}{dt}&=\frac{N_2}{N}\frac{\kappa_2}{N_2}\sum_{l=1}^{N_2}\phi_2(\bar x_k-\bar x_l)\left(\frac{\bar v_l}{\bar \theta_l}-\frac{\bar v_k}{\bar \theta_k}\right)+\frac{N_1}{N}\frac{\kappa_c}{N_1}\sum_{j=1}^{N_1}\phi_c(\bar x_k-x_j)\left(\frac{v_j}{\theta_j}-\frac{\bar v_k}{\bar \theta_k}\right),
		\end{align*}
		for $i=1,\ldots,N_1$ and $k=1,\ldots,N_2$. Since we are assuming $N_1\ll N_2$, we obtain that
		$$\frac{N_1}{N}\approx 0\quad \mbox{ and }\quad \frac{N_2}{N}\approx 1.$$
		Thus, the small species essentially does not affect the large one. However, we also observe that the presence of a large species almost halts the dynamics of the small one. Namely, TCS- and aggregation-type self-interactions within the small species become negligible. This violates the basic principle of locality of interactions, see \cite{S91}. Our correction in \eqref{E-micro-micro-1st}-\eqref{E-micro-micro-2nd} keeps mean-field cross-interactions, but self-interactions are scaled according to the number of particles in the corresponding species. We shall see that the above unrealistic behavior disappears and agents in the small species will be affected both by self and cross interactions in the same order.
	\end{remark}

	\subsection{From micro-micro to meso-meso}
	
	Here, we sketch the formal coupled mean-field limit as $N_1\rightarrow\infty$ and $N_2\rightarrow \infty$ under the constraint $N_1\ll N_2$. This will justify the meso-meso system. Notice that since $N_1\ll N_2$, we can set $N_2=\sigma(N_1)$ for an appropriate function $\sigma:\mathbb{N}\longrightarrow \mathbb{N}$ such that 
	\begin{equation}\label{E-relation-N1-N2}
	\lim_{N_1\rightarrow \infty}\frac{\sigma(N_1)}{N_1}=\infty.
	\end{equation}
	We will see that our limiting meso-meso system will not depend on the specific function $\sigma$ linking $N_2$ to $N_1$. Associated with each species, we can define the empirical measures
	\begin{align*}
	\mu^{N_1}(t,x,v,\theta)&:=\frac{1}{N_1}\sum_{i=1}^{N_1}\delta_{x_i(t)}(x)\otimes \delta_{v_i(t)}(v)\otimes \delta_{\theta_i(t)}(\theta),\\
	\bar\mu^{N_1}(t,x,v,\theta)&:=\frac{1}{\sigma(N_1)}\sum_{k=1}^{\sigma(N_1)} \delta_{\bar x_k(t)}(x)\otimes \delta_{\bar v_k(t)}(v)\otimes \delta_{\bar\theta_k(t)}(\theta),
	\end{align*}
	where we have omitted the dependence on $\sigma$ for simplicity. Using \eqref{E-micro-micro-1st}-\eqref{E-micro-micro-2nd}, standard arguments show that $ \mu^{N_1}$ and $\bar \mu^{N_1}$ satisfy the following Vlasov-type kinetic equations in distributional sense
	\begin{align}\label{E-meso-empirical-measures}
	\begin{aligned}
	\partial_t\mu^{N_1}+v\cdot \nabla_x\mu^{N_1}&+\frac{1}{m_1}\nabla_v\cdot \left(F_1[\mu^{N_1}]\,\mu^{N_1}+H_1[\mu^{N_1}]\mu^{N_1}+\frac{\sigma(N_1)}{N_1+\sigma(N_1)}F_c[\bar\mu^{N_1}]\,\mu^{N_1}\right)\\
	&+\partial_\theta\left(G_1[\mu^{N_1}]\,\mu^{N_1}+\frac{\sigma(N_1)}{N_1+\sigma(N_1)}G_c[\bar \mu^{N_1}]\,\mu^{N_1}\right)=0,\\
	\partial_t\bar \mu^{N_1}+v\cdot \nabla_x \bar \mu^{N_1}&+\frac{1}{m_2}\nabla_v\cdot \left(F_2[\bar \mu^{N_1}]\,\bar \mu^{N_1}+\frac{N_1}{N_1+\sigma(N_1)}F_c[\mu^{N_1}]\,\bar{\mu}^{N_1}\right)\\
	&+\partial_\theta\left(G_2[\bar \mu^{N_1}]\,\bar \mu^{N_1}+\frac{N_1}{N_1+\sigma(N_1)}G_c[\mu^{N_1}]\,\bar \mu^{N_1}\right)=0,
	\end{aligned}
	\end{align}
	where the operators $F_i$, $G_i$ and $H_1$ take the form
	\begin{align*}
	F_i[f](t,x,v,\theta)&:=\kappa_i\int_{\mathbb{R}^{2d}\times\bbr_+}\phi_i(x-x_*)\left(\frac{v_*}{\theta_*}-\frac{v}{\theta}\right)f(t,z_*)\,dz_*,\\
	G_i[f](t,x,v,\theta)&:=\nu_i\int_{\mathbb{R}^{2d}\times\bbr_+}\zeta_i(x-x_*)\left(\frac{1}{\theta}-\frac{1}{\theta_*}\right)f(t,z_*)\,dz_*,\\
	H_1[f](t,x,v,\theta)&:=-\kappa_a\int_{\mathbb{R}^{2d}\times \mathbb{R}_+}\nabla W_1(x-x_*)f(t,z_*)\,dz_*,
	\end{align*}
	for any curve of probability measures $t\in \mathbb{R}_+\mapsto f(t,\cdot)\in \mathcal{P}(\mathbb{R}^d\times \mathbb{R}^d\times \mathbb{R}_+)$ and each index $i=1,2,c$. The mean-field limit as $N_1\rightarrow \infty$ allows showing that, if $\mu^{N_1}(0,\cdot)\rightarrow f(0,\cdot)$ and $\bar\mu^{N_1}(0,\cdot)\rightarrow \bar f(0,\cdot)$ appropriately in the sense of probability measures, then
	$$\bar \mu^{N_1}(t,\cdot)\rightarrow \bar f(t,\cdot)\quad \mbox{ and }\quad \mu^{N_1}(t,\cdot)\rightarrow f(t,\cdot),$$
	for each $t\in \mathbb{R}_+$. Under such condition, it is clear that we can formally pass to the limit as $N_1\rightarrow\infty$ in the above equation \eqref{E-meso-empirical-measures} and find a closed equation in terms of the limiting probability distributions $f=f(t,x,v,\theta)$ and $\bar f=\bar f(t,x,v,\theta)$. Indeed, notice that the size relation \eqref{E-relation-N1-N2} implies that 
	$$\lim_{N_1\rightarrow\infty}\frac{\sigma(N_1)}{N_1+\sigma(N_1)}=1\quad \mbox{ and }\quad \lim_{N_1\rightarrow \infty} \frac{N_1}{N_1+\sigma(N_1)}=0.$$
	Then, whilst self and cross interactions persist in the first mesoscopic species, only self-interactions remain in the second mesoscopic species. To conclude, the joint final mean-field limit of system \eqref{E-micro-micro-1st}-\eqref{E-micro-micro-2nd} as $N_1\rightarrow \infty$ and $N_2\rightarrow\infty$ with $N_1\ll N_2$ takes the form of the following meso-meso system
	\begin{align}\label{E-meso-meso}
	\begin{aligned}
	&\partial_t f+v\cdot \nabla_x f+\frac{1}{m_1}\nabla_v\cdot (F_1[f]\,f+H_1[f]f+F_c[\bar f]f)+\partial_\theta(G_1[f]\,f+G_c[\bar f]\,f)=0,\\
	&\partial_t\bar f+v\cdot \nabla_x\bar f+\frac{1}{m_2}\nabla_v\cdot (F_2[\bar f]\,\bar f)+\partial_\theta(G_2[\bar f]\,\bar f)=0.
	\end{aligned}
	\end{align}

	\subsection{From meso-meso to meso-macro}
	
	Now, our goal is to find a hydrodynamic closure for the second species $\bar f$ in the meso-meso system \eqref{E-meso-meso}. Typically, one looks at moments of the distribution function, e.g.,
	\begin{align*}
	\bar \rho(t,x)&:=\int_{\mathbb{R}^d\times\bbr_+} \bar f(t,x,v,\theta)\,dv\,d\theta,\\
	(\bar \rho\,\bar u)(t,x)&:=\int_{\mathbb{R}^{d}\times\bbr_+}v \bar f(t,x,v,\theta)\,dv\,d\theta,\\
	(\bar \rho\,\bar e)(t,x)&:=\int_{\mathbb{R}^{d}\times\bbr_+}\theta \bar f(t,x,v,\theta)\,dv\,d\theta.
	\end{align*}
	Specifically, multiplying $\eqref{E-meso-meso}_2$ by $(1,v,\theta)$ and integrating with respect to $v$ and $\theta$ yields a non-closed hierarchy of PDEs for moments $\bar \rho$, $\bar u$ and $\bar e$. Although it is not closed, a typical method is to assume the mono-kinetic distribution of $\bar f$, i.e.,
	$$\bar f(t,x,v,\theta)=\bar \rho(t,x)\otimes \delta_{\bar u(t,x)}(v)\otimes \delta_{\bar e(t,x)}(\theta).$$
	By doing so, we obtain the following meso-macro multiscale system 
	\begin{align}\label{E-meso-macro-1st}
	\begin{aligned}
	\partial_t f&+v\cdot \nabla_x f+\frac{1}{m_1}\nabla_v\cdot (F_1[f]\,f+H_1[f]f+F_c[\bar \rho,\bar u,\bar e]f)+\partial_\theta(G_1[f]\,f+G_c[\bar \rho,\bar e]\,f)=0,\\
	&F_c[\bar\rho,\bar u,\bar e](t,x,v,\theta):=\kappa_c\int_{\mathbb{R}^d}\phi_c(x-x_*)\left(\frac{\bar u(t,x_*)}{\bar e(t,x_*)}-\frac{v}{\theta}\right)\bar \rho(t,x_*)\,dx_*,\\
	&G_c[\bar \rho,\bar e](t,x,v,\theta):=\nu_c\int_{\mathbb{R}^d}\zeta_c(x-x_*)\left(\frac{1}{\theta}-\frac{1}{\bar e(t,x_*)}\right)\bar\rho(t,x_*)\,dx_*,
	\end{aligned}
	\end{align}
	\begin{align}\label{E-meso-macro-2nd}
	\begin{aligned}
	\partial_t \bar \rho +\nabla\cdot (\bar \rho\, \bar u)&=0,\\
	\partial_t(\bar \rho\,\bar u)+\nabla\cdot (\bar \rho\,\bar u\otimes \bar u)&=\frac{\kappa_2}{m_2}\int_{\mathbb{R}^d}\phi_2(x-x_*)\left(\frac{\bar u(t,x_*)}{\bar  e(t,x_*)}-\frac{\bar u(t,x)}{\bar e(t,x)}\right)\bar \rho(t,x)\,\bar \rho(t,x_*)\,dx_*,\\
		\partial_t(\bar \rho\,\bar e)+\nabla\cdot (\bar \rho\,\bar e\,\bar u)&=\nu_2\int_{\mathbb{R}^d}\zeta_2(x-x_*)\left(\frac{1}{\bar e(t,x)}-\frac{1}{\bar e(t,x_*)}\right)\bar \rho(t,x)\,\bar \rho(t,x_*)\,dx_*.
	\end{aligned}
	\end{align}

	\begin{remark}[A Stokes drag force compatible with TCS interactions]
		The meso-macro system \eqref{E-meso-macro-1st}-\eqref{E-meso-macro-2nd} is a particular instance of the dynamics of particles interacting with a fluid. This is an old-standing issue that has been treated in the literature for specific problems both from the points of view of fluid mechanics and kinetic theory. For instance, we refer to \cite{BDGM09,GJV04-1,GJV04-2,H98} for the interaction of aerosols with a viscous fluid. In \cite{BCHK12,HKK14} the authors studied the interaction of Cucker-Smale particles with a viscous fluid. Also, the interaction of TCS particles with a viscous fluid has been considered recently in \cite{CHJK19,CHJK20}. In the above setting, particles interact with a viscous ideal fluid so that the cross-interaction terms are described by the so-called Stokes drag force (Stokes' law)
		$$F_d(t,x,v)=\bar u(t,x)-v.$$
		See \cite{BDGR17,BDGR18,B49,DGR08,H18,JB08}, for some results on its mathematical derivation. In our case, the macroscopic species is no longer a Newtonian viscous fluid. Instead, it is replaced by the pressureless and inviscid hydrodynamic TCS model, which describes the active dynamics of a large population of agents through TCS alignment interactions. Also, the particle species is described through a TCS-aggregation kinetic equation weakly coupled with the former one. Given the different nature of the macroscopic equation, Stokes' drag force is not necessarily admissible. In fact, our microscopic derivation of \eqref{E-meso-macro-1st}-\eqref{E-meso-macro-2nd} suggests a different coupling in terms of the nonlinear and nonlocal drag forces $F_c[\bar \rho,\bar u,\bar e]$ and $G_c[\bar \rho,\bar e]$ exerted by the macroscopic species. In particular, if we assume that $f$ is supported at $\theta=\theta_0$, $\bar e\equiv \theta_0$ for some $\theta_0\in \mathbb{R}_+$, and $\phi_c=\zeta_c\equiv 1$, then cross-interactions reduce to an averaged drag force
		$$
		\bar F_d(t,v)=\frac{1}{\theta_0}\left(\int_{\mathbb{R}^d}\bar \rho(t,x)\bar u(t,x)\,dx-v\right).
		$$
		\end{remark}
	
	\subsection{Dimensional analysis and scaling}\label{SS-dimensional-analysis}
	
	Now, we will introduce a dimensionless version of \eqref{E-meso-macro-1st}-\eqref{E-meso-macro-2nd} in terms of characteristic units of the system and we will propose the main scaling assumptions. Here, we do not assume an explicit form of the influence functions $\phi_i$ and $\zeta_i$ and the aggregation potential $W_1$. Instead, we suppose that they are given by proper scaling of general influence functions $\phi$, $\zeta$, and aggregation potential $W$ verifying the regularity conditions \eqref{H-hypothesis-phi-zeta-W}. Specifically, we assume that
	\begin{equation}\label{E-coupling-weights-2}
	\phi_i(x)=\phi\left(\frac{x}{\sigma_i}\right),\quad \zeta_i(x)=\zeta\left(\frac{x}{\sigma_i}\right),\quad W_1(x)=\sigma_a\,W\left(\frac{x}{\sigma_a}\right),\quad x\in \mathbb{R}^d,
	\end{equation}
	for $i=1,2,c$, where $\sigma_1,\sigma_2,\sigma_c\in \mathbb{R}_+$ are length units representing the effective range of TCS (self and cross) interactions. Similarly, $\sigma_a\in \mathbb{R}_+$ represents the effective range of aggregation self-interactions in the kinetic species. Consider $T,L,V,\Theta\in \mathbb{R}_+$ characteristic values for time, length, velocity, and internal variable to be described later. Notice that no further effects (e.g. thermal bath) is applied to velocity apart from the TCS and aggregation interactions. Then, it is natural to assume the relation $V:=\frac{L}{T}$. In addition, consider the following rescaled (dimensionless) variables
	$$\widehat{t}:=\frac{t}{T},\qquad \widehat{x}:=\frac{x}{L},\qquad \widehat{v}:=\frac{v}{V},\qquad \widehat{\theta}:=\frac{\theta}{\Theta}.$$
	For consistency, we also define
	\begin{align*}
	&\widehat{f}(\widehat{t},\widehat{z})=L^dV^d\Theta f(t,z),\quad \widehat{\bar \rho}(\widehat{t},\widehat{x})=L^d\bar \rho(t,x),\quad \\
	&\widehat{\bar u}(\widehat{t},\widehat{x})=\frac{\bar u(t,x)}{V},\quad \widehat{\bar e}(\widehat{t},\widehat{x})=\frac{\bar e(t,x)}{\Theta}.
	\end{align*}
	Inserting them into \eqref{E-meso-macro-1st}-\eqref{E-meso-macro-2nd} and dropping the hats for simplicity of notation, we obtain the following dimensionless system 
	\begin{align}\label{E-meso-macro-1st-dimensionless}
	\begin{aligned}
	&\partial_t f+v\cdot \nabla_x f+\nabla_v\cdot \left(\frac{T\kappa_1}{m_1\Theta} F_{1,\delta_1}[f]\,f+\frac{T\kappa_a}{m_1LV}H_{1,\delta_a}[f]f+\frac{T\kappa_c}{m_1\Theta} F_{c,\delta_c}[\bar \rho,\bar u,\bar e]f\right)\\
	&\hspace{2.5cm}+\partial_\theta\left(\frac{T\nu_1}{\Theta^2} G_{1,\delta_1}[f]\,f+\frac{T\nu_c }{\Theta^2}G_{c,\delta_c}[\bar \rho,\bar e]\,f\right)=0,\\
	&F_{1,\delta_1}[f](t,x,v,\theta):=\int_{\mathbb{R}^{2d}\times\bbr_+}\phi\left(\frac{x-x_*}{\delta_1}\right)\left(\frac{v_*}{\theta_*}-\frac{v}{\theta}\right)f(t,z_*)\,dz_*,\\
	&H_{1,\delta_a}[f](t,x,v,\theta):=-\int_{\mathbb{R}^{2d}\times\bbr_+}\nabla W\left(\frac{x-x_*}{\delta_a}\right)f(t,z_*)\,dz_*,\\
	&G_{1,\delta_1}[f](t,x,v,\theta):=\int_{\mathbb{R}^{2d}\times\bbr_+}\zeta\left(\frac{x-x_*}{\delta_1}\right)\left(\frac{1}{\theta}-\frac{1}{\theta_*}\right)f(t,z_*)\,dz_*,\\
	&F_{c,\delta_c}[\bar\rho,\bar u,\bar e](t,x,v,\theta):=\int_{\mathbb{R}^d}\phi\left(\frac{x-x_*}{\delta_c}\right)\left(\frac{\bar u(t,x_*)}{\bar e(t,x_*)}-\frac{v}{\theta}\right)\bar \rho(t,x_*)\,dx_*,\\
	&G_{c,\delta_c}[\bar \rho,\bar e](t,x,v,\theta):=\int_{\mathbb{R}^d}\zeta\left(\frac{x-x_*}{\delta_c}\right)\left(\frac{1}{\theta}-\frac{1}{\bar e(t,x_*)}\right)\bar\rho(t,x_*)\,dx_*,
	\end{aligned}
	\end{align}
	\begin{align}\label{E-meso-macro-2nd-dimensionless}
	\begin{aligned}
	\partial_t \bar \rho +\nabla_x\cdot (\bar \rho\, \bar u)&=0,\\
	\partial_t(\bar \rho\,\bar u)+\nabla_x\cdot (\bar \rho\,\bar u\otimes \bar u)&=\frac{T\kappa_2}{m_2\Theta}\int_{\mathbb{R}^d}\phi\left(\frac{x-x_*}{\delta_2}\right)\left(\frac{\bar u(t,x_*)}{\bar  e(t,x_*)}-\frac{\bar u(t,x)}{\bar e(t,x)}\right)\bar \rho(t,x)\,\bar \rho(t,x_*)\,dx_*,\\
	\partial_t(\bar \rho\,\bar e)+\nabla_x\cdot (\bar \rho\,\bar e\,\bar u)&=\frac{T\nu_2}{\Theta^2}\int_{\mathbb{R}^d}\zeta\left(\frac{x-x_*}{\delta_2}\right)\left(\frac{1}{\bar e(t,x)}-\frac{1}{\bar e(t,x_*)}\right)\bar \rho(t,x)\,\bar \rho(t,x_*)\,dx_*,
	\end{aligned}
	\end{align}
	where we have denoted the scaled effective range of interactions
	$$\delta_1:=\frac{\sigma_1}{L},\quad \delta_2:=\frac{\sigma_2}{L},\quad \delta_c:=\frac{\sigma_c}{L},\quad \delta_a:=\frac{\sigma_a}{L}.$$
	We now chose the characteristic values as follows:
	\begin{equation}\label{E-characteristic-values}
	L:=\sigma_2,\qquad T:=\frac{m_2}{\kappa_2}=\frac{1}{\nu_2},\qquad \Theta:=1. 
	\end{equation}
	Specifically, $L$ is taken as the effective range of self-interactions of the macroscopic species, $T$ is taken as the relaxation time under such TCS interactions of the second species, $\Theta$ is set to any constant dimensionless value, say $1$. On the one hand, since those are the typical units of the second species, all the parameters in \eqref{E-meso-macro-2nd-dimensionless} disappear. On the other hand, we can define the following rescaled parameters for the first species
	\begin{equation}\label{E-scaled-parameters}
	\mu:=\frac{m_2}{m_1},\quad \tau^v:=\frac{\kappa_2}{\kappa_1},\quad \tau^a:=\frac{\kappa_2LV}{\kappa_a},\quad \tau_c^v:=\frac{\kappa_2}{\kappa_c},\quad \tau^\theta:=\frac{\nu_2}{\nu_1},\quad \tau_c^\theta:=\frac{\nu_2}{\nu_c}.
	\end{equation}
	We substitute the choice \eqref{E-characteristic-values} and notation \eqref{E-scaled-parameters} into \eqref{E-meso-macro-1st-dimensionless}-\eqref{E-meso-macro-2nd-dimensionless} to obtain
	\begin{align}\label{E-meso-macro-1st-dimensionless-2}
	\begin{aligned}
	&\partial_t f+v\cdot \nabla_x f+\mu\nabla_v\cdot \left(\frac{1}{\tau^v} F_{1,\delta_1}[f]\,f+\frac{1}{\tau^a}H_{1,\delta_a}[f]\,f+\frac{1}{\tau_c^v} F_{c,\delta_c}[\bar \rho,\bar u,\bar e]f\right)\\
	&\hspace{2.3cm}+\partial_\theta\left(\frac{1}{\tau^\theta} G_{1,\delta_1}[f]\,f+\frac{1}{\tau_c^\theta} G_{c,\delta_c}[\bar \rho,\bar e]\,f\right)=0,
	\end{aligned}
	\end{align}
	\begin{align}\label{E-meso-macro-2nd-dimensionless-2}
	\begin{aligned}
	\partial_t \bar \rho +\nabla_x\cdot (\bar \rho\, \bar u)&=0,\\
	\partial_t(\bar \rho\,\bar u)+\nabla_x\cdot (\bar \rho\,\bar u\otimes \bar u)&=\int_{\mathbb{R}^d}\phi(x-x_*)\left(\frac{\bar u(t,x_*)}{\bar  e(t,x_*)}-\frac{\bar u(t,x)}{\bar e(t,x)}\right)\bar \rho(t,x)\,\bar \rho(t,x_*)\,dx_*,\\
	\partial_t(\bar \rho\,\bar e)+\nabla_x\cdot (\bar \rho\,\bar e\,\bar u)&=\int_{\mathbb{R}^d}\zeta(x-x_*)\left(\frac{1}{\bar e(t,x)}-\frac{1}{\bar e(t,x_*)}\right)\bar \rho(t,x)\,\bar \rho(t,x_*)\,dx_*.
	\end{aligned}
	\end{align}
	
	Note that the second species is decoupled and can be fixed as an a priori given fluid. We conclude this part by proposing the different scaling limits of the first species that we shall study throughout this paper. For simplicity, we shall make a long-range assumption for cross-interactions between species. Specifically, we will assume that agents in the first species interact with agents in the second species within an effective range that is much larger than the effective range of self-interactions in the second species. In other words, we assume that $\delta_c\rightarrow \infty$. Notice that, as a consequence of \eqref{H-hypothesis-phi-zeta-W} we obtain $\phi(x/\delta_c)\rightarrow \phi(0)=1$ and $\zeta(x/\delta_c)\rightarrow \zeta(0)=1$, so that operators $F_{c,\delta_c}$ and $G_{c,\delta_c}$ get substantially simplified into
	\begin{align*}
	F_c[\bar \rho,\bar u,\bar e](t,x,v,\theta)&:=\int_{\mathbb{R}^d}\frac{\bar \rho(t,x)\bar u(t,x)}{\bar e(t,x)}\,dx-\frac{v}{\theta},\\
	G_c[\bar\rho,\bar e](t,x,v,\theta)&:=\frac{1}{\theta}-\int_{\mathbb{R}^d}\frac{\bar \rho(t,x)}{\bar e(t,x)}\,dx.
	\end{align*}
	Note that the operators $F_c$ and $G_c$ are now independent with the spatial variable $x$. To conclude, let us finally introduce the specific scale relation between the effective range of self-interactions for the first and second species, using the scaling parameter $\varepsilon$. This leads to two distinguished regimes as follows.
	
	\subsubsection{\bf (Regular influence functions)} In this case, we assume that the effective range of TCS and aggregation self-interactions in the first species is comparable to that of the second species, i.e.,
	$$\delta_1=\mathcal{O}(1),\quad \delta_a=\mathcal{O}(1),$$
	as $\varepsilon\rightarrow 0$. Then, we can propose two different scalings, depending on the relation between the strength of relaxation and TCS-aggregation interactions in each species.
	
	\medskip
	
		$\bullet$ {\bf (Strong relaxation of $\theta$)}
		In this case, we assume that
		\begin{equation}\label{E-scaling-regular-fast}
		\mu=\mathcal{O}(1),\quad \tau^v=\tau^a=\tau^v_c=\tau^\theta=\tau_c^\theta=\mathcal{O}(\varepsilon),
		\end{equation}
		as $\varepsilon\rightarrow 0$. This means that both species consist of particles with comparable mass. In addition, TCS, aggregation and relaxation interactions for the first species are stronger than for the second one. In particular, the relaxation of the internal variable of the first species towards the second one is strong. This choice leads to the system
		\begin{align}\label{E-kinetic-regular-fast}
		\begin{aligned}
		&\partial_t f_\e+v\cdot \nabla_x f_\e+\frac{1}{\varepsilon}\nabla_v\cdot \left(F[f_\e]\,f_\e+H[f_\e]\,f_\e+F_c[\bar \rho,\bar u,\bar e]f_\e\right)\\
		&\hspace{2.5cm}+\frac{1}{\varepsilon}\partial_\theta\left(G[f_\e]\,f_\e+G_c[\bar \rho,\bar e]\,f_\e\right)=0,\\
		&F[f](t,x,v,\theta):=\int_{\mathbb{R}^{2d}\times\bbr_+}\phi(x-x_*)\left(\frac{v_*}{\theta_*}-\frac{v}{\theta}\right)f(t,z_*)\,dz_*,\\
		&H[f](t,x,v,\theta):=-\int_{\mathbb{R}^{2d}\times\bbr_+}\nabla W(x-x_*)f(t,z_*)\,dz_*,\\
		&G[f](t,x,v,\theta):=\int_{\mathbb{R}^{2d}\times\bbr_+}\zeta(x-x_*)\left(\frac{1}{\theta}-\frac{1}{\theta_*}\right)f(t,z_*)\,dz_*,\\
		&F_c[\bar \rho,\bar u,\bar e](t,x,v,\theta):=\int_{\mathbb{R}^d}\frac{\bar \rho(t,x)\bar u(t,x)}{\bar e(t,x)}\,dx-\frac{v}{\theta},\\
		&G_c[\bar\rho,\bar e](t,x,v,\theta):=\frac{1}{\theta}-\int_{\mathbb{R}^d}\frac{\bar \rho(t,x)}{\bar e(t,x)}\,dx.
		\end{aligned}
		\end{align}
		
		\medskip
		
	    $\bullet$ {\bf (Weak relaxation of $\theta$)}
		In this case, we assume that
		\begin{equation}\label{E-scaling-regular-slow}
		\mu=\mathcal{O}(1),\quad \tau^v=\tau^a=\tau^v_c=\tau^\theta=\mathcal{O}(\varepsilon),\quad \tau_c^\theta=\mathcal{O}(1),
		\end{equation}
		as $\varepsilon\rightarrow 0$. This has similar implications as \eqref{E-scaling-regular-fast} except for the fact that relaxation of the internal variable of the first species towards the second one is slow and it occurs at the same scale as alignment interactions in the second species. This choice leads to the system
		\begin{align}\label{E-kinetic-regular-slow}
		\begin{aligned}
		&\partial_t f_\e+v\cdot \nabla_x f_\e+\frac{1}{\varepsilon}\nabla_v\cdot \left(F[f_\e]\,f_\e+H[f_\e]\,f_\e+F_c[\bar \rho,\bar u,\bar e]f_\e\right)\\
		&\hspace{2.5cm}+\partial_\theta\left(\frac{1}{\varepsilon} G[f_\e]\,f_\e+G_c[\bar \rho,\bar e]\,f_\e\right)=0,\\
		\end{aligned}
		\end{align}

	\begin{remark}
	Note that the parameters $\kappa_i$, $\nu_i$ and $\kappa_a$ for coupling strengths are absorbed by the scaling parameter $\e$ in the dimensionless versions of the meso-macro multiscale systems \eqref{E-kinetic-regular-fast} and \eqref{E-kinetic-regular-slow}. Thus, from now on, we focus on the case when these coupling strengths are normalized to 1, i.e., $\kappa_i=\nu_i=\kappa_a=1$ as in \eqref{A-4} and \eqref{A-6}.
    \end{remark}
	
	\subsubsection{\bf (Singular influence functions)}\label{SSS-scaling-singular}
	In this case, we assume that the effective range of TCS self-interactions in the first species is much smaller than in the second species, but, for simplicity, aggregation self-interactions are of the same order, i.e.,
	$$\delta_1=\mathcal{O}(\varepsilon),\quad \delta_a=\mathcal{O}(1),$$
	as $\varepsilon\rightarrow 0$. Again, we propose two different scalings depending on the relation between the strength of relaxation and alignment interactions in each species. Here, we consider the following choices of $\phi$ and $\zeta$ inspired by the typical influence functions of the C-S model 
	\begin{equation}\label{influence-function}
	\phi(x)=\frac{1}{(1+c_{\lambda_1} \vert x\vert^2)^{\lambda_1/2}},\quad \zeta (x)=\frac{1}{(1+c_{\lambda_2} \vert x\vert^2)^{\lambda_2/2}},\quad x\in \mathbb{R}^d,
	\end{equation}
	for general parameters $\lambda_1,\,\lambda_2>0$. The coefficients $c_{\lambda_i}=2^{2/{\lambda_i}}-1$ have been chosen so that the tails of influence functions $\phi_1$ and $\zeta_1$ in \eqref{E-coupling-weights-2} are small. Namely, we obtain that
	$$\phi_1(x)=\phi\left(\frac{x}{\sigma_1}\right)\leq \frac{1}{2},\quad \zeta_1(x)=\zeta\left(\frac{x}{\sigma_1}\right)\leq \frac{1}{2},$$
	whenever $\vert x\vert\geq \sigma_1$. Hence, $\sigma_1$ really reflects an effective range of TCS self-interactions in the first species.

	\medskip
	
		$\bullet$ {\bf (Strong relaxation of $\theta$)}
		In this case, we assume that
		\begin{equation}\label{E-scaling-singular-fast}
		\mu=\mathcal{O}(1),\quad \tau^v=\mathcal{O}(\varepsilon^{1+\lambda_1}),\quad \tau^\theta=\mathcal{O}(\varepsilon^{1+\lambda_2}),\quad \tau^a=\mathcal{O}(\e),\quad\tau_c^v=\tau_c^\theta=\mathcal{O}(\varepsilon),
		\end{equation}
		as $\varepsilon\rightarrow 0$. This means that both species consist of particles with comparable mass. In addition, TCS, aggregation and relaxation interactions for the first species are stronger than for the second one. In particular, relaxation of the internal variable of the first species towards the second one is weak. This choice leads to the system
		\begin{align}\label{E-kinetic-singular-fast}
		\begin{aligned}
		&\partial_t f_\e+v\cdot \nabla_x f_\e+\frac{1}{\varepsilon}\nabla_v\cdot \left(F_\varepsilon[f_\e]\,f_\e+H[f_\e]f_\e+F_c[\bar \rho,\bar u,\bar e]f_\e\right)\\
		&\hspace{2.5cm}+\frac{1}{\varepsilon}\partial_\theta\left(G_\varepsilon[f_\e]\,f+G_c[\bar \rho,\bar e]\,f_\e\right)=0,\\
		&F_\varepsilon[f](t,x,v,\theta):=\int_{\mathbb{R}^{2d}\times\bbr_+}\phi_\varepsilon(x-x_*)\left(\frac{v_*}{\theta_*}-\frac{v}{\theta}\right)f(t,z_*)\,dz_*,\\
		&G_\varepsilon[f](t,x,v,\theta):=\int_{\mathbb{R}^{2d}\times\bbr_+}\zeta_\varepsilon(x-x_*)\left(\frac{1}{\theta}-\frac{1}{\theta_*}\right)f(t,z_*)\,dz_*,
		\end{aligned}
		\end{align}
		where $\phi_\varepsilon$ and $\zeta_\varepsilon$ are now the singularly scaled influence functions
		\begin{equation}\label{E-coupling-weights-epsilon}
		\phi_\varepsilon(x)=\frac{1}{(\varepsilon^2+c_{\lambda_1}\vert x\vert^2)^{\lambda_1/2}},\quad \zeta_\varepsilon(x)=\frac{1}{(\varepsilon^2+c_{\lambda_2}\vert x\vert^2)^{\lambda_2/2}},\quad x\in \mathbb{R}^d.
		\end{equation}
        Here, $W$ is a generic aggregation potential verifying the regularity condition \eqref{H-hypothesis-phi-zeta-W}.

		\medskip
		
		$\bullet$ {\bf (Weak relaxation of $\theta$)}
		In this case, we assume that
		\begin{equation}\label{E-scaling-singular-slow}
		\mu=\mathcal{O}(1),\quad \tau^v=\mathcal{O}(\varepsilon^{1+\lambda_1}),\quad \tau^\theta=\mathcal{O}(\varepsilon^{1+\lambda_2}),\quad \tau^a=\mathcal{O}(\e),\quad \tau_c^v=\mathcal{O}(\varepsilon),\quad \tau_c^\theta=\mathcal{O}(\varepsilon^{\lambda_2}),
		\end{equation}
		as $\varepsilon\rightarrow 0$. This has similar implications as \eqref{E-scaling-singular-fast} except for the fact that relaxation of the internal variable of the first species towards the second one is slow and it occurs at an intermediate scale. This choice leads to the system
		\begin{align}\label{E-kinetic-singular-slow}
		\begin{aligned}
		&\partial_t f_\e+v\cdot \nabla_x f_\e+\frac{1}{\varepsilon}\nabla_v\cdot \left(F_\varepsilon[f_\e]\,f_\e+H[f_\e]\,f_\e+F_c[\bar \rho,\bar u,\bar e]f_\e\right)\\
		&\hspace{2.5cm}+\partial_\theta\left(\frac{1}{\varepsilon}G_\varepsilon[f_\e]\,f_\e+G_c[\bar \rho,\bar e]\,f_\e\right)=0.
		\end{aligned}
		\end{align}
	
	\begin{remark}[Singular influence functions]
		Notice that in the singular regime, the scale influence functions $\phi_\varepsilon$ and $\zeta_\varepsilon$ become singular at the origin asymptotically when $\varepsilon\rightarrow 0$. In fact, they take the form
		\begin{align}\label{E-singular-kernels}
		\begin{aligned}
		\phi_0(x)&:=\frac{1}{c_{\lambda_1}^{\lambda_1/2}\vert x\vert^{\lambda_1}},\quad x\in \mathbb{R}^d\setminus\{0\},\\
		\zeta_0(x)&:=\frac{1}{c_{\lambda_2}^{\lambda_2/2}\vert x\vert^{\lambda_2}},\quad x\in \mathbb{R}^d\setminus\{0\}.
		\end{aligned}
		\end{align}
	\end{remark}
	
	\begin{remark}[Singular aggregation potentials]\label{R-singular-aggregation-potentials}
	Due to our scaling assumption $\delta_a=\mathcal{O}(1)$, the limiting potential $W$ remains smooth in the limit. However, we could also assume $\delta_a=\mathcal{O}(\e)$ and a specific form of $W$, so that we recover meaningful singular potentials in the limit. For instance, Cucker and Dong \cite{CD14} proposed a novel alignment-aggregation model with similar choices of potentials $W$ like the influence functions \eqref{influence-function} in C-S model:
	\begin{equation}\label{E-Cucker-Dong-potentials}
	W(x)=\begin{dcases}
	\frac{1}{(1-\lambda_3)c_{\lambda_3}^{1/2}}(1+c_{\lambda_3}\vert x\vert^2)^{\frac{1-\lambda_3}{2}},  & \mbox{if }\lambda_3\in \mathbb{R}_+\setminus\{1\},\\
	\frac{1}{2c_{\lambda_3}^{1/2}}\log(1+c_{\lambda_3}\vert x\vert^2), & \mbox{if }\lambda_3=1.
	\end{dcases}
	\end{equation}
     These are attractive potentials, but a similar argument builds repulsive ones. Taking $\delta_a=\mathcal{O}(\e)$ and $\tau^a=\mathcal{O}(\e^{\lambda_3})$ when $\e\to 0$, we note that $H[f_\e]$ in \eqref{E-kinetic-singular-fast} and \eqref{E-kinetic-singular-slow} is replaced by
    $$
    H_\e[f_\e](t,x):=-\int_{\mathbb{R}^{2d}\times \mathbb{R}_+}\nabla W_\e(x-x_*)f_\e(t,z_*)\,dz_*,
    $$
    where the scaled aggregation potentials $W_\e$ take the form
    \begin{equation}\label{E-Cucker-Dong-potentials-scaled}
    W_\e(x)=\begin{dcases}
	\frac{1}{(1-\lambda_3)c_{\lambda_3}^{1/2}}(\e^2+c_{\lambda_3}\vert x\vert^2)^{\frac{1-\lambda_3}{2}},  & \mbox{if }\lambda_3\in \mathbb{R}_+\setminus\{1\},\\
	\frac{1}{2c_{\lambda_3}^{1/2}}\log(\e^2+c_{\lambda_3}\vert x\vert^2), & \mbox{if }\lambda_3=1.
	\end{dcases}
    \end{equation}
	Note that we find various singular potentials when $\e\to0$, ranging from H\"{o}lder-continuous Lennard-Jones potentials with sublinear growth, to Riesz and Newtonian potentials with logarithmic and algebraic singularities, see Figure \ref{fig:CD-singular-potentials}. We skip the analysis of singular aggregation kernels in this paper, but we shall briefly elaborate on some cases in Section \ref{subsec:3.5}.
	\end{remark}
	
	\begin{figure}
	\centering
    \begin{subfigure}[b]{0.3\textwidth}
    \begin{tikzpicture}[scale=0.65]
    \begin{axis}[
        axis x line=middle, axis y line=middle,
        xmin=0, xmax=2, xtick={0,0.5,1,1.5,2},
        ymin=0, ymax=2, ytick={0,0.5,1,1.5,2},
    ]
    \addplot [
        domain=0:2, 
        samples=200, 
        color=blue,
        line width=0.4mm,
    ]
    {pow(x,1-0.5)/((1-0.5)*pow(pow(2,2/0.5)-1,0.5/2))};
    \end{axis}
    \end{tikzpicture}
    \caption{$\lambda_3=0.5$}
    \label{fig:CD-singular-potentials-lambda3-0.5}
    \end{subfigure}
    \hspace{0.3cm}
    \begin{subfigure}[b]{0.3\textwidth}
    \begin{tikzpicture}[scale=0.65]
    \begin{axis}[
        axis x line=middle, axis y line=middle,
        xmin=0, xmax=3, xtick={0,...,4},
        ymin=-4, ymax=2, ytick={-4,...,2},
    ]
    \addplot [
        domain=0:3, 
        samples=200, 
        color=red,
        line width=0.4mm,
    ]
    {ln((pow(2,2/2)-1)*pow(x,2))/(2*pow(pow(2,2/2)-1,1/2))};
    \end{axis}
    \end{tikzpicture}
    \caption{$\lambda_3=1$}
    \label{fig:CD-singular-potentials-lambda3-1}
    \end{subfigure}
    \hspace{0.3cm}
    \begin{subfigure}[b]{0.3\textwidth}
    \begin{tikzpicture}[scale=0.65]
    \begin{axis}[
        axis x line=middle, axis y line=middle,
        xmin=0, xmax=5, xtick={0,...,5},
        ymin=-5, ymax=0, ytick={-5,...,0},
    ]
    \addplot [
        domain=0:5, 
        samples=200, 
        color=magenta,
        line width=0.4mm,
    ]
    {pow(x,1-2)/((1-2)*pow(pow(2,2/2)-1,2/2))};
    \end{axis}
    \end{tikzpicture}
    \caption{$\lambda_3=2$}
    \label{fig:CD-singular-potentials-lambda3-2}
    \end{subfigure}
    \caption{Singular aggregation potential $W_0$ with $\lambda_3=0.5$, $\lambda_3=1$ and $\lambda_3=2$.}\label{fig:CD-singular-potentials}
    \end{figure}
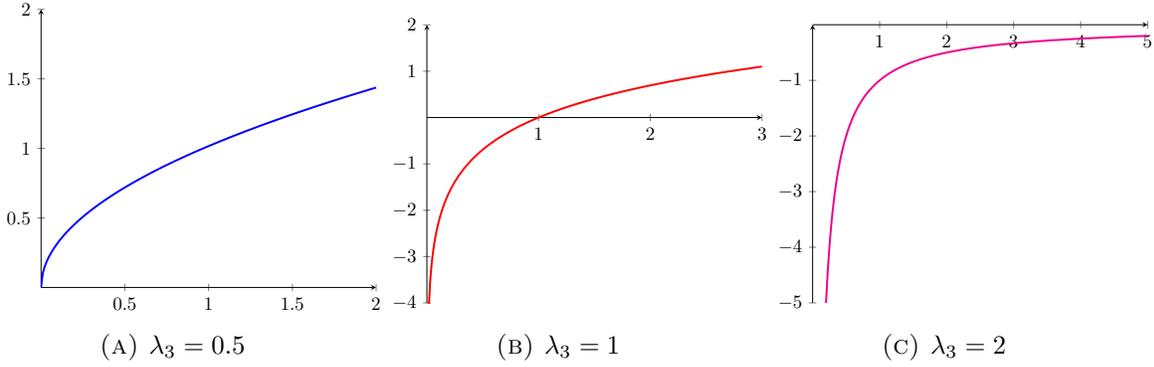

	\section{Hydrodynamic limit in the strong relaxation regime}
	\label{sec:2}
	\setcounter{equation}{0}
	
	In this section, we derive the hydrodynamic limit of system \eqref{A-4} towards \eqref{A-5}. To this end, we first prove that the internal variable support of $f_\e$ shrinks fast enough so that $\theta$ concentrates at the background mean value $\theta^\infty(t)$ (see \eqref{E-background-mean-theta}), after some very small initial time layer. With such control in hand, we obtain appropriate \textit{a priori} estimates for the velocity moments. Finally, we combine the estimates for velocity moments and the internal variable support to derive the rigorous hydrodynamic limit. 
	Throughout this part, we will always assume the following hypothesis on initial data:
	\begin{align}\label{E-hypothesis-initial-data}
	\begin{aligned}
	& f_\varepsilon^0(x,v,\theta)\geq 0\ \mbox{ and }\ f_\varepsilon^0\in C^\infty_c(\mathbb{R}^{2d}\times \mathbb{R}_+),\\
	& \Vert f_\varepsilon^0\Vert_{L^1(\mathbb{R}^{2d}\times \mathbb{R}_+)}=1\ \mbox{ and }\ \rho_\varepsilon^0\overset{*}{\rightharpoonup} \rho^0\mbox{ in }\mathcal{M}(\mathbb{R}^d),\\
	&\Vert \vert x\vert f_\varepsilon^0\Vert_{L^1(\mathbb{R}^{2d}\times \mathbb{R}_+)}\leq M_0,\ \Vert \vert v\vert^2 f_\varepsilon^0\Vert_{L^1(\mathbb{R}^{2d}\times \mathbb{R}_+)}\leq E_0\ \mbox{ and }\ \supp_\theta f_\varepsilon^0\subseteq [\theta_m^0,\theta_M^0],
	\end{aligned}
	\end{align}
	for each $\varepsilon>0$, where $M_0$, $E_0$, and $\theta_m^0<\theta_M^0$ are positive $\varepsilon$-independent constants. In addition, we will assume that $(\bar\rho,\bar u,\bar e)$ is any fixed strong (Lipschitz) solution to the macroscopic species \eqref{eq-hydro-TCS-couple}, which verifies the following properties within a time interval $[0,T]$
	\begin{align}\label{E-hypothesis-macro-species}
	\begin{aligned}
	&\bar \theta_m\leq \bar e(t,x)\leq \bar \theta_M,\\
	&|\bar u(t,x)|\le \bar{v}_M,
	\end{aligned}
	\end{align}
	for any $t\in [0,T]$ and $x\in\supp\bar\rho(t,\cdot)$, where $\bar\theta_m<\bar \theta_M$ and $\bar{v}_M$ are positive $\varepsilon$-independent constants. Finally, assume the compatibility condition
	\begin{equation}\label{E-hypothesis-thetaM-thetam}
	\theta_M^0-\theta_m^0< \frac{\min\{\theta_m^0,\bar\theta_m\}^2}{\max\{\theta_M^0,\bar\theta_M\}}.
	\end{equation}

	\begin{remark}[Strong solutions of the background fluid \eqref{eq-hydro-TCS-couple}]
	We emphasize that we are assuming the existence of a strong solution to \eqref{eq-hydro-TCS-couple} verifying the uniform control \eqref{E-hypothesis-macro-species}.

	$\bullet$ {\bf (Periodic domain)} This is justified when the spatial domain is the periodic box $\mathbb{T}^d$. Indeed, in \cite{HKMRZ18} the authors proved the existence of unique local-in-time strong solutions to the hydrodynamic  equations \eqref{eq-hydro-TCS-couple} under mild assumptions on the initial data $(\bar\rho^0,\bar u^0 ,\bar e^0)\in H^s(\mathbb{T}^d)\times H^{s+1}(\mathbb{T}^d)\times H^{s+1}(\mathbb{T}^d)$, for any integer $s>\frac{d}{2}+1$ and Lipschitz-continuous influence functions $\phi$ and $\zeta$. In particular, this suggests the existence of a small enough $T>0$ such that assumptions \eqref{E-hypothesis-macro-species} fulfill due to the compactness of the domain. Indeed, for small enough initial data, solutions were proved to be global in time, so that we could indeed set $T=\infty$.
	
	$\bullet$ {\bf (Free space)} For the free space $\mathbb{R}^d$ an analogue well-posedness result is not available in the literature. However, we still expect the coexistence of blow-ups and global strong solutions. Note that the persistence of strong solutions is guaranteed as long as $\nabla_x \bar u(t,x)$ remains uniformly bounded. This has been exploited in \cite{TT14} for the Euler-alignment system \eqref{A-0-2} in $\mathbb{R}^d$. Namely, the authors found an explicit critical threshold so that sub-critical initial configurations lead to global strong solutions whilst super-critical ones imply finite time blow-ups.	In fact, for influence function $\phi$ with slow tails (i.e. $\lambda_1\in (0,1)$) \textit{global strong solutions must flock}, according to Theorem 2.2 in \cite{TT14}. Then, we recover similar uniform control \eqref{E-hypothesis-macro-species} up to $T=\infty$ for sub-critical initial configurations.
	\end{remark}

	\begin{remark}[Control on $\theta^\infty$ and $u^\infty$]\label{R-hypothesis-macro-species-gain}
		Under conditions \eqref{E-hypothesis-macro-species} on the solution $(\bar\rho,\bar u,\bar e)$ to \eqref{eq-hydro-TCS-couple}, we obtain extra regularity $\theta^\infty\in W^{1,\infty}(0,T)$ and $\frac{u^\infty}{\theta^\infty}\in L^2(0,T)$. Specifically,
		\begin{align}\label{E-hypothesis-macro-species-gain-theta-infty}
		\begin{aligned}
		&\bar \theta_m\leq \theta^\infty(t)\leq \bar\theta_M,\\
		&0\leq \frac{d\theta^\infty}{dt}\leq \Vert \zeta\Vert_{L^\infty}\frac{\bar \theta_M^2}{\bar \theta_m^5}(\bar \theta_M-\bar\theta_m)^2
		\end{aligned}
		\end{align}
		for all $t\in [0,T]$ and
			\begin{equation}\label{E-hypothesis-macro-species-gain-u-infty}
			\int_0^T \left|\frac{u^\infty(t)}{\theta^\infty(t)}\right|^2\,dt<T\left|\frac{\bar{v}_M}{\bar{\theta}_m}\right|^2=:F_0^2.
			\end{equation}
		On the one hand, $\eqref{E-hypothesis-macro-species-gain-theta-infty}_1$ follows by integrating in $\eqref{E-hypothesis-macro-species}_1$ against $\bar\rho$. On the other hand, using the hydrodynamic equations for the background fluid
		\begin{align*}
		\frac{d\theta^\infty}{dt}&=\frac{d}{dt}\left(\int_{\bbr^d}\frac{\bar{\rho}}{\bar{e}}\,dx\right)^{-1}=-(\theta^\infty)^2\frac{d}{dt}\int_{\bbr^d}\frac{\bar{\rho}}{\bar{e}}\,dx = -(\theta^\infty)^2\int_{\bbr^d} \left(\frac{\partial_t \bar{\rho}}{\bar{e}}-\frac{\bar{\rho}\partial_t\bar{e}}{\bar{e}^2}\right)\,dx\\
		&=-(\theta^\infty)^2\int_{\bbr^d}\left(\frac{\nabla\cdot(\bar{\rho}\bar{u})}{\bar{e}}-\frac{\bar{\rho}\bar{u}\cdot\nabla\bar{e}}{\bar{e}^2}\right)\,dx \\
		&\quad +\frac{(\theta^\infty)^2}{2}\int_{\bbr^{2d}}\zeta(x-x_*)\left(\frac{1}{\bar{e}^2(t,x)}-\frac{1}{\bar e^2(t,x_*)}\right)\\
		& \hspace{2.7cm} \times \left(\frac{1}{\bar{e}(t,x)}-\frac{1}{\bar{e}(t,x_*)}\right)\bar{\rho}(t,x)\bar{\rho}(t,x_*)\,dx\,dx_*,
		\end{align*}
		for every $t\in [0,T]$. By integrating by parts, the first term in the right-hand side vanishes. Using the preceding control on $\theta^\infty$ along with the upper and lower bounds of $\bar e$ in $\eqref{E-hypothesis-macro-species}_1$ we conclude $\eqref{E-hypothesis-macro-species-gain-theta-infty}_2$. In particular, the smaller $(\bar\theta_m,\bar\theta_M)$, the flatter the slope $\frac{d\theta^\infty}{dt}$. Similarly, it follows from the definition of $u^\infty(t)$ that
			\begin{align*}
			\left|\frac{u^\infty(t)}{\theta^\infty(t)}\right|=\left|\int_{\bbr^d}\frac{\bar\rho(t,x)\bar u(t,x)}{\bar e(t,x)}\,dx\right|\le \left|\frac{\bar{v}_M}{\bar\theta_m}\right|, 
			\end{align*}
		for every $t\in [0,T]$, which implies the claimed $L^2$ estimate \eqref{E-hypothesis-macro-species-gain-u-infty}.
	\end{remark}

	\subsection{Hierarchy of moments}
	Under the above assumptions \eqref{E-hypothesis-initial-data}, \eqref{E-hypothesis-macro-species}, there exists a unique strong solution $f_\varepsilon$ to system \eqref{A-4} with initial data $f_\varepsilon^0$. Then, we can define the following hierarchy of moments
	\begin{align}\label{E-hierarchy-moments}
	\begin{aligned}
	\rho_\varepsilon(t,x)&:=\int_{\mathbb{R}^d\times \mathbb{R}_+}f_\varepsilon\,dv\,d\theta, & \mathcal{S}_\varepsilon^v(t,x)&:=\int_{\mathbb{R}^d\times \mathbb{R}_+}v\otimes vf_\varepsilon\,dv\,d\theta,\\
	j_\varepsilon(t,x)&:=\int_{\mathbb{R}^d\times \mathbb{R}_+}vf_\varepsilon\,dv\,d\theta, & \mathcal{S}_\varepsilon^\theta (t,x)&:=\int_{\mathbb{R}^d\times \mathbb{R}_+}v\,\theta f_\varepsilon\,dv\,d\theta,\\
	h_\varepsilon(t,x)&:=\int_{\mathbb{R}^d\times \mathbb{R}_+}\theta f_\varepsilon\,dv\,d\theta, & A_\varepsilon(t,x)&:=\int_{\mathbb{R}^d\times\bbr_+}\frac{v}{\theta} f_\varepsilon\,dv\,d\theta,\\
	& & B_\varepsilon(t,x)&:=\int_{\mathbb{R}^d\times \mathbb{R}_+}\frac{1}{\theta}f_\varepsilon\,dv\,d\theta.
	\end{aligned}
	\end{align}
	We notice that, given the strong nonlinearity of the forcing terms $F$, $F_c$, $G$ and $G_c$ in \eqref{A-4}, many nonlinear moments are involved in the dynamics. By multiplying \eqref{A-4} by $1$, $v$ and $\theta$ and integrating by parts, and using the definition of moments in \eqref{E-hierarchy-moments}, we obtain the following hierarchy of equations
	\begin{align}\label{E-hierarchy-moments-equations}
	\begin{aligned}
	&\partial_t\rho_\e+\nabla\cdot j_\e=0,\\
	&\e\partial_tj_\e+\e\nabla\cdot\mathcal{S}_\varepsilon^v+A_\varepsilon(\phi*\rho_\varepsilon)-\rho_\varepsilon(\phi*A_\varepsilon)+(\nabla W*\rho_\e)\rho_\e+A_\varepsilon-\rho_\varepsilon\frac{u^\infty(t)}{\theta^\infty(t)}=0,\\
	&\e\partial_th_\e+\e\nabla\cdot\mathcal{S}_\varepsilon^\theta+\rho_\varepsilon(\zeta*B_\varepsilon)-B_\varepsilon(\zeta*\rho_\varepsilon)+\frac{\rho_\varepsilon}{\theta^\infty(t)}-B_\varepsilon=0,
	\end{aligned}
	\end{align}
	where $\theta^\infty$ and $u^\infty$ are the background mean value of the internal variable and velocity respectively, see \eqref{E-background-mean-theta} and \eqref{E-background-mean-velocity}. Our final goal is to show rigorously that we can close the above hierarchy \eqref{E-hierarchy-moments-equations} of equations when $\varepsilon\rightarrow 0$. As a consequence of the scaling, the inertial terms $\e\nabla_x\cdot \mathcal{S}_\e^v$ and $\e\nabla_x\cdot\mathcal{S}_\e^\theta$ in second and third equations will disappear in the limit, see e.g. \cite{FST16,PS17} for the Cucker-Smale model.

	\subsection{Concentration of the internal variable support}
	In this part, we present the exponential concentration of the internal variable support of $f_\e$. To such an end, we define 
	\begin{align}\label{E-diameters}
	\begin{aligned}
	D_x^\varepsilon(t)&:=\diam(\supp_x f_\varepsilon(t)),\\
	D_v^\varepsilon(t)&:=\diam(\supp_v f_\varepsilon(t)),\\
	D_\theta^\varepsilon(t)&:=\diam(\supp_\theta f_\varepsilon(t)),
	\end{aligned}
	\end{align}
	for every $t\in [0,T]$ and $\varepsilon>0$. Recall that the characteristic system of \eqref{A-4} consists in the trajectories $Z_\varepsilon(t;0,z_0)=(X_\varepsilon(t;0,z_0),V_\varepsilon(t;0,z_0),\Theta_\varepsilon(t;0,z_0))$ solving
	\begin{align}\label{E-characteristic-system-A-4}
	\begin{aligned}
	&\frac{dX_\varepsilon}{dt}=V_\varepsilon,\\
	&\frac{dV_\varepsilon}{dt}=\frac{1}{\varepsilon}F[f_\varepsilon](t,Z_\varepsilon)+\frac{1}{\e}H[f_\e](t,Z_\e)+\frac{1}{\varepsilon}F_c[\bar\rho,\bar u,\bar e](t,Z_\varepsilon),\\
	&\frac{d\Theta_\varepsilon}{dt}=\frac{1}{\varepsilon}G[f_\varepsilon](t,Z_\varepsilon)+\frac{1}{\varepsilon}G_c[\bar \rho,\bar e](t,Z_\varepsilon),\\
	&Z_\varepsilon(0;0,z_0)=z_0,
	\end{aligned}
	\end{align}
	for $t\in [0,T]$ and any $z_0\in \mathbb{R}^{2d}\times \mathbb{R}_+$.
	
	\begin{lemma}[Concentration of internal variable]\label{lemma_temp}
		Let $f_\e$ be the solution of \eqref{A-4} subject to initial data $f^0_\e$. Suppose that initial data verify \eqref{E-hypothesis-initial-data} and $(\bar \rho,\bar u,\bar e)$ verifies \eqref{E-hypothesis-macro-species}. Consider the constants $0<\theta_m<\theta_M$ defined by
		\begin{equation}\label{E-thetam-thetaM}
		\theta_M:=\max\{\theta_M^0,\bar \theta_M\},\quad \theta_m:=\min\{\theta_m^0,\bar \theta_m\}.
		\end{equation}
		Then, the following estimates hold:
		\begin{enumerate}
			\item {\bf (Confinement of support)}
			\[\textup{supp}_\theta f_\varepsilon(t)\subset (\theta_m,\theta_M),\quad \mbox{for all} \quad t\in [0,T].\]
			\item {\bf (Exponential concentration)}
			\[D_\theta^\varepsilon(t)\le D_\theta^\varepsilon(0)e^{-\frac{1}{\e\theta_M^2}t},\quad \mbox{for all}\quad t\in [0,T].\]
			\item {\bf (Deviation from the background mean value $\theta^\infty(t)$)}
			\[\textup{supp}_\theta f_\varepsilon(t)\subset(\theta^\infty(t)-\eta_\e(t),\theta^\infty(t)+\eta_\e(t)),\quad \mbox{for all}\quad t\in [0,T].\]
			Here, the error $\eta_\e(t)$ is given as
			\[\eta_\e(t)=Ce^{-\frac{c}{\varepsilon}t}+C\varepsilon,\quad t\in [0,T],\]
			where $C,c\in \mathbb{R}_+$ only depend on parameters $\theta_M^0$, $\bar \theta_M$ and $\Vert \theta^\infty\Vert_{W^{1,\infty}(0,T)}$.
		\end{enumerate}
	\end{lemma}
	
	\begin{proof}
		Recall that by the  upper and lower bounds on $\bar e$, we readily infer analogue estimates $\bar\theta_m\leq \theta^\infty\leq \bar \theta_M$ for the mean value $\theta^\infty(t)$ of the internal variable in the background species, see \eqref{E-hypothesis-macro-species-gain-theta-infty} in Remark \ref{R-hypothesis-macro-species-gain}. We shall systematically use such a control along the proof. Let us define the maximum and minimum value of the phase support
		\begin{align*}
		\theta_\varepsilon^M(t)&:=\max\supp_\theta f_\varepsilon(t)=\max_{z_0\in \supp f_\varepsilon(0)}\Theta_\varepsilon(t;0,z_0),\\
		\theta_\varepsilon^m(t)&:=\min\supp_\theta f_\varepsilon(t)=\min_{z_0\in \supp f_\varepsilon(0)}\Theta_\varepsilon(t;0,z_0),
		\end{align*}
		where $\Theta_\varepsilon=\Theta_\varepsilon(t;0,z_0)$ is the $\theta$-component of the characteristic system \eqref{E-characteristic-system-A-4}. We notice that it is not necessarily true that maximum and minimum is propagated along a fixed characteristic; in other words, $\theta_\varepsilon^M(t)$ and $\theta_\varepsilon^m(t)$ are not necessarily characteristics. However, since the flow $(t,z_0)\mapsto \Theta_\varepsilon(t;0,z_0)$ is continuous, differentiable with respect to $t$ with continuous derivative, then a usual argument (see e.g. \cite[Corollary 3.3, Chapter 1]{DR95}) yields
		\begin{align}\label{E-derivative-maximum}
		\begin{aligned}
		\frac{d^+\theta_\varepsilon^M(t)}{dt}&=\max_{z_0\in M_\varepsilon(t)}\frac{\partial \Theta_\varepsilon}{\partial t}(t;0,z_0),\\
		\frac{d^+\theta_\varepsilon^m(t)}{dt}&=\min_{z_0\in m_\varepsilon(t)}\frac{\partial \Theta_\varepsilon}{\partial t}(t;0,z_0),
		\end{aligned}
		\end{align}
		for every $t\in [0,T]$, where $d^+/dt$ stands for the right sided derivative, and the critical sets $M_\varepsilon(t)$ and $m_\varepsilon(t)$ take the form
		\begin{align*}
		M_\varepsilon(t)&:=\{z_0\in \supp_\theta f_\varepsilon(0):\,\theta_\varepsilon^M(t)=\Theta_\varepsilon(t;0,z_0)\},\\
		m_\varepsilon(t)&:=\{z_0\in \supp_\theta f_\varepsilon(0):\,\theta_\varepsilon^m(t)=\Theta_\varepsilon(t;0,z_0)\}.
		\end{align*}
		We fix any $t\in [0,T]$ and $z^M_0\in M_\varepsilon(t)$ and use the characteristic system \eqref{E-characteristic-system-A-4} to achieve:
		\begin{align}\label{E-Theta_epsilon-max-derivative}
		\begin{aligned}
		\frac{\partial\Theta_\varepsilon}{\partial t}(t;0,z^M_0)&=\frac{1}{\e}\left(G[f_\e](t,Z_\varepsilon(t;0,z_0^M))+G_c[\bar{\rho},\bar{e}](t,Z_\varepsilon(t;0,z_0^M))\right)\\
		&=\frac{1}{\e}\int_{\bbr^{2d}\times\bbr_+}\zeta(X_\varepsilon(t;0,z_0^m)-x_*)\left(\frac{1}{\theta^M_\e(t)}-\frac{1}{\theta_*}\right)f_\e(t,z_*)\,dz_*\\
		& \quad +\frac{1}{\e}\left(\frac{1}{\theta^M_\e(t)}-\frac{1}{\theta^\infty(t)}\right)\\
		&\leq \frac{1}{\e}\int_{\bbr^{2d}\times\bbr_+}\zeta(D_x^\varepsilon(t))\left(\frac{1}{\theta^M_\e(t)}-\frac{1}{\theta_*}\right)f_\e(t,z_*)\,dz_*+\frac{1}{\e}\left(\frac{1}{\theta^M_\e(t)}-\frac{1}{\theta^\infty(t)}\right),
		\end{aligned}
		\end{align}
		where we used a maximality of $\theta_\e^M$ in the last inequality. Since $z_0^M\in M_\varepsilon(t)$ is arbitrary, \eqref{E-derivative-maximum} implies
		\begin{align}\label{thetaM}
		\begin{aligned}
		\frac{d^+\theta_\varepsilon^M}{dt}&\leq  \frac{1}{\e}\int_{\bbr^{2d}\times\bbr_+}\zeta(D_x^\varepsilon(t))\left(\frac{1}{\theta^M_\e(t)}-\frac{1}{\theta_*}\right)f_\e(t,z_*)\,dz_*+\frac{1}{\e}\left(\frac{1}{\theta^M_\e(t)}-\frac{1}{\theta^\infty(t)}\right)\\
		&\le \frac{1}{\e}\left(\frac{1}{\theta^M_\e(t)}-\frac{1}{\theta^\infty(t)}\right)\le \frac{1}{\e}\left(\frac{1}{\theta^M_\e(t)}-\frac{1}{\bar\theta_M}\right)=\frac{1}{\e}\frac{\bar \theta_M-\theta^M_\e(t)}{\bar\theta_M\theta^M_\e(t)},
		\end{aligned}
		\end{align}
		for every $t\in [0,T]$. A similar argument with $\theta_\varepsilon^m$ yields
		\begin{align}\label{thetam}
		\begin{aligned}
		\frac{d^+\theta_\varepsilon^m}{dt}&\geq  \frac{1}{\e}\int_{\bbr^{2d}\times\bbr_+}\zeta(D_x^\varepsilon(t))\left(\frac{1}{\theta^m_\e(t)}-\frac{1}{\theta_*}\right)f_\e(t,z_*)\,dz_*+\frac{1}{\e}\left(\frac{1}{\theta^m_\e(t)}-\frac{1}{\theta^\infty(t)}\right)\\
		&\geq \frac{1}{\e}\left(\frac{1}{\theta^m_\e(t)}-\frac{1}{\theta^\infty(t)}\right)\geq \frac{1}{\e}\left(\frac{1}{\theta^m_\e(t)}-\frac{1}{\bar\theta_m}\right)=\frac{1}{\e}\frac{\bar \theta_m-\theta^m_\e(t)}{\bar\theta_m\theta^m_\e(t)},
		\end{aligned}
		\end{align}
		for every $t\in [0,T]$.

		\medskip
		
		$\bullet$ {\sc Step 1}.	We claim that
		\[\theta^M_{\e}(t)\le \max\{\theta_M^0,\bar \theta_M\}=:\theta_M,\quad \mbox{for}\quad t\in [0,T].\]
		To show this result, we show that $\theta^M_\e(t)< \theta_M+\delta$ for any $\delta>0$ and $t\in [0,T]$. We define
		\[t_2:=\sup\{t\in [0,T]:~\theta^M_\e(s)<\theta_M+\delta \quad\mbox{for}\quad 0\le s\le t\}.\]
		Assume by contradiction that $t_2<+\infty$. Then, we can define
		$$t_1:=\inf\{t\in [0,t_2]:\,\theta_M+\frac{\delta}{2}\leq \theta_\varepsilon^M(s),\quad \mbox{for}\quad t\leq s\leq t_2\}.$$
		By continuity along the fact that $\theta_\varepsilon^M(0)\leq \theta_M$, we obtain that $0<t_1<t_2$ and
		\begin{align}
		&\theta_\varepsilon^M(t_1)=\theta_M+\frac{\delta}{2},\qquad \theta_\varepsilon^M(t_2)=\theta_M+\delta,\label{E-t1-t2}\\
		&\theta_M+\frac{\delta}{2}\leq \theta_\varepsilon^M(t)\leq \theta_M+\delta, \quad \mbox{for}\quad t_1\leq t\leq t_2.\label{E-t1-t2-intermediate}
		\end{align}
		On the one hand, by \eqref{E-t1-t2} we obtain
		$$\frac{\theta_\varepsilon^M(t_2)-\theta_\varepsilon^M(t_1)}{t_2-t_1}=\frac{\delta}{2(t_2-t_2)}>0.$$
		On the other hand, by \eqref{thetaM} and \eqref{E-t1-t2-intermediate}, we achive
		$$\frac{d^+\theta_\varepsilon^M}{dt}(t)\leq \frac{1}{\varepsilon}\frac{\theta_M-\theta_\varepsilon^M(t)}{\bar \theta_M\theta_\varepsilon^M(t)}\leq 0,$$
		for each $t\in [t_1,t_2]$. Consequently, by the mean value theorem adapted to one sided derivatives (see e.g. \cite{MV86}) we infer
		$$\frac{\theta_\varepsilon^M(t_2)-\theta_\varepsilon^M(t_1)}{t_2-t_2}\leq \sup_{t \in (t_1,t_2)}\frac{d ^+\theta_\varepsilon^M}{dt}(t)\leq 0,$$
		which yields a contradiction. Therefore, we conclude our claim. An analogous continuity argument can be applied to the minimal value $\theta_\varepsilon^m(t)$ and we obtain
		$$\theta_\varepsilon^m(t)\geq \min\{\theta_m^0,\bar \theta_m\}=:\theta_m,\quad \mbox{for}\quad t\in [0,T].$$
		
		\medskip
		
		$\bullet$ {\sc Step 2}. To estimate the diameter of the internal variable, we notice that
		$$D_\theta^\varepsilon(t)=\theta_\varepsilon^M(t)-\theta_\varepsilon^m(t),$$
		for each $t\in [0,T]$. Taking the difference between estimates \eqref{thetaM} and \eqref{thetam} we obtain
		\[
		\frac{d^+ D_\theta^\varepsilon}{dt}\le-\frac{1}{\e}\zeta(D_x^\varepsilon(t))\frac{D_\theta^\varepsilon(t)}{\theta_\varepsilon^M(t)\theta_\varepsilon^m(t)}-\frac{1}{\e}\frac{D_\theta^\varepsilon(t)}{\theta^M_\e(t)\theta^m_\e(t)}\le -\frac{D_\theta^\varepsilon(t)}{\e \theta_M^2},
		\]
		for each $t\in [0,T]$, where we have used the uniform bounds in the previous step. This implies the desired exponential decay for the diameter of internal variable.

		\medskip
		
		$\bullet$ {\sc Step 3}. Finally, we estimate the distance of the support of $f_\e(t)$ to $\theta^\infty(t)$. To such an end, we apply the ideas \eqref{E-derivative-maximum} and similiar ideas as in \eqref{E-Theta_epsilon-max-derivative} to obtain
		
		\begin{align*}
		\frac{d^+}{dt}(\theta_\varepsilon^M-\theta^\infty)&\leq \int_{\mathbb{R}^{2d}\times \mathbb{R}_+}\zeta(D_x^\varepsilon(t))\left(\frac{1}{\theta_\varepsilon^M(t)}-\frac{1}{\theta_*}\right)f_\varepsilon(t,z_*)\,dz_*\\
		&\quad -\frac{1}{\varepsilon}\frac{\theta_\varepsilon^M(t)-\theta^\infty(t)}{\theta_\varepsilon^M(t)\theta^\infty(t)}-\frac{d\theta^\infty}{dt}\leq -\frac{C_\varepsilon^M(t)}{\varepsilon}(\theta_\varepsilon^M(t)-\theta^\infty(t))+\left\vert\frac{d\theta^\infty}{dt}\right\vert,\\
		\frac{d^+}{dt}(\theta_\varepsilon^m-\theta^\infty)&\geq \int_{\mathbb{R}^{2d}\times \mathbb{R}_+}\zeta(D_x^\varepsilon(t))\left(\frac{1}{\theta_\varepsilon^m(t)}-\frac{1}{\theta_*}\right)f_\varepsilon(t,z_*)\,dz_*\\
		&\quad -\frac{1}{\varepsilon}\frac{\theta_\varepsilon^m(t)-\theta^\infty(t)}{\theta^m_\varepsilon(t)\theta^\infty(t)}-\frac{d\theta^\infty}{dt}\geq -\frac{C_\varepsilon^m(t)}{\varepsilon}(\theta_\varepsilon^m(t)-\theta^\infty(t))-\left\vert\frac{d\theta^\infty}{dt}\right\vert,
		\end{align*}
		for every $t\in [0,T]$, where again we have neglected the first term in the right-hand sides by the definition of $\theta_\varepsilon^M$ and $\theta^m$ and we have defined the time dependent coefficients
		\begin{align*}
		C_\varepsilon^M(t)&:=\left\{
		\begin{array}{ll}
		\frac{1}{\theta_M\bar \theta_M}, & \mbox{if}\quad \theta_\varepsilon^M(t)-\theta^\infty(t)\geq 0,\\
		\frac{1}{\theta_m\bar \theta_m}, & \mbox{if}\quad  \theta_\varepsilon^M(t)-\theta^\infty(t)< 0.
		\end{array}
		\right.\\
		C_\varepsilon^m(t)&:=\left\{
		\begin{array}{ll}
		\frac{1}{\theta_m\bar \theta_m}, & \mbox{if}\quad \theta_\varepsilon^m(t)-\theta^\infty(t)\geq 0,\\
		\frac{1}{\theta_M\bar \theta_M}, & \mbox{if}\quad  \theta_\varepsilon^m(t)-\theta^\infty(t)< 0.
		\end{array}
		\right.
		\end{align*}
		By using Gr\"{o}nwall's lemma for $\theta^M_\e$ and use a similar argument for $\theta_\varepsilon^m$, we obtain
		\begin{align*}
		\theta_\varepsilon^M(t)-\theta^\infty(t)&\leq (\theta_M+\bar \theta_M)e^{-\frac{1}{\varepsilon\theta_M\bar\theta_M}t}+\theta_M\bar \theta_M\left\Vert \frac{d\theta^\infty}{dt}\right\Vert_{L^\infty(0,T)}\varepsilon,\\
		\theta_\varepsilon^m(t)-\theta^\infty(t)&\geq- (\theta_M+\bar \theta_M)e^{-\frac{1}{\varepsilon\theta_M\bar\theta_M}t}-\theta_M\bar \theta_M\left\Vert \frac{d\theta^\infty}{dt}\right\Vert_{L^\infty(0,T)}\varepsilon,
		\end{align*}
		for each $t\in [0,T]$. Then, we recover the desired error estimate for appropriate coefficients $C$ and $c$ depending on $\theta_M$, $\bar\theta_M$ and $\Vert \theta^\infty\Vert_{W^{1,\infty}(0,T)}$, that is finite by the property \eqref{E-hypothesis-macro-species-gain-theta-infty} in Remark \ref{R-hypothesis-macro-species-gain}.
	\end{proof}

	Notice that relaxation is exponential with rate $\mathcal{O}(\varepsilon^{-1})$. Then, we can simplify the above estimates to achieve the following polynomial control for the concentration of internal variable support and its deviation from the background mean value $\theta^\infty$, in terms of $\e$, after a very short initial time layer.
	
	\begin{corollary}[Initial time layer]\label{cor_temp}
		Let $f_\e$ be the solution of \eqref{A-4} subject to initial data $f^0_\e$. Suppose that initial data verify \eqref{E-hypothesis-initial-data} and $(\bar \rho,\bar u,\bar e)$ verifies \eqref{E-hypothesis-macro-species}. Then, for any $\alpha>0$ and $0<\beta<1$, there exists a constant $C\in \mathbb{R}_+$, such that the internal variable support $D_\theta^\varepsilon(t)$ and the error $\eta_\varepsilon(t)$ in Lemma \ref{lemma_temp} verify
		\begin{equation}\label{temp_diam2}
		D_\theta^\varepsilon(t)\le C\e^\alpha,\quad  \eta_\varepsilon(t)\leq C\varepsilon^{\min\{\alpha,1\}},
		\end{equation}
		for every $t\in [\varepsilon^\beta,T]$. Here, $C$ depends on the exponents $\alpha,\beta$ along with the parameters $\theta_M^0$, $\theta_m^0$, $\bar \theta_M$, $\bar \theta_m$, $\Vert \zeta\Vert_{L^\infty}$ and $\Vert \theta^\infty\Vert_{W^{1,\infty}(0,T)}$.
	\end{corollary}
	
	\begin{proof}
		Taking $t\in [\varepsilon^\beta,T]$ and using Lemma \ref{lemma_temp}, we obtain
		\begin{align*}
		D_\theta^\varepsilon(t)&\leq (\theta_M^0-\theta_m^0) e^{-\frac{1}{\varepsilon^{1-\beta}\theta_M^2}},\\
		\eta_\varepsilon(t)&\leq Ce^{-\frac{c}{\varepsilon^{1-\beta}}}+C\varepsilon .
		\end{align*}
		Since all the above negative exponential functions decay faster than any polynomial $\varepsilon^\alpha$, we conclude the claimed result by appropriately modifying $C$ if necessary.
	\end{proof}
	
	To summarize, the previous results show that after a small initial time layer $\e_0:=\e^\beta$, the diameter of $\textup{supp}_\theta f_\e(t)$ is smaller than $C\e^\alpha$ and its distance to the background mean value $\theta^\infty(t)$ is smaller than $C\e^{\min\{\alpha,1\}}$. This justifies that,  as $\e\rightarrow 0$, $f_\e$ will concentrate around $\theta=\theta^\infty(t)$. That will be crucially used to derive the rigorous hydrodynamic limit.

	\subsection{Moment estimates}
	
	In this part, we prove a priory estimates for the moments of $f_\varepsilon$. We start by controlling the $k$-th order velocity moment of $f_\e$.

	\begin{lemma}[$k$-th order velocity moment]\label{L2.2}
		Let $f_\e$ be a solution of \eqref{A-4} subjects to initial data $f^0_\e$. Suppose that initial data verify \eqref{E-hypothesis-initial-data} and $(\bar \rho,\bar u,\bar e)$ verifies \eqref{E-hypothesis-macro-species}. Then, for any $k\ge1$, the following estimate 
		\begin{align}\label{B-0}
		\begin{aligned}
		&k\left(\frac{1}{\theta_M}-\Vert \phi\Vert_{L^\infty}\frac{D_\theta^\varepsilon(0)}{\theta_m^2}\right)\||v|^kf_\e\|_{L^1(0,T;L^1(\bbr^{2d}\times\bbr_+))}\\
		&+\frac{k}{2\theta_M}\int_0^T\int_{\bbr^{4d}\times\bbr^2_+}\phi(x-x_*)(|v|^{k-2}v-|v_*|^{k-2}v_*)\cdot(v-v_*)f_\varepsilon(t,z)f_\varepsilon(t,z_*)\,dz\,dz_*\,dt\\
		&\qquad \le\e\||v|^kf_\e(0)\|_{L^1(\bbr^{2d}\times\bbr_+)}+k\left\|\left(\left|\frac{u^\infty(t)}{\theta^\infty(t)}\right|+\Vert \nabla W\Vert_{L^\infty}\right)\||v|^{k-1} f_\e\|_{L^1(\bbr^{2d}\times\bbr_+)}\right\|_{L^1(0,T)},
		\end{aligned}
		\end{align}
		holds, where the constants $\theta_m$ and $\theta_M$ are given by formula \eqref{E-thetam-thetaM} in Lemma \ref{lemma_temp}.
	\end{lemma}
	
	\begin{proof}
		We multiply the kinetic equation \eqref{A-4} by $|v|^k$ and integrate over $\bbr^{2d}\times\bbr_+$ to obtain
		\begin{align}
		\begin{aligned}\label{B-4}
		&\e\frac{d}{dt}\int_{\bbr^{2d}\times\bbr_+} |v|^k f_\e\,dz\\
		&=k\int_{\mathbb{R}^{2d}\times \mathbb{R}_+}\vert v\vert^{k-2}v\cdot F_c[\bar\rho,\bar u,\bar e]\,f_\varepsilon\,dz+k\int_{\bbr^{2d}\times\bbr_+}|v|^{k-2}v\cdot H[f_\e]f_\e\,dz\\
		&\quad +k\int_{\bbr^{2d}\times\bbr_+}|v|^{k-2}v\cdot F[f_\varepsilon]\,f_\e\,dz=:\mathcal{I}_{11}+\mathcal{I}_{12}+\mathcal{I}_{13}.
		\end{aligned}
		\end{align}

		$\bullet$ {\sc Estimate of $\mathcal{I}_{11}$}: We first recall that
		\[F_c[\bar{\rho},\bar{u},\bar{e}](t,x,v,\theta)=\int_{\bbr^d}\frac{\bar{\rho}\bar{u}}{\bar{e}}\,dx -\frac{v}{\theta}\]
		Therefore, $\mathcal{I}_{11}$ is estimated as
		\begin{align*}
		\mathcal{I}_{11} &= k\int_{\bbr^{2d}\times\bbr_+}|v|^{k-2}v\cdot \left(\int_{\bbr^d}\frac{\bar{\rho}\bar{u}}{\bar{e}}\,dx-\frac{v}{\theta}\right)f_\e\,dz\\
		&\le k\left|\frac{u^\infty(t)}{\theta^\infty(t)}\right|\int_{\bbr^{2d}\times\bbr_+}|v|^{k-1}f_\e\,dz -\frac{k}{\theta_M}\int_{\bbr^{2d}\times\bbr_+}|v|^kf_\e\,dz,
		\end{align*}
		where $\theta_M$ is given in the diameter estimate in Lemma \ref{lemma_temp}.
		
		\medskip
		
		$\bullet$ {\sc Estimate of $\mathcal{I}_{12}$}: Similarly, by the regularity assumption \eqref{H-hypothesis-phi-zeta-W} we note that
		\[\vert H[f_\e](t,x)\vert = \left\vert\int_{\bbr^{2d}\times\bbr_+}\nabla W(x-x_*)f_\e(t,z_*)\,dz_*\right\vert\leq \Vert \nabla W\Vert_{L^\infty}.\]
		Therefore, we derive the following estimate for $\mathcal{I}_{12}$:
		\[\mathcal{I}_{12} \le k\|\nabla W\|_{L^\infty}\int_{\bbr^{2d}\times\bbr_+}|v|^{k-1}f_\e\,dz.\]

		$\bullet$ {\sc Estimate of $\mathcal{I}_{13}$}: Recall now that
		$$F[f_\varepsilon](t,x,v,\theta)=\int_{\mathbb{R}^{2d}\times \mathbb{R}_+}\phi(x-x_*)\left(\frac{v_*}{\theta_*}-\frac{v}{\theta}\right)f_\varepsilon(t,z_*)\,dz_*.$$
		By the change of variables $(x,v,\theta)\leftrightharpoons (x_*,v_*,\theta_*)$ in the integrals in $\mathcal{I}_{13}$, we obtain the usual symmetrized form
		\begin{align*}
		\mathcal{I}_{13}&=\frac{k}{2}\int_{\bbr^{4d}\times\bbr^2_+}\phi(x-x_*)(|v|^{k-2}v-|v_*|^{k-2}v_*)\cdot\left(\frac{v_*}{\theta_*}-\frac{v}{\theta}\right)f_\varepsilon(t,z)f_\varepsilon(t,z_*)\,dz\,dz_*\\
		&=\frac{k}{2}\int_{\bbr^{4d}\times\bbr^2_+}\phi(x-x_*)\frac{1}{\theta}(|v|^{k-2}v-|v_*|^{k-2}v_*)\cdot(v_*-v)\,f_\varepsilon(t,z)f_\varepsilon(t,z_*)\,dz\,dz_*\\
		&\quad+\frac{k}{2}\int_{\bbr^{4d}\times\bbr^2_+}\phi(x-x_*)(|v|^{k-2}v-|v_*|^{k-2}v_*)\cdot v_*\frac{\theta-\theta_*}{\theta_*\theta}f_\varepsilon(t,z)f_\varepsilon(t,z_*)\,dz\,dz_*\\
		&=:\mathcal{I}_{131}+\mathcal{I}_{132}.
		\end{align*}
		
		\medskip
		
		$\diamond$ {\sc Estimate of $\mathcal{I}_{131}$}: It is straightforward to see that for $k\ge1$,
		\begin{align*}
		(|v|^{k-2}v-|v_*|^{k-2}v_*)\cdot(v-v_*) &= |v|^k+|v_*|^k-(|v|^{k-2}+|v_*|^{k-2})\,v\cdot v_*\\
		&\geq |v|^k+|v_*|^k-(|v|^{k-2}+|v_*|^{k-2})|v||v_*|\\
		&=(|v|^{k-1}-|v_*|^{k-1})(|v|-|v_*|)\ge 0.
		\end{align*}
		Therefore, we directly obtain that $\mathcal{I}_{131}\le0$. We move it to the left-hand side of \eqref{B-4}.
		
		\medskip
		
		$\diamond$ {\sc Estimate of $\mathcal{I}_{132}$}: To estimate $\mathcal{I}_{132}$, we use the contraction of the internal variable support in Lemma \ref{lemma_temp} to observe that
		\begin{align*}
		|\mathcal{I}_{132}|&=\frac{k}{2}\left|\int_{\bbr^{4d}\times\bbr^2_+}\phi(x-x_*)\left(|v|^{k-2}v-|v_*|^{k-2}v_*\right)\cdot v_*\frac{\theta-\theta_*}{\theta_*\theta}f_\varepsilon(t,z)f_\varepsilon(t,z_*)\,dz\,dz_*\right|\\
		&\le \frac{k\|\phi\|_{L^\infty}D_\theta^\varepsilon(t)}{2\theta_m^2} \int_{\bbr^{4d}\times\bbr^2_+}\left(|v|^{k-1}|v_*|+|v_*|^k\right)f_\varepsilon(t,z)f_\varepsilon(t,z_*)\,dz\,dz_*\\
		&\le\frac{k\|\phi\|_{L^\infty}D_\theta^\varepsilon(t)}{2\theta_m^2} \int_{\bbr^{4d}\times\bbr^2_+}\left(\frac{k-1}{k}|v|^k+\frac{1}{k}|v_*|^k+|v_*|^k\right)f_\varepsilon(t,z)f_\varepsilon(t,z_*)\,dz\,dz_*\\
		&\le\frac{k\|\phi\|_{L^\infty}D_\theta^\varepsilon(t)}{\theta_m^2}\int_{\bbr^{2d}\times\bbr_+}|v|^kf\,dz.
		\end{align*}
		Combining all the preceding estimates for $\mathcal{I}_{11}$, $\mathcal{I}_{12}$ and $\mathcal{I}_{13}$ into \eqref{B-4}, we derive
		\begin{align}
		\begin{aligned}\label{B-1}
		\e\frac{d}{dt}\int_{\bbr^{2d}\times\bbr_+}|v|^kf_\e \,dz-\mathcal{I}_{131}&\le k\left(\left|\frac{u^\infty(t)}{\theta^\infty(t)}\right|+\Vert \nabla W\Vert_{L^\infty}\right)\int_{\bbr^{2d}\times\bbr_+}|v|^{k-1}f_\e\,dz \\&\quad-k\left(\frac{1}{\theta_M}-\frac{\|\phi\|_{L^\infty}D_\theta^\varepsilon(0)}{\theta_m^2}\right)\int_{\bbr^{2d}\times\bbr_+}|v|^kf_\e \,dz.
		\end{aligned}
		\end{align}
		Finally, we integrate \eqref{B-1} from $t=0$ to $t=T$ to achieve the desired estimate.
	\end{proof}
	
	Then, under the assumption on the initial diameter of internal variable so that the coefficient $\frac{1}{\theta_M}-\frac{\|\phi\|_{L^\infty}D_\theta^\varepsilon(0)}{\theta_m^2}$ is positive, we find uniform-in-$\e$ bounds for the first and second velocity moments of $f_\e$, along with the first-order position moment.
	
	\begin{corollary}[Velocity and position moments]\label{C2.3}
		Let $f_\e$ be the solution of \eqref{A-4} subject to initial data $f^0_\e$. Assume that the initial data verify \eqref{E-hypothesis-initial-data}, $(\bar \rho,\bar u,\bar e)$ verifies \eqref{E-hypothesis-macro-species} and the  parameters fulfill the compatibility condition \eqref{E-hypothesis-thetaM-thetam}. 
		Then,
		\begin{align}\label{E-moment-estimates}
		\begin{aligned}
		\Vert \vert x\vert f_\varepsilon\Vert_{L^\infty(0,T;L^1(\mathbb{R}^{2d}\times \mathbb{R}_+))}&\leq M_0+T^{1/2}\Vert \vert v\vert f_\varepsilon\Vert_{L^2(0,T;L^1(\bbr^{2d}\times\bbr_+))},\\ 
		\||v|f_\e\|_{L^2(0,T;L^1(\bbr^{2d}\times\bbr_+))}& \le\left(\||v|^2f_\e\|_{L^1(0,T;L^1(\bbr^{2d}\times\bbr_+))}\right)^{1/2},\\
		\||v|^2f_\e\|_{L^1(0,T;L^1(\bbr^{2d}\times\bbr_+))}&\le C\left(\e E_0+CG_0^2\right),
		\end{aligned}
		\end{align}
		for every $\varepsilon>0$, where $G_0:=F_0+T^{1/2}\Vert \nabla W\Vert_{L^\infty}$. In addition, the next estimate holds
		\begin{equation}\label{E-dissipation-estimate}
		\frac{1}{\theta_M}\int_0^T\int_{\mathbb{R}^{4d}\times \mathbb{R}_+^2}\phi(x-x_*)\vert v-v_*\vert^2\,f_\varepsilon(t,z)f_\varepsilon(t,z_*)\,dz\,dz_*\,dt\leq \varepsilon E_0+CG_0^2.
		\end{equation}
	\end{corollary}

	\begin{proof}
		Multiplying \eqref{A-4} by $\vert x\vert$ and integrating by parts we achieve the inequality
		$$\frac{d}{dt}\int_{\mathbb{R}^{2d}\times \mathbb{R}_+}\vert x\vert\,f_\varepsilon(t,z)\,dz\leq \int_{\mathbb{R}^{2d}\times \mathbb{R}_+} \vert v\vert\,f_\varepsilon(t,z)\,dz,$$
		for each $t\in [0,T]$. Integrating with respect to time from $0$ to $t$ we obtain
		$$\Vert \vert x\vert f_\varepsilon(t)\Vert_{L^1(\mathbb{R}^{2d}\times \mathbb{R}_+)}\leq \Vert \vert x\vert f_\varepsilon^0\Vert_{L^1(\mathbb{R}^{2d}\times\bbr_+)}+ \int_0^T\Vert \vert v\vert f_\varepsilon(t)\Vert_{L^1(\mathbb{R}^{2d}\times \mathbb{R}_+)}\,dt,$$
		for every $t\in [0,T]$. Hence, the assumption \eqref{E-hypothesis-initial-data} on the initial data and the Cauchy-Schwartz inequality yields the first estimate in \eqref{E-moment-estimates}. Again, by the Cauchy-Schwartz inequality and the conservation of mass, we readily obtain the second estimate in \eqref{E-moment-estimates}. We finally focus on the third estimate in \eqref{E-moment-estimates} and \eqref{E-dissipation-estimate}. To such an end, we apply \eqref{B-0} in Lemma \ref{lemma_temp} with $k=2$ and recover
		\begin{align}
		\begin{aligned}\label{B-6}
		2&\left(\frac{1}{\theta_M}-\frac{\|\phi\|_{L^\infty}D_\theta^\varepsilon(0)}{\theta_m^2}\right)\||v|^2f_\e\|_{L^1(0,T;L^1(\bbr^{2d}\times\bbr_+))}\\
		&+\frac{1}{\theta_M}\int_0^T\int_{\bbr^{4d}\times\bbr^2_+}\phi(x-x_*)\vert v-v_*\vert^2\,f_\varepsilon(t,z)f_\varepsilon(t,z_*)\,dz\,dz_*\,dt\\
		&\quad \le \e\||v|^2f_\e^0\|_{L^1(\bbr^{2d}\times\bbr_+)}+2\left\|\left(\left|\frac{u^\infty(t)}{\theta^\infty(t)}\right|+\Vert \nabla W\Vert_{L^\infty}\right)\||v| f_\e\|_{L^1(\bbr^{2d}\times\bbr_+)}\right\|_{L^1(0,T)}\\
		&\quad \le\e\||v|^2f_\e^0\|_{L^1(\bbr^{2d}\times\bbr_+)}+2\||v|f_\e\|_{L^2(0,T;L^1(\bbr^{2d}\times\bbr_+))}\left\|\frac{u^\infty}{\theta^\infty}+\Vert \nabla W\Vert_{L^\infty}\right\|_{L^2(0,T)}\\
		&\quad \leq \varepsilon E_0+2\left(F_0+T^{1/2}+\Vert \nabla W\Vert_{L^\infty}\right)\Vert \vert v\vert^2 f_\varepsilon\Vert_{L^1(0,T;L^1(\mathbb{R}^{2d}\times \mathbb{R}_+))}^{1/2},
		\end{aligned}
		\end{align}
		where in the last line we have used the assumptions \eqref{E-hypothesis-initial-data} on initial data and \eqref{E-hypothesis-macro-species} for the macroscopic species (see also Remark \ref{R-hypothesis-macro-species-gain}), along with the second estimate in \eqref{E-moment-estimates}.
		We now define the parameter
		$$\delta := \frac{1}{\theta_M}-\frac{\|\phi\|_{L^\infty}D_\theta^\varepsilon(0)}{\theta_m^2}>0,$$
		which is positive thanks to the assumption \eqref{E-hypothesis-thetaM-thetam}. Using Young's inequality in the second term of the last line in \eqref{B-6} we obtain
		$$
		2G_0\Vert \vert v\vert^2 f_\varepsilon\Vert_{L^1(0,T;L^1(\mathbb{R}^{2d}\times \mathbb{R}_+))}^{1/2}\leq \frac{G_0^2}{\delta}+\delta\Vert \vert v\vert^2 f_\varepsilon\Vert_{L^1(0,T;L^1(\mathbb{R}^{2d}\times \mathbb{R}_+))}.
		$$
		We substitute the above estimate on \eqref{B-6} to derive
		\begin{multline}\label{E-velocity-moment-regular-phi}
		\left(\frac{2}{\theta_M}-\frac{2\|\phi\|_{L^\infty}D_\theta^\varepsilon(0)}{\theta_m^2}-\delta\right)\||v|^2f_\e\|_{L^1(0,T;L^1(\bbr^{2d}\times\bbr_+))}\\
		+\frac{1}{\theta_M}\int_0^T\int_{\bbr^{4d}\times\bbr^2_+}\phi(x-x_*)\vert v-v_*\vert^2\,f_\varepsilon(t,z)f_\varepsilon(t,z_*)\,dz\,dz_*\,dt \leq \varepsilon E_0+\frac{G_0^2}{\delta}.
		\end{multline}
		To conclude, we choose $C:=\delta^{-1}$ and we obtain the desired estimates.
	\end{proof}

	The preceding estimates for the internal variable support in Lemma \ref{lemma_temp} and the position and velocity moments in Lemma \ref{L2.2} will conform to the key tool in order to derive compactness of the system as $\varepsilon\rightarrow 0$. That will be the first step with regards to the hydrodynamic limit of system \eqref{A-4} and will be addressed in the next section.

	\subsection{Hydrodynamic limit of \eqref{A-4}}
	Our goal here is to present the rigorous hydrodynamic limit of \eqref{A-4} whose explicit statement takes the following form:
	
	\begin{theorem}\label{T2.1}
		Let $f_\e$ be a solution to the equation \eqref{A-4} subject to the initial data $f_\e^0$. Assume that the initial data verify \eqref{E-hypothesis-initial-data}, $(\bar \rho,\bar u,\bar e)$ verifies \eqref{E-hypothesis-macro-species} and the  parameters fulfill the compatibility condition \eqref{E-hypothesis-thetaM-thetam}. Then,
		\begin{align*}
		&\rho_\varepsilon\rightarrow\rho,\quad \mbox{in }C([0,T],\mathcal{M}(\mathbb{R}^{d})-\mbox{narrow}),\\
		&j_\varepsilon\overset{*}{\rightharpoonup} j,\quad \mbox{in }L^2_w(0,T;\mathcal{M}(\mathbb{R}^d)^d),
		\end{align*}
		when $\varepsilon\rightarrow 0$, for some probability measure $\rho$, some finite Radon measure $j$ and some subsequence of $\{\rho_\varepsilon\}_{\varepsilon>0}$ and $\{j_\varepsilon\}_{\varepsilon>0}$ that we denote in the same way. In addition, $(\rho,j)$ solves the following problem
		$$
		\begin{array}{ll}
		\displaystyle \partial_t \rho+\nabla\cdot j =0, & (t,x)\in [0,T)\times \mathbb{R}^d,\\
		\displaystyle j-u^\infty(t)\rho+\theta^\infty(t)(\nabla W*\rho)\rho=\rho (\phi*j)-j(\phi*\rho), & (t,x)\in (0,T)\times \mathbb{R}^d,\\
		\rho(t=0)=\rho^0, & x\in \mathbb{R}^d,
		\end{array}
		$$
		in distributional sense, where $\theta^\infty(t)$ and $u^\infty(t)$ are the background mean value of the internal variable and velocity, see \eqref{E-background-mean-theta} and \eqref{E-background-mean-velocity}.
	\end{theorem}
	
	We refer to Appendix \ref{Appendix-LB} for a summarized presentation of weak-* Lebesgue-Bochner spaces $L^p_w(0,T;X^*)$ for a Banach space $X^*$, their comparison with classical Lebesgue-Bochner spaces
	$L^p(0,T;X^*)$ along with their duality properties.
	
	\begin{remark}[Weak formulations]
		Notice that the first and second equation of the above limiting system are verified in distributional sense in $[0,T)\times \mathbb{R}^d$ and $(0,T)\times \mathbb{R}^d$ respectively. That is, the following equations fulfill
		\begin{align*}
		&\int_0^T\int_{\mathbb{R}^d} \partial_t\varphi\rho(t,dx)\,dt+\int_0^T\int_{\mathbb{R}^d}\nabla\varphi\cdot j(t,dx)\,dt=-\int_{\mathbb{R}^d}\varphi(0,\cdot)\rho^0(dx),\\
		&\int_0^T\int_{\mathbb{R}^d}\psi(j(t,dx)-u^\infty(t)\rho(t,dx)+\theta^\infty(t)(\nabla W*\rho)\rho(t,dx))\,dt\\
		&\qquad=\int_0^T\int_{\mathbb{R}^d}\psi(\phi*j)\rho(t,dx)\,dt-\int_0^T\int_{\mathbb{R}^d}\psi(\phi*\rho)j(t,dx)\,dt,
		\end{align*}
		for any $\varphi\in C^1_c([0,T)\times \mathbb{R}^d)$ and $\psi\in C^1_c((0,T)\times \mathbb{R}^d)$. Whilst the former involves generic test functions $\varphi$, in the latter test functions $\psi$ must vanish at $t=0$. Moreover, notice that the regularity of the influence function $\phi$ guarantees that $\phi*\rho$ and $\phi*j$ belong to $C_b(\mathbb{R}^d)$, so that the nonlinear terms in the right-hand side are well defined.
	\end{remark}
	
	Now, we derive the proof of Theorem \ref{T2.1}.
	
	\subsection*{Proof of compactness}
	From the conservation of mass, Lemma \ref{lemma_temp} and Corollary \ref{C2.3} we obtain the following uniform estimates for moments
	\begin{align}\label{E-a-priori-estimates-fast-regular}
	\begin{aligned}
	\|\rho_\e\|_{L^\infty(0,T; L^1(\bbr^d))}&=1,\\
	\Vert \vert x\vert\rho_\varepsilon\Vert_{L^\infty(0,T; L^1(\mathbb{R}^d))}&\leq M_0+(CT(\varepsilon E_0+CG_0^2))^{1/2},\\
	\|j_\e\|_{L^2(0,T;L^1(\bbr^d))}&\le \left(C\left(\e E_0+CG_0^2\right)\right)^{1/2},\\
	\Vert \mathcal{S}_\varepsilon^v\Vert_{L^1(0,T;L^1(\mathbb{R}^d))}&\leq C(\varepsilon E_0+CG_0^2),\\
	\Vert \mathcal{S}_\varepsilon^\theta\Vert_{L^2(0,T;L^1(\mathbb{R}^d))}&\leq \theta_M\left(C\left(\e E_0+CG_0^2\right)\right)^{1/2},
	\end{aligned}
	\end{align}
	for every $\varepsilon>0$. As a consequence of the first and third inequalities in \eqref{E-a-priori-estimates-fast-regular}, we obtain that $\{\rho_\varepsilon\}_{\varepsilon>0}$ is bounded in $L^\infty(0,T;L^1(\mathbb{R}^d))$ and $\{j_\varepsilon\}_{\varepsilon>0}$ is bounded in $L^2(0,T;L^1(\mathbb{R}^d))$. In addition, by the representation Theorem \ref{Appendix-LB-Riesz-representation-no-RNP} in Appendix \ref{Appendix-LB} we obtain that
	\begin{align*}
	L^\infty(0,T;L^1(\mathbb{R}^d))&\subseteq L^\infty_w(0,T;\mathcal{M}(\mathbb{R}^d))\equiv L^1(0,T;C_0(\mathbb{R}^d))^*,\\
	L^2(0,T;L^1(\mathbb{R}^d))&\subseteq L^2_w(0,T;\mathcal{M}(\mathbb{R}^d))\equiv L^2(0,T;C_0(\mathbb{R}^d))^*.
	\end{align*}
	Hence, by Alaoglu--Bourbaki's theorem
	\begin{align}\label{E-compactness-fast-regular}
	\begin{aligned}
	&\rho_\e\xrightharpoonup{*}\rho,\quad \mbox{in}\quad L^\infty_w(0,T;\mathcal{M}(\bbr^d)),\\
	&j_\e\xrightharpoonup{*} j,\quad \mbox{in}\quad L^2_w(0,T;\mathcal{M}(\bbr^d)^d),
	\end{aligned}
	\end{align}
	as $\varepsilon\rightarrow 0$, modulo subsequence, for some $(\rho,j)\in L^\infty_w(0,T;\mathcal{M}(\mathbb{R}^d))\times L^2_w(0,T;\mathcal{M}(\mathbb{R}^d)^d)$. In fact, an argument like in \cite[Theorem 3.8]{PS17} allows proving a gain of time regularity of $\rho_\varepsilon$. Combining it with \eqref{E-a-priori-estimates-fast-regular} we indeed obtain that
	\begin{equation}\label{E-compactness-fast-regular-gain}
	\rho_\varepsilon\rightarrow \rho\quad \mbox{in }C([0,T],\mathcal{M}(\mathbb{R}^d)-\mbox{narrow}),
	\end{equation}
	that is,
	$$\lim_{\varepsilon\rightarrow 0}\sup_{t\in [0,T]}\left\vert \int_{\mathbb{R}^d}\varphi(\rho_\varepsilon(t,dx)-\rho(t,dx))\right\vert=0,$$
	for any test function $\varphi\in C_b(\mathbb{R}^d)$.
	
	\subsection*{Proof of the limit}
	
	Writing the first equation in \eqref{E-hierarchy-moments-equations} in weak form, we achieve
	$$\int_0^T\int_{\mathbb{R}^d}\partial_t\varphi \rho_\varepsilon\,dx\,dt+\int_0^T\int_{\mathbb{R}^d} \nabla\varphi\cdot j_\varepsilon\,dx\,dt=-\int_{\mathbb{R}^d}\varphi(0,\cdot)\rho_\varepsilon^0\,dx,$$
	for any $\varphi\in C^1_c([0,T)\times \mathbb{R}^d)$. By \eqref{E-compactness-fast-regular} we readily identify the limit of the terms in the left-hand side. In addition, restricting \eqref{E-compactness-fast-regular-gain} to $t=0$, we can also identify the initial datum $\rho(t=0)=\lim_{\varepsilon\rightarrow 0}\rho_\varepsilon^0=\rho^0$ in order to pass to the limit in the right-hand side. Putting everything together yields 
	$$\int_0^T\int_{\mathbb{R}^d}\partial_t\varphi \rho(t,dx)\,dt+\int_0^T\int_{\mathbb{R}^d} \nabla\varphi\cdot j(t,dx)\,dt=-\int_{\mathbb{R}^d}\varphi(0,\cdot)\rho^0(dx).$$
	We focus on the remaining equations of the hierarchy \eqref{E-hierarchy-moments-equations}. In weak formulation, they read
	\begin{align}
	\varepsilon\int_0^T\int_{\mathbb{R}^d}&\partial_t\psi \,j_\varepsilon\,dx\,dt+\varepsilon\int_0^T\int_{\mathbb{R}^d}\mathcal{S}_\varepsilon^v\,\nabla\psi\,dx\,dt\nonumber\\
	&=\int_0^T\int_{\mathbb{R}^d}\psi(\phi*A_\varepsilon)\rho_\varepsilon\,dx\,dt-\int_0^T\int_{\mathbb{R}^d}\psi(\phi*\rho_\varepsilon)A_\varepsilon\,dx\,dt\label{E-velocity-weak-form-fast-regular}\\
	&\quad -\int_0^T\int_{\bbr^d}\psi(\nabla W*\rho_\e)\rho_\e\,dx\,dt+ \int_0^T\int_{\mathbb{R}^d}\psi\frac{u^\infty(t)}{\theta^\infty(t)}\rho_\varepsilon\,dx\,dt-\int_0^T\int_{\mathbb{R}^d}\psi A_\varepsilon\,dx\,dt,\nonumber\\
	\varepsilon\int_0^T\int_{\mathbb{R}^d}&\partial_t\psi \,h_\varepsilon\,dx\,dt+\varepsilon\int_0^T\int_{\mathbb{R}^d}\mathcal{S}_\varepsilon^\theta\nabla\psi\,dx\,dt\nonumber\\
	&=\int_0^T\int_{\mathbb{R}^d}\psi(\zeta*\rho_\varepsilon)B_\varepsilon\,dx\,dt-\int_0^T\int_{\mathbb{R}^d}\psi(\zeta*B_\varepsilon)\rho_\varepsilon\,dx\,dt\label{E-internal-variable-weak-form-fast-regular}\\
	&\quad +\int_0^T\int_{\mathbb{R}^d}\psi B_\varepsilon\,dx\,dt-\int_0^T\int_{\mathbb{R}^d}\psi\frac{\rho_\varepsilon}{\theta^\infty(t)}\,dx\,dt,\nonumber
	\end{align}
	for any $\psi\in C^1_c((0,T)\times \mathbb{R}^d)$. Notice that the inertial terms in the left-hand sides of \eqref{E-velocity-weak-form-fast-regular} and \eqref{E-internal-variable-weak-form-fast-regular} vanish as $\varepsilon\rightarrow 0$ thanks to \eqref{E-a-priori-estimates-fast-regular}. However, the right-hand sides yield to a not closed system as they depend on the nonlinear moments $A_\varepsilon$, $B_\varepsilon$. Then, the last step will be to identify their limits in terms of $\rho$ and $j$. Indeed, recall that the internal variable support of $f_\varepsilon$ shrinks rapidly and concentrates on the relaxation value $\theta^\infty(t)$ after very short initial layer $t=\e_0$, see Corollary \ref{cor_temp}. This suggests that
	$$A_\varepsilon\approx\frac{j}{\theta^\infty(t)}\quad \mbox{and}\quad B_\varepsilon\approx \frac{\rho}{\theta^\infty(t)},$$
	after a short initial time layer $t>\varepsilon^\beta$. In fact, we have the following result.

	\begin{lemma}[Identifying nonlinear moments]\label{L2.3}
		Let $f_\e$ be a solution to the equation \eqref{A-4} subject to the initial data $f_\e^0$. Suppose that initial data verify \eqref{E-hypothesis-initial-data}, $(\bar \rho,\bar u,\bar e)$ verifies \eqref{E-hypothesis-macro-species} and parameters fulfill the compatibility condition \eqref{E-hypothesis-thetaM-thetam}. Then, for any $\e_0>0$ we obtain
		\begin{align*}
		&A_\e\xrightharpoonup{*} \frac{j}{\theta^\infty(t)},\quad \mbox{in}\quad L^2_w({\e_0},T;\mathcal{M}(\bbr^d)^d),\\
		&B_\e \rightarrow \frac{\rho}{\theta^\infty(t)},\quad \mbox{in}\quad C([\varepsilon_0,T],\mathcal{M}(\bbr^d)-\mbox{narrow}),
		\end{align*}
		when $\varepsilon\rightarrow 0$.
	\end{lemma}
	
	\begin{proof}
		Fix any $\varepsilon_0>0$ and let $\psi\in L^2(\varepsilon_0,T;C_0(\mathbb{R}^d))$ be any test functions. Then, 
		\begin{align*}
		&\left|\int_{\e_0}^T \int_{\bbr^d} \psi(t,x) \left(A_\e(t,dx)-\frac{j(t,dx)}{\theta^\infty(t)}\right)\,dt\right|\\
		&\qquad \leq \left\vert\int_{\e_0}^T\int_{\mathbb{R}^{2d}\times \mathbb{R}_+}\psi(t,x)\frac{\theta^\infty(t)-\theta}{\theta^\infty(t)\theta}v f_\varepsilon\,dz\,dt\right\vert\\
		&\qquad \qquad \qquad +\left|\int_{\e_0}^T\int_{\bbr^d}\psi(t,x)\frac{1}{\theta^\infty(t)}(j_\e(t,dx)-j(t,dx))\,dt\right|\\
		&\qquad \le \frac{\|\eta_\e\|_{L^\infty(\e_0,T)}}{\bar\theta_m\theta_m}\Vert \psi\Vert_{L^2(\varepsilon_0,T;C_0(\mathbb{R}^d))}\|\vert v\vert f_\varepsilon\|_{L^2(0,T;L^1(\bbr^{2d}\times \bbr_+))}\\
		&\qquad \qquad\qquad  +\left|\int_{\e_0}^T\int_{\bbr^d}\frac{\psi(t,x)}{\theta^\infty(t)}(j_\e(t,dx)-j(t,dx))\,dt\right|,
		\end{align*} 
		where in the last step we have used the uniform lower bound on $\theta^\infty(t)$ and the concentration estimate of $\supp_\theta f_\varepsilon$ in Lemma \ref{lemma_temp}. Since $\|\vert v\vert f_\e\|_{L^2(0,T;L^1(\bbr^{2d}\times \bbr_+))}$ is uniformly bounded by Corollary \ref{C2.3} and $\|\eta_\varepsilon\|_{L^\infty(\e_0,T)}\le C\e^{\min\{\alpha,1\}}$ for $\varepsilon\leq \varepsilon_0^{1/\beta}$ by Corollary \ref{cor_temp}, then the first term vanishes as $\e\to0$. Likewise, since $j_\e\xrightharpoonup{*} j$ in $L^2_w(0,T;\mathcal{M}(\bbr^d)^d)$ by \eqref{E-compactness-fast-regular} and $\frac{\psi}{\theta^\infty}\in L^2(\varepsilon_0,T;C_0(\mathbb{R}^d))$, the second term also vanishes as $\e\to0$. In conclusion, we have
		\[\lim_{\e\to0}\int_{\e_0}^T \int_{\bbr^d} \psi(t,x) \left(A_\e(t,dx)-\frac{j(t,dx)}{\theta^\infty(t)}\right)\,dt=0,\]
		which implies the desired weak convergence. For the convergence of $B_\e$, using that $\rho_\varepsilon\overset{*}{\rightharpoonup} \rho$ in $L^\infty_w(0,T;\mathcal{M}(\mathbb{R}^d))$ by \eqref{E-compactness-fast-regular}, a similar argument shows that
		$$B_\varepsilon\overset{*}{\rightharpoonup} \frac{\rho}{\theta^\infty(t)},\quad \mbox{in}\quad L^\infty_w(\varepsilon_0,T;\mathcal{M}(\mathbb{R}^d)).$$
		However, we can improve the convergence as follows. Take any $\varphi\in C_b(\mathbb{R}^d)$ and compute
		\begin{align*}
		&\sup_{t\in [\varepsilon_0,T]}\left\vert\int_{\mathbb{R}^d}\varphi\,\left(B_\varepsilon(t,dx)-\frac{\rho(t,dx)}{\theta^\infty}\right)\right\vert,\\
		&\leq \sup_{t\in [\varepsilon_0,T]}\left\vert \int_{\mathbb{R}^{2d}\times \mathbb{R}_+}\varphi(x)\left(\frac{1}{\theta}-\frac{1}{\theta^\infty(t)}\right)f_\varepsilon(t,z)\,dz \right\vert+\sup_{t\in [\varepsilon_0,T]}\left\vert\int_{\mathbb{R}^d} \frac{\varphi(x)}{\theta^\infty(t)}(\rho_\varepsilon(t,dx)-\rho(t,dx))\right\vert\\
		&\leq \frac{\Vert \eta_\varepsilon\Vert_{L^\infty(\varepsilon_0,T)}}{\bar \theta_m\theta_m}\Vert \varphi\Vert_{C_b(\mathbb{R}^d)}+\frac{1}{\bar\theta_m}\sup_{t\in [\varepsilon_0,T]}\left\vert\int_{\mathbb{R}^d}\varphi(x)\,(\rho_\varepsilon(t,dx)-\rho(t,dx))\right\vert.
		\end{align*}
		By the same argument, the first term vanishes,  as $\varepsilon\rightarrow 0$. Since $\rho_\varepsilon\rightarrow \rho$ in $C([0,T],\mathcal{M}(\mathbb{R}^d)-\mbox{narrow})$ by \eqref{E-compactness-fast-regular-gain}, the second term also vanishes, as $\varepsilon\rightarrow 0$, and we end the proof.
	\end{proof}
	
	The convergence of the nonlinear moments $A_\varepsilon$ also implies the convergence of the nonlinear term $\rho_\varepsilon\otimes A_\varepsilon$.

	\begin{lemma}[Weak convergence of tensor product]\label{L2.4}
		Let $f_\e$ be a solution to the equation \eqref{A-4} subject to the initial data $f_\e^0$. Assume that the initial data verify \eqref{E-hypothesis-initial-data}, $(\bar \rho,\bar u,\bar e)$ verifies \eqref{E-hypothesis-macro-species} and the  parameters fulfill the compatibility condition \eqref{E-hypothesis-thetaM-thetam}. Then,  we obtain
		\begin{align*}
		\begin{aligned}
		&\rho_\e\otimes A_\e \xrightharpoonup{*} \frac{\rho\otimes j}{\theta^\infty(t)}, & & \mbox{in}\quad L^2_w(\e_0,T;\mathcal{M}(\bbr^{2d})),\\ 
		&\rho_\e\otimes\rho_\e\xrightharpoonup{*}\rho\otimes\rho, & &\mbox{in}\quad C([0,T],\mathcal{M}(\bbr^{2d})-\mbox{narrow}),
		\end{aligned}
		\end{align*}
		for any $\e_0>0$, as $\varepsilon\rightarrow 0$.
	\end{lemma}
	
	\begin{proof}
	
	$\bullet$ {\sc Step 1}. Recall that by Appendix \ref{Appendix-LB} we obtain the following representation
		$$L^2_w(\varepsilon_0,T;\mathcal{M}(\mathbb{R}^{2d}))=L^2_w(\varepsilon_0,T;C_0(\mathbb{R}^{2d}))^*.$$ In order to derive the first convergence we set any test function $\varphi\in L^2(\varepsilon_0,T;C_0(\mathbb{R}^{2d}))$. By a density argument, it suffices to deal with  the case when $\varphi$ takes the following form:
		$$\varphi(t,x,x_*)=\chi(t)\sigma(x)\psi(x_*),$$
		for each $t\in (\varepsilon_0,T)$ and $x,x_*\in \mathbb{R}^d$, where $\chi\in L^2(\varepsilon_0,T)$ and $\sigma,\psi\in C_0(\bbr^d)$.
		\begin{align*}
		\mathcal{I}_\varepsilon&:=\int_{\e_0}^T\int_{\bbr^{2d}}\varphi(t,x,x_*)\left((\rho_\e(t,dx))(A_\e(t,dx_*))-\rho(t,dx)\frac{j(t,dx_*)}{\theta^\infty(t)}\right)\\
		&=\int_{\e_0}^T \chi(t)\left(\int_{\bbr^d}\sigma(x)(\rho_\e(t,dx)-\rho(t,dx))\right)\left(\int_{\bbr^d}\psi(x_*)\left(A_\e(t,dx_*)-\frac{j(t,dx_*)}{\theta^\infty(t)}\right)\right)\,dt\\
		&\quad+\int_{\e_0}^T \chi(t)\left(\int_{\bbr^d}\sigma(x)(\rho_\e(t,dx)-\rho(t,dx))\right)\left(\int_{\bbr^d}\psi(y)\frac{j(t,dy)}{\theta^\infty(t)}\right)\,dt\\
		&\quad+\int_{\e_0}^T \chi(t)\left(\int_{\bbr^d}\sigma(x)\rho(t,dx)\right)\left(\int_{\bbr^d}\psi(x_*)\left(A_\e(t,dx_*)-\frac{j(t,dx_*)}{\theta^\infty(t)}\right)\right)\,dt\\
		&=: \mathcal{I}_{\varepsilon,1}+\mathcal{I}_{\varepsilon,2}+\mathcal{I}_{\varepsilon,3}.
		\end{align*}
		To show the desired convergence we need to show that $\mathcal{I}_{\varepsilon,i}\rightarrow 0$ as $\varepsilon\rightarrow 0$. We will restrict to the first term $\mathcal{I}_{\varepsilon,1}$ since the reasoning in the remaining two terms is similar. Since $\rho_\e\rightarrow\rho$ in $C([0,T],\mathcal{M}(\mathbb{R}^d)-\mbox{narrow})$ by \eqref{E-compactness-fast-regular-gain} and $\chi\in L^2(\varepsilon_0,T)$, we have that
		\[\chi(t)\int_{\bbr^d}\sigma(x)(\rho_\e(t,dx)-\rho(t,dx))\to 0,\quad\mbox{in}\quad L^2(\e_0,T).\]
		Consequently, 
		\[\chi(t)\left(\int_{\bbr^d}\sigma(x)(\rho_\e(t,dx)-\rho(t,dx))\right)\psi(x_*)\to 0,\quad\mbox{in}\quad  L^2(\e_0,T;C_0(\bbr^d)).\]
		Moreover, since $A_\e\xrightharpoonup{*}\frac{j}{\theta^\infty}$ in $L^2_w(\e_0,T;\mathcal{M}(\bbr^d))$ by Lemma \ref{L2.3}, we have that $\mathcal{I}_{\varepsilon,1}\to0$ as $\e\to 0$. 
		
        \medskip

        $\bullet$ {\sc Step 2}. The convergence of $\rho_\e\otimes\rho_\e$ in $L^\infty_w(0,T;\mathcal{M}(\mathbb{R}^d))$ can be attained by applying a similar argument. However, its convergence in $C([0,T],\mathcal{M}(\mathbb{R}^d)-\mbox{narrow})$ requires a more delicate treatment. Specifically, set any $\delta>0$ and $\varphi\in C_0(\mathbb{R}^{2d})$, and apply the Stone-Weierstrass theorem to find functions $\sigma_1,\ldots,\sigma_n$ and $\psi_1,\ldots,\psi_n$ in $C_0(\mathbb{R}^d)$ so that
        \begin{equation}\label{E-Stone-Weierstrass}
        \left\vert \varphi(x,x_*)-\sum_{k=1}^n\sigma_k(x)\psi_k(x_*)\right\vert\leq \delta,
        \end{equation}
        for each $x,x_*\in \mathbb{R}^d$. Then, we obtain the decomposition
        \begin{align*}
        \mathcal{I}_\e&:=\sup_{t\in [0,T]}\left\vert\int_{\mathbb{R}^{2d}}\varphi\,(\rho_\varepsilon\otimes \rho_\varepsilon-\rho\otimes \rho)\,dx\,dx_*\right\vert\\
        &\leq \sup_{t\in [0,T]}\left\vert\sum_{k=1}^n\int_{\mathbb{R}^{2d}}\sigma_k\otimes \psi_k\,(\rho_\varepsilon-\rho)\otimes \rho_\varepsilon\,dx\,dx_*\right\vert\\ &\quad +\sup_{t\in [0,T]}\left\vert\sum_{k=1}^n\int_{\mathbb{R}^{2d}}\sigma_k\otimes \psi_k\,\rho\otimes (\rho_\varepsilon-\rho)\,dx\,dx_*\right\vert\\
        &\quad +\sup_{t\in [0,T]}\left\vert\int_{\mathbb{R}^{2d}}\left(\varphi-\sum_{k=1}^n\sigma_k\otimes \psi_k\right)(\rho_\varepsilon\otimes \rho_\varepsilon-\rho\otimes \rho)\,dx\,dx_*\right\vert =: \mathcal{I}_{\e,1}+\mathcal{I}_{\e,2}+\mathcal{I}_{\e,3}.
        \end{align*}
        On the one hand, by \eqref{E-Stone-Weierstrass} we obtain $\mathcal{I}_{\e,3}\leq 2\delta$. On the other hand, we find that 
        $$\mathcal{I}_{\e,1}\leq \sum_{k=1}^n\Vert \psi_k\Vert_{L^\infty}\sup_{t\in [0,T]}\left\vert\int_{\mathbb{R}^d}\sigma_k(x)(\rho_\varepsilon(t,dx)-\rho(t,dx))\right\vert,$$
        and a similar estimate holds for $\mathcal{I}_{\e,2}$. Since $\rho_\varepsilon\rightarrow \rho$ in $C([0,T],\mathcal{M}(\mathbb{R}^d)-\mbox{narrow})$, by \eqref{E-compactness-fast-regular-gain} we obtain that $\mathcal{I}_{\e,1}\to 0$ and $\mathcal{I}_{\e,2}\to 0$, as $\e\to 0$. Moreover, given that $\delta>0$ is arbitrary, we have that $\mathcal{I}_\e\to 0$ when $\e\to0$. Thus, by arbitrariness of $\varphi\in C_0(\mathbb{R}^{2d})$, we find that $\rho_\varepsilon\otimes \rho_\varepsilon\rightarrow\rho\otimes \rho$ in the weaker space $C([0,T],\mathcal{M}(\mathbb{R}^{2d})-\mbox{weak}*)$. To improve such a convergence into narrow convergence, we a use standard cut-off argument and the uniform tightness of $\rho_\varepsilon$ in $\eqref{E-a-priori-estimates-fast-regular}_2$, thus ending the proof.
	\end{proof}
	
	We are now ready to end the proof of Theorem \ref{T2.1}. Fix any $\psi\in C^1_c((0,T)\times \mathbb{R}^d)$ and set $\varepsilon_0>0$ small enough so that $\supp\psi\subset [\varepsilon_0,T]\times \mathbb{R}^d$. Then, passing to the limit as $\varepsilon\rightarrow 0$ in the weak formulation \eqref{E-velocity-weak-form-fast-regular} for velocity we obtain
	\begin{align*}
	0&=\lim_{\varepsilon\rightarrow 0}\left\{\int_{\varepsilon_0}^T\int_{\mathbb{R}^d}\psi(\phi*A_\varepsilon)\rho_\varepsilon\,dx\,dt-\int_{\varepsilon_0}^T\int_{\mathbb{R}^d}\psi(\phi*\rho_\varepsilon)A_\varepsilon\,dx\,dt\right\}\\
	&\quad-\lim_{\e\rightarrow0}\int_0^T\int_{\bbr^d}\psi(\nabla W*\rho_\e)\rho_\e\,dx\,dt\\
	&\quad+\lim_{\varepsilon\rightarrow 0}\left\{\int_0^T\int_{\mathbb{R}^d}\psi\frac{u^\infty(t)}{\theta^\infty(t)}\rho_\varepsilon\,dx\,dt-\int_{\varepsilon_0}^T\int_{\mathbb{R}^d}\psi A_\varepsilon\,dx\,dt\right\}.
	\end{align*}
	On the one hand, the linear terms in the second line easily pass to the limit, which  can be identified in terms of $\rho$ and $j$ by virtue of \eqref{E-compactness-fast-regular-gain} and Lemma \ref{L2.3}, thus obtaining
	\begin{multline*}
	\lim_{\varepsilon\rightarrow 0}\left\{\int_0^T\int_{\mathbb{R}^d}\psi\frac{u^\infty(t)}{\theta^\infty(t)}\rho_\varepsilon\,dx\,dt-\int_{\varepsilon_0}^T\int_{\mathbb{R}^d}\psi A_\varepsilon\,dx\,dt\right\}\\
	=\int_0^T\int_{\mathbb{R}^d}\psi\frac{u^\infty(t)}{\theta^\infty(t)}\rho(t,dx)\,dt-\int_0^T\int_{\mathbb{R}^d}\frac{\psi}{\theta^\infty(t)}j(t,dx)\,dt.
	\end{multline*}
	On the other hand, let us restate the nonlinear terms in the first line as follows
	\begin{align*}
	\int_{\varepsilon_0}^T&\int_{\mathbb{R}^d}\psi(\phi*A_\varepsilon)\rho_\varepsilon\,dx\,dt-\int_{\varepsilon_0}^T\int_{\mathbb{R}^d}\psi(\phi*\rho_\varepsilon)A_\varepsilon\,dx\,dt\\
	&=\int_{\varepsilon_0}^T\int_{\mathbb{R}^{2d}}\psi(t,x)\phi(x-x_*)(\rho_\varepsilon(t,x)A_\varepsilon(t,x_*)-A_\varepsilon(t,x)\rho_\varepsilon(t,x_*))\,dx\,dx_*\,dt\\
	&=\frac{1}{2}\int_{\varepsilon_0}^T\int_{\mathbb{R}^{2d}} K_\psi(t,x,x_*)(\rho_\varepsilon(t,dx) A_\varepsilon(t,dx_*)-A_\varepsilon(t,dx) \rho_\varepsilon(t,dx_*))\,dt,
	\end{align*}
	where $K_\psi(t,x,x_*):=(\psi(t,x)-\psi(t,x_*))\phi(x-x_*)$ and we have used the usual symmetrization of the integral in the last term. Since $K_\psi\in L^2(\varepsilon_0,T;C_0(\mathbb{R}^{2d}))$, we can then apply Lemma \ref{L2.4} to pass to the limit and we find
	\begin{align*}
	\lim_{\varepsilon\rightarrow 0}&\int_{\varepsilon_0}^T\int_{\mathbb{R}^d}\psi(\phi*A_\varepsilon)\rho_\varepsilon\,dx\,dt-\int_{\varepsilon_0}^T\int_{\mathbb{R}^d}\psi(\phi*\rho_\varepsilon)A_\varepsilon\,dx\,dt\\
	&=\frac{1}{2}\int_{\varepsilon_0}^T\int_{\mathbb{R}^{2d}} \frac{1}{\theta^\infty(t)}K_\psi(t,x,x_*)(\rho(t,dx) j(t,dx_*)-j(t,dx)\rho(t,dx_*))\,dt\\
	&=\int_0^T\int_{\mathbb{R}^d}\frac{\psi}{\theta^\infty(t)}(\phi*j)\rho(t,dx)\,dt-\int_0^T\int_{\mathbb{R}^d}\frac{\psi}{\theta^\infty(t)}(\phi*\rho)j(t,dx)\,dt,
	\end{align*}
	where we have undone the symmetrization. Similarly, we use the convergence of $\rho_\e\otimes\rho_\e$ to obtain the limit of the nonlinear term in the second line:
	\begin{align*}
	\lim_{\e\to0} \int_0^T \psi(\nabla W*\rho_\e)\rho_\e\,dx\,dt=\int_0^T\psi(\nabla W*\rho)\rho\,dx\,dt.
	\end{align*} 
	Putting everything together yields the equation
	\begin{align*}
	0&=\int_0^T\int_{\mathbb{R}^d}\psi\frac{u^\infty(t)}{\theta^\infty(t)}\rho(t,dx)\,dt-\int_0^T\int_{\mathbb{R}^d}\frac{\psi}{\theta^\infty(t)}j(t,dx)\,dt,\\
	&\quad +\int_0^T\int_{\mathbb{R}^d}\frac{\psi}{\theta^\infty}(\phi*j)\rho(t,dx)\,dt-\int_0^T\int_{\mathbb{R}^d}\frac{\psi}{\theta^\infty(t)}(\phi*\rho)j(t,dx)\,dt\\
	&\quad -\int_0^T\int_{\bbr^d}\psi(\nabla W*\rho)\rho(t,dx)\,dt,
	\end{align*}
	for each $\psi\in C^1_c((0,T)\times \mathbb{R}^d)$, then identifying the limiting velocity equation. Regarding the internal variable equation, it suffices to show that the limit of the first term of the right-hand side in \eqref{E-internal-variable-weak-form-fast-regular} vanishes as $\e\to0$. Indeed, we have
	\begin{align*}
    	&\left\vert\int_0^T\int_{\mathbb{R}^d}\psi(t,x)(\zeta*\rho_\varepsilon)B_\varepsilon\,dx\,dt-\int_0^T\int_{\mathbb{R}^d}\psi(t,x)(\zeta*B_\varepsilon)\rho_\varepsilon\,dx\,dt\right\vert\\
    	&\quad \leq \Vert \psi\Vert_{C_c((0,T)\times \mathbb{R}^d)}\int_{\varepsilon_0}^T\int_{\mathbb{R}^{4d}\times \mathbb{R}^2_+}\zeta(x-x_*)\left\vert\frac{1}{\theta}-\frac{1}{\theta_*}\right\vert f_\varepsilon(t,z)f_\varepsilon(t,z_*)\,dz\,dz_*\,dt\\
    	&\quad \leq \frac{\Vert \zeta\Vert_{L^\infty}\Vert \psi\Vert_{C_c((0,T)\times \mathbb{R}^d)}}{\theta_m^2}\int_{\varepsilon_0}^T D_\theta^\varepsilon(t)\,dt\leq C\frac{\Vert \zeta\Vert_{L^\infty}\Vert \psi\Vert_{C_c((0,T)\times \mathbb{R}^d)}}{\theta_m^2}T\varepsilon^{\alpha},
    \end{align*}
    for any $\varepsilon<\varepsilon^{1/\beta}$, where we have used the concentration estimate of the diameter of internal variable in Corollary \ref{cor_temp}. Therefore, the full right-hand side of \eqref{E-internal-variable-weak-form-fast-regular} vanishes as $\e\to0$, which concludes the proof of Theorem \ref{T2.1}.

	\subsection{Weakly singular influence functions}\label{subsec:3.5}
	We emphasize that the fact that $\phi$, $\zeta$, and $\nabla W$ are  uniformly bounded with respect to $\varepsilon$ has strongly be used in many parts of the proof in the previous subsection. In particular, the hydrodynamic limit of the system \eqref{E-kinetic-singular-fast}, where influence functions $\phi$, $\zeta$ are replaced by the singularly scaled $\phi_\varepsilon$, $\zeta_\varepsilon$ in \eqref{E-coupling-weights-epsilon}, does not immediately follow from the previous Theorem \ref{T2.1}, but appropriate modifications are required. In this section we sketch the main ideas guaranteeing that the corresponding hydrodynamic limit of \eqref{E-kinetic-singular-fast} holds rigorously when $W$ still verifies the strong regularity condition \eqref{H-hypothesis-phi-zeta-W}. We will briefly elaborate on the case of singularly scaled potentials $W_\e$ like in \eqref{E-Cucker-Dong-potentials-scaled} later in Remark \ref{R-hydro-fast-singular-potentials}. Unless otherwise specified, we will only assume conditions \eqref{E-hypothesis-initial-data} on the initial data and assumption \eqref{E-hypothesis-macro-species} on $(\bar\rho,\bar u,\bar e)$.
	
	\medskip
	
	$\bullet$ {\bf Concentration of internal variable support}.\\
	In this case, the characteristic system associated with \eqref{E-kinetic-singular-fast} takes a similar form
	\begin{align*}
	&\frac{dX_\varepsilon}{dt}=V_\varepsilon,\\
	&\frac{dV_\varepsilon}{dt}=\frac{1}{\varepsilon}F_\varepsilon[f_\varepsilon](t,Z_\varepsilon)+\frac{1}{\e}H[f_\e](t,Z_\e)+\frac{1}{\varepsilon}F_c[\bar\rho,\bar u,\bar e](t,Z_\varepsilon),\\
	&\frac{d\Theta_\varepsilon}{dt}=\frac{1}{\varepsilon}G_\varepsilon[f_\varepsilon](t,Z_\varepsilon)+\frac{1}{\varepsilon}G_c[\bar \rho,\bar e](t,Z_\varepsilon),\\
	&Z_\varepsilon(0;0,z_0)=z_0,
	\end{align*}
	for $t\in [0,T]$ and any $z_0\in \mathbb{R}^{2d}\times \mathbb{R}_+$, where the operators $F_\varepsilon$ and $G_\varepsilon$ now contain the information about the scaled influence functions $\phi_\varepsilon$ and $\zeta_\varepsilon$, which become singular as $\varepsilon\rightarrow 0$. Note that the operator $H$ with the Lipschitz aggregation force $\nabla W$ remains. By inspection in the proofs of Lemma \ref{lemma_temp} and Corollary \ref{cor_temp} we notice that the terms coming from $G_\varepsilon$ in the estimate of the internal variable were neglected. Therefore, Lemma \ref{lemma_temp} and Corollary \ref{cor_temp} remain valid for the singular regime under the same assumptions.
	
	\medskip
	
	$\bullet$ {\bf A priori estimates}.\\
	Unfortunately, by inspection on the proofs of Lemma \ref{L2.2} and Corollary \ref{C2.3}, we notice that there is a delicate point where the uniform-in-$\varepsilon$ bound of $\phi$ and $W$ are used. Specifically, we apply \eqref{B-4} with $k=2$, integrate over $[0,T]$ and use the control of the internal variable concentration in Lemma \ref{lemma_temp} and Corollary \ref{cor_temp} to derive
	\begin{multline}\label{E-v-moment-singular-phi}
	2\left(\frac{1}{\theta_M}-\frac{\|\phi_\e\|_{L^\infty}D_\theta^\e(0)}{\theta_m^2}\right)\int_0^T\int_{\bbr^{2d}\times\bbr_+}|v|^2f_\e\,dz\,dt\\
	+\frac{1}{\theta_M}\int_0^T\int_{\bbr^{4d}\times\bbr_+^2}\phi_\e(x-x_*)|v-v_*|^2f_\e(t,z)f_\e(t,z_*)\,dz\,dz_*\,dt\\
	\leq \e\int_{\bbr^{2d}\times\bbr_+}|v|^2f_\e^0(z)\,dz+2\int_0^T\left|\frac{u^\infty(t)}{\theta^\infty(t)}\right|\int_{\bbr^{2d}\times\bbr_+}|v|f_\e\,dz\,dt\\
	-2\int_0^T\int_{\bbr^{4d}\times\bbr_+^2}\nabla W (x-x_*)\cdot vf_\e(t,z)f_\e(t,z_*)\,dz\,dz_*\,dt.
	\end{multline}
	On the one hand, by the control of $u^\infty/\theta^\infty$ in Remark \ref{R-hypothesis-macro-species-gain} and the Cauchy--Schwartz inequality, we obtain
    $$2\int_0^T\left|\frac{u^\infty(t)}{\theta^\infty(t)}\right|\int_{\bbr^{2d}\times\bbr_+}|v|f_\e\,dz\,dt\leq 2F_0\Vert \vert v\vert^2 f_\varepsilon\Vert_{L^1(0,T;L^1(\mathbb{R}^{2d}\times \mathbb{R}_+))}^{1/2}.
    $$
    On the other hand, by the uniform bound of $\nabla W$ a similar argument yields
    \begin{multline*}
    -2\int_0^T\int_{\bbr^{4d}\times\bbr_+^2}\nabla W (x-x_*)\cdot vf_\e(t,z)f_\e(t,z_*)\,dz\,dz_*\,dt\\
    \leq  2T^{1/2}\Vert \nabla W\Vert_{L^\infty}\Vert \vert v\vert^2 f_\varepsilon\Vert_{L^1(0,T;L^1(\mathbb{R}^{2d}\times \mathbb{R}^+))}^{1/2}.
    \end{multline*}
    Plugging both bounds into \eqref{E-v-moment-singular-phi}, using Young's inequality and noting that $\Vert \phi_\varepsilon\Vert_{L^\infty}=1/\varepsilon$, we recover an analogue estimate like \eqref{E-velocity-moment-regular-phi}
    \begin{multline*}
		\left(\frac{2}{\theta_M}-\frac{2D_\theta^\varepsilon(0)}{\varepsilon\theta_m^2}-\delta\right)\||v|^2f_\e\|_{L^1(0,T;L^1(\bbr^{2d}\times\bbr_+))}\\
		+\frac{1}{\theta_M}\int_0^T\int_{\bbr^{4d}\times\bbr^2_+}\phi(x-x_*)\vert v-v_*\vert^2\,f_\varepsilon(t,z)f_\varepsilon(t,z_*)\,dz\,dz_*\,dt \leq \varepsilon E_0+\frac{G_0^2}{\delta},
    \end{multline*}
    for any $\delta>0$ and any $\varepsilon>0$. Then, the compatibility condition \eqref{E-hypothesis-thetaM-thetam} is not enough and further information on $D_\theta^\varepsilon(0)$ is necessary in order to compensate the new factor $\frac{1}{\varepsilon}$ coming from singularities. Specifically we shall assume that the initial diameter of the internal variable shrinks asymptotically at a specific rate, namely,
	\begin{equation}\label{E-hypothesis-thetaM-thetam-strong}
	\limsup_{\varepsilon\rightarrow 0}\frac{D^\varepsilon_\theta(0)}{\varepsilon}< \frac{\theta_m^2}{\theta_M}.
	\end{equation}
	Under this conditions, we recover the a priori estimates in Lemma \ref{L2.2} and Corollary \ref{C2.3} modulo an appropriate subsequence, that we denote in the same way for simplicity.

\begin{remark}[Strong initial concentration]
		We emphasize that the new compatibility condition \eqref{E-hypothesis-thetaM-thetam-strong} is more restrictive than the previous one \eqref{E-hypothesis-thetaM-thetam}, as it does not account for generic initial data $f_\varepsilon^0$ with uniformly bounded internal variable supports. Instead, $\supp_\theta f_\varepsilon^0$ must shrink fast enough. Then, the concentration phenomenon does not only rely on the dynamics itself but also on the choice of initial data. This suggests that the initial layer in Corollary \ref{cor_temp} is no longer necessary for the asymptotic concentration of the internal variable, that shrinks instantaneously after $t=0$ at rate $\mathcal{O}(\varepsilon)$.
	\end{remark}
	
	$\bullet$ {\bf Passing to the limit}.\\
	The above estimates suggest that an analogue of Theorem \ref{T2.1} also holds.
	
	\begin{theorem}\label{T-hydro-fast-singular}
		Let $f_\e$ be a solution to the equation \eqref{E-kinetic-singular-fast} subject to the initial data $f_\e^0$ with parameters $\lambda_1\in (0,1)$, $\lambda_2>0$. Assume that the initial data verify \eqref{E-hypothesis-initial-data}, $(\bar \rho,\bar u,\bar e)$ verifies \eqref{E-hypothesis-macro-species} and the parameters fulfill the stronger compatibility condition \eqref{E-hypothesis-thetaM-thetam-strong}. Then,
		\begin{align*}
		&\rho_\varepsilon\rightarrow\rho,\quad \mbox{in }C([0,T],\mathcal{M}(\mathbb{R}^{d})-\mbox{narrow}),\\
		&j_\varepsilon\overset{*}{\rightharpoonup} j,\quad \mbox{in }L^2_w(0,T;\mathcal{M}(\mathbb{R}^d)^d),
		\end{align*}
		as $\varepsilon\rightarrow 0$, for some probability measure $\rho$, some finite Radon measure $j$ and some subsequence of $\{\rho_\varepsilon\}_{\varepsilon>0}$ and $\{j_\varepsilon\}_{\varepsilon>0}$ that we denote in the same way. In addition, $(\rho,j)$ solves the following problem
		$$
		\begin{array}{ll}
		\displaystyle \partial_t \rho+\nabla\cdot j =0, & (t,x)\in [0,T)\times \mathbb{R}^d,\\
		\displaystyle j-u^\infty(t)\rho+\theta^\infty(t)(\nabla W*\rho)\rho=\rho (\phi_0*j)-j(\phi_0*\rho), & (t,x)\in (0,T)\times \mathbb{R}^d,\\
		\rho(t=0)=\rho^0, & x\in \mathbb{R}^d,
		\end{array}
		$$
		in distributional sense, for an appropriate meaning of the nonlinear term. Here, $\phi_0$ is the singular influence function in \eqref{E-singular-kernels} and, $\theta^\infty(t)$, $u^\infty(t)$ are the background mean value of internal variable and velocity, see \eqref{E-background-mean-theta} and \eqref{E-background-mean-velocity}.
	\end{theorem}
	
	Most of the steps in the previous proof remains unchanged. In fact, since Corollary \ref{C2.3} remains true under \eqref{E-hypothesis-thetaM-thetam-strong}, similar a priori estimates as in \eqref{E-a-priori-estimates-fast-regular} fulfill. Then, we still recover the same compactness results in \eqref{E-compactness-fast-regular} and \eqref{E-compactness-fast-regular-gain} along with the identification of the nonlinear moments in Lemma \ref{L2.3} and Lemma \ref{L2.4}. Indeed, passing to the limit in all the linear terms of \eqref{E-velocity-weak-form-fast-regular}-\eqref{E-internal-variable-weak-form-fast-regular} in weak formulation is identical. The only delicate point concerns the identification of the nonlinear terms. Recall that given any test function $\psi\in C^1_c((0,T)\times \mathbb{R}^d)$, and $\varepsilon_0>0$ small enough so that $\supp\psi\subset [\varepsilon_0,T]\times \mathbb{R}^d$, the nonlinear alignment term in the right-hand side of the weak formulation for the velocity equation \eqref{E-velocity-weak-form-fast-regular} can be restated in a symmetrized way as follows
	\begin{align*}
	\int_0^T\int_{\mathbb{R}^d}&\psi(\phi_\varepsilon*A_\varepsilon)\rho_\varepsilon\,dx\,dt-\int_0^T\int_{\mathbb{R}^d}\psi(\phi_\varepsilon*\rho_\varepsilon)A_\varepsilon\,dx\,dt\\
	&=\int_{\e_0}^T\int_{\mathbb{R}^{2d}}\psi(t,x)\phi_\varepsilon(x-x_*)(\rho_\varepsilon(t,x)A_\varepsilon(t,x_*)-A_\varepsilon(t,x)\rho_\varepsilon(t,x_*))\,dx\,dx_*\,dt\\
	&=\frac{1}{2}\int_{\e_0}^T\int_{\mathbb{R}^{2d}} K_{\psi,\varepsilon}(t,x,x_*)(\rho_\varepsilon(t,dx) A_\varepsilon(t,dx_*)-A_\varepsilon(t,dx) \rho_\varepsilon(t,dx_*))\,dt,
	\end{align*}
	where $K_{\psi,\varepsilon}(t,x,x_*):=(\psi(t,x)-\psi(t,x_*))\phi_\varepsilon(x-x_*)$. Consider the integral
	\begin{align}\label{E-nonlinear-term-fast-singular}
	\begin{aligned}
	\mathcal{I}_\varepsilon&:=\frac{1}{2}\int_{\e_0}^T\int_{\mathbb{R}^{2d}}K_{\psi,\varepsilon}(\rho_\varepsilon\otimes A_\varepsilon-A_\varepsilon\otimes \rho_\varepsilon)\,dx\,dx_*\,dt\\
	&\hspace{3.5cm}-\frac{1}{2}\int_{\e_0}^T\int_{\mathbb{R}^{2d}}K_{\psi,0}\left(\rho\otimes \frac{j}{\theta^\infty}-\frac{j}{\theta^\infty}\otimes \rho\right)\,dx\,dx_*\,dt\\
	&=\frac{1}{2}\int_{\e_0}^T\int_{\mathbb{R}^{2d}}(K_{\psi,\varepsilon}-K_{\psi,0})(\rho_\varepsilon\otimes A_\varepsilon-A_\varepsilon\otimes \rho_\varepsilon)\,dx\,dx_*\,dt\\
	&\quad+\frac{1}{2}\int_{\e_0}^T\int_{\mathbb{R}^{2d}}K_{\psi,0}\left[(\rho_\varepsilon\otimes A_\varepsilon-A_\varepsilon\otimes \rho_\varepsilon)-\left(\rho\otimes \frac{j}{\theta^\infty}-\frac{j}{\theta^\infty}\otimes \rho\right)\right]\,dx\,dx_*\,dt\\
	&=:\mathcal{I}_{\varepsilon,1}+\mathcal{I}_{\varepsilon,2}.
	\end{aligned}
	\end{align}
	Our goal is to show that $\mathcal{I}_{\varepsilon,i}\rightarrow 0$, as $\varepsilon\rightarrow 0$. On the one hand, since $\lambda_1\in (0,1)$, then the smoothness of the test function $\psi$ kills the singularity at the diagonal and $K_{\psi,0}\in L^2(0,T;C_0(\mathbb{R}^{2d}))$. Thus, using the weak-$*$ convergence of tensor product $\rho_\varepsilon\otimes A_\varepsilon$ in $L^2_w(0,T;$ $\mathcal{M}(\mathbb{R}^{2d}))$ in Lemma \ref{L2.4}, we obtain that $\mathcal{I}_{\varepsilon,2}\rightarrow 0$, as $\varepsilon\rightarrow 0$. On the other hand, by the mean value theorem and the smoothness of $\psi$ we obtain the following error estimate
	$$\vert K_{\psi,\varepsilon}(t,x,x_*)-K_{\psi,0}(t,x,x_*)\vert \leq C\varepsilon^{1-\lambda_1}\Vert \nabla_x\psi\Vert_{C_c((0,T)\times \mathbb{R}^d)},$$
	for each $t\in (0,T)$, each $x,x_*\in \mathbb{R}^d$ and each $\varepsilon>0$. Therefore, $$\vert \mathcal{I}_{\varepsilon,1}\vert \leq C\varepsilon^{1-\lambda_1}\Vert \nabla\psi\Vert_{C_c((0,T)\times \mathbb{R}^d)}T^{1/2}\Vert A_\varepsilon\Vert_{L^2(0,T;L^1(\mathbb{R}^d))}.$$
	Hence, the a priori estimates allow concluding that $\mathcal{I}_{\varepsilon,1}\rightarrow 0$, when $\varepsilon\rightarrow 0$. Second, the convergence of the aggregation nonlinear term involving $W$ can be obtained in a similar way by using the regularity assumption \eqref{H-hypothesis-phi-zeta-W} and the convergence of $\rho_\e\otimes \rho_\e$ in $C([0,T],\mathcal{M}(\mathbb{R}^{2d})-\mbox{narrow})$ obtained in Lemma \ref{L2.4}. Finally, the nonlinear term for the internal variable equation in the weak formulation \eqref{E-internal-variable-weak-form-fast-regular} can be estimated by the same argument as in the regular influence function case. Precisely, by Corollary \ref{cor_temp}, we obtain
	\begin{align*}
    	&\left\vert\int_{\e_0}^T\int_{\mathbb{R}^d}\psi(t,x)(\zeta_\e*\rho_\varepsilon)B_\varepsilon\,dx\,dt-\int_{\e_0}^T\int_{\mathbb{R}^d}\psi(t,x)(\zeta_\e*B_\varepsilon)\rho_\varepsilon\,dx\,dt\right\vert\\
    	&\qquad \leq C\frac{\Vert \zeta_\e\Vert_{L^\infty}\Vert \psi\Vert_{C_c((0,T)\times \mathbb{R}^d)}}{\theta_m^2}T\varepsilon^{\alpha}\le C\frac{\Vert \psi\Vert_{C_c((0,T)\times \mathbb{R}^d)}}{\theta_m^2}T\varepsilon^{\alpha-\lambda_2},
    \end{align*}
   for any $\varepsilon<\varepsilon_0^{1/\beta}$. Since $\alpha$ in Corollary \ref{cor_temp} can be taken larger than $\lambda_2$, the right-hand side again vanishes as $\e\to0$. Therefore, the limit of the internal variable equation \eqref{E-internal-variable-weak-form-fast-regular} becomes trivial again, which conclude the proof.

	\medskip
	
	$\bullet$ {\bf Non-concentration for $\lambda_1=1$}.
	Notice that only coefficients $\lambda_1$ in the restricted range $(0,1)$ have been considered in the preceding proof of Theorem \ref{T-hydro-fast-singular}. Indeed, $K_{\psi,0}$ in the symmetrized formulation \eqref{E-nonlinear-term-fast-singular} of the nonlinear term becomes discontinuous at diagonal points $x=x_*$ for any value $\lambda_1\geq 1$. However, notice that $K_{\psi,0}$ remains bounded for $\lambda_1=1$. In this part we extend the proof of Theorem \ref{T-hydro-fast-singular} to this limiting singular regime.  
	On the one hand, using the previous extra dissipation estimate \eqref{E-dissipation-estimate} in Corollary \ref{C2.3} and arguing as in \cite[Lemma 3.10]{PS17} for the singular Cucker--Smale model, we obtain the following non-concentration of the commutator $\rho_\e\otimes A_\e-A_\e\otimes \rho_\e$ at diagonal points:
	\begin{align}
	\liminf_{\delta,\varepsilon\rightarrow 0}\vert\rho_\varepsilon(t,\cdot) \otimes A_\varepsilon(t,\cdot)-A_\varepsilon(t,\cdot)\otimes \rho_\varepsilon(t,\cdot)\vert(\Delta+B_\delta)&=0,\label{commut_A}
	\end{align} 
	for a.e. $t\in [0,T]$, where $\Delta :=\{(x,x)\in\bbr^{2d}:\,x\in\bbr^d\}$ is the diagonal set. 

	On the other hand, let us show the rigorous passage to the limit in \eqref{E-nonlinear-term-fast-singular}. Thanks to the nonconcentration property \eqref{commut_A} and \cite[Proposition 3.11]{PS17} we can pass to the limit on $\mathcal{I}_{\varepsilon,2}$ in \eqref{E-nonlinear-term-fast-singular} and we obtain $\mathcal{I}_{\varepsilon,2}\rightarrow	0$ when $\varepsilon\rightarrow 0$. Also, since $\lambda_1=1$ we obtain the following estimate by the mean value theorem
	$$\vert K_{\psi,\varepsilon}(t,x,x_*)-K_{\psi,0}(t,x,x_*)\vert\leq C\phi_\varepsilon(x-x_*)^{1/2}\varepsilon^{1/2},$$
	for each $t\in (0,T)$, any $x,x_*\in \mathbb{R}^d$ and appropriate $C\in \mathbb{R}_+$. Therefore, we estimate $\mathcal{I}_{\e,1}$ by splitting it into two terms as follows
	\begin{align*}
	\vert \mathcal{I}_{\varepsilon,1}\vert&\leq C\varepsilon^{1/2}\int_{\e_0}^T\int_{\mathbb{R}^{4d}\times \mathbb{R}_+^2}\left\vert \frac{v}{\theta}-\frac{v_*}{\theta_*}\right\vert\phi_\varepsilon(x-x_*)^{1/2}f_\varepsilon(t,z)f_\varepsilon(t,z_*)\,dz\,dz_*\,dt\\
	&\leq C\varepsilon^{1/2}\int_{\varepsilon_0}^T\int_{\mathbb{R}^{4d}\times\mathbb{R}_+^2}\phi_\varepsilon(x-x_*)^{1/2}\frac{1}{\theta_*}\vert v-v_*\vert f_\varepsilon(t,z)f_\varepsilon(t,z_*)\,dz\,dz_*\,dt\\
	&\quad +C\varepsilon^{1/2}\int_{\varepsilon_0}^T\int_{\mathbb{R}^{4d}\times \mathbb{R}_+^2}\phi_\varepsilon(x-x_*)^{1/2}\vert v\vert\,\left\vert\frac{1}{\theta}-\frac{1}{\theta_*}\right\vert f_\varepsilon(t,z)f_\varepsilon(t,z_*)\,dz\,dz_*\,dt\\
	&\leq \frac{C}{\theta_m}\varepsilon^{1/2}\int_{\varepsilon_0}^T\int_{\mathbb{R}^{4d}\times \mathbb{R}_+^2}\phi_\varepsilon(x-x_*)^{1/2}\vert v-v_*\vert f_\varepsilon(t,z)f_\varepsilon(t,z_*)\,dz\,dz_*\,dt\\
	&\quad +\frac{C}{\theta_m^2}\int_{\varepsilon_0}^TD_\theta^\varepsilon(t)\int_{\mathbb{R}^{2d}\times \mathbb{R}_+}\vert v\vert f_\varepsilon(t,z)\,dz\,dt,
	\end{align*}
	where in the last line we have used that $\Vert\phi_\varepsilon\Vert_{L^\infty}=\frac{1}{\varepsilon}$. Hence, using Jensen's and the Cauchy--Schwartz inequality in the right-hand side we obtain
	\begin{align*}
	\vert \mathcal{I}_{\varepsilon,1}\vert&\leq \frac{CT^{1/2}}{\theta_m^2}\varepsilon^{1/2}\left(\int_0^T\int_{\mathbb{R}^{4d}\times \mathbb{R}_+^2}\phi_\varepsilon(x-x_*)\vert v-v_*\vert^2 f_\varepsilon(t,z)f_\varepsilon(t,z_*)\,\,dz\,dz_*\,dt\right)^{1/2}\\
	&\quad +\frac{C}{\theta_m^2}\left(\int_{\varepsilon_0}^T D_\theta^\varepsilon(t)^2\,dt\right)^{1/2}\Vert \vert v\vert f_\varepsilon\Vert_{L^2(0,T;L^1(\mathbb{R}^{2d}\times \mathbb{R}_+))}\\
	&\leq \frac{CT^{1/2}}{\theta_m^2}(\theta_M (E_0+C G_0^2))^{1/2}\varepsilon^{1/2}+\frac{C}{\theta_m^2}(C(E_0+CG_0^2))^{1/2}T^{1/2}\varepsilon^\alpha,
	\end{align*}
	for any $\varepsilon<\varepsilon^{1/\beta}$, where we have used the estimates in Corollary \ref{C2.3} and the concentration of the internal variable support in Corollary \ref{cor_temp}. Then, $\mathcal{I}_{\varepsilon,1}\rightarrow 0$, as $\varepsilon\rightarrow 0$.

\begin{remark}[Hydrodynamic limit with singular         aggregation potentials]\label{R-hydro-fast-singular-potentials}
    We emphasize that uniform bound assumption for $\nabla W$ in \eqref{H-hypothesis-phi-zeta-W} has been critical in two steps: 
    \begin{itemize}
    \item The control of second order velocity moment \eqref{E-v-moment-singular-phi}.
    \item The identification of the nonlinear aggregation term in the weak formulation of the velocity equation \eqref{E-velocity-weak-form-fast-regular}.
    \end{itemize}
    However, the scaled aggregation potentials $W_\varepsilon$ in \eqref{E-Cucker-Dong-potentials-scaled} (leading to singular aggregation potentials) can also be addressed in the special case $\lambda_3=\frac{\lambda_1}{2}$. In fact, a similar result like in Theorem \ref{T-hydro-fast-singular} holds true for the choice of parameters $\lambda_1\in (0,1]$, $\lambda_2>0$ and $\lambda_3=\lambda_1/2$. On the one hand, note that identifying the limit of the singular nonlinear aggregation term can be addressed using the same usual symmetrization method, that kills singularities, for any value $\lambda_3\in (0,1)$ (note that $\nabla W_\varepsilon$ is anti-symmetric). On the other hand, although $\nabla W_\varepsilon$ cannot be bounded uniformly in $\e$, the last term in the right-hand side of \eqref{E-v-moment-singular-phi} can be controlled by the dissipation term:
    \begin{align*}
	&2\left|\int_0^T\int_{\bbr^{4d}\times\bbr_+^2}\nabla W_\e(x-x_*)\cdot vf_\e(t,z)f_\e(t,z_*)\,dz\,dz_*\,dt\right|\\
	&\qquad=\left|\int_0^T\int_{\bbr^{4d}\times\bbr_+^2}\nabla W_\e(x-x_*)\cdot (v-v_*)f_\e(t,z)f_\e(t,z_*)\,dz\,dz_*\,dt\right|\\
	&\qquad \le \int_0^T\int_{\bbr^{4d}\times\bbr_+^2}\phi_\e(x-x_*)^{1/2}|v-v_*|f_\e(t,z)f_\e(t,z_*)\,dz\,dz_*\,dt\\
	&\qquad \le \int_0^T\left(\int_{\bbr^{4d}\times\bbr^2_+}\phi_\e(x-x_*)|v-v_*|^2f_\e(t,z)f_\e(t,z_*)\,dz\,dz_*\right)^{1/2}\,dt\\
	&\qquad \le \frac{T}{2\delta}+\delta\int_0^T\int_{\bbr^{4d}\times\bbr_+^2}\phi_\e(x-x_*)|v-v_*|^2f_\e(t,z)f_\e(t,z_*)\,dz\,dz_*\,dt,
	\end{align*}
	for any $\delta>0$. In the second line, we have symmetrized the integral, in the third line we have noticed that  $\vert \nabla W_\varepsilon(x)\vert\leq \phi_\varepsilon(x)$ for $\lambda_3=\frac{\lambda_1}{2}$, and in the last steps we have applied the Cauchy--Schwartz and Young inequalities. By virtue of the above strong initial concentration condition \eqref{E-hypothesis-thetaM-thetam-strong}, we can chose a suitable $\delta>0$ leading to an analogue control \eqref{E-velocity-moment-regular-phi} for the second-order velocity moment and the dissipation estimate.
\end{remark}

\section{Hydrodynamic limit in the weak relaxation regime}\label{sec:3}
\setcounter{equation}{0}

In this section, we derive the hydrodynamic limit of system \eqref{A-6} towards \eqref{A-7}. We follow a similar procedure as in Section \ref{sec:2}. First, we prove that the internal variable support of $f_\e$ still shrinks fast enough, as $\varepsilon\rightarrow 0$. Notice that in contrast with the previous scaling \eqref{A-4}, under the new scaling \eqref{A-6} we no longer have strong relaxation of the internal variable. In turns, such a relaxation term $G_c[\bar\rho,\bar e]f$ involves a weaker scale now. Despite this defect, we still preserve a strong alignment term $\frac{1}{\e}G[f]f$. Therefore, we expect that the internal variable support will rapidly concentrate at some value $\theta(t)$ when $\varepsilon\rightarrow 0$ after a small enough time layer. Although $\theta(t)$ does not instantaneously agree with the background mean value $\theta^\infty(t)$ (see \eqref{E-background-mean-theta}) due to the weak relaxation scale, $\theta(t)$ asymptotically converges to $\theta^\infty(t)$ after sufficiently long time. With such control in hand, we recover appropriate \textit{a priori} estimates for the velocity moments. Finally, we combine the estimates for velocity moments and the internal variable support to derive the rigorous hydrodynamic limit. 

Along this part, we will assume again the previous hypothesis \eqref{E-hypothesis-initial-data} for initial data $f_\varepsilon^0$, the properties \eqref{E-hypothesis-macro-species} for the solution $(\bar \rho,\bar u,\bar e)$ and the compatibility condition \eqref{E-hypothesis-thetaM-thetam} for parameters. 
Since the weak relaxation regime is technically more involved, we will need further assumptions that we list here. On the one hand, define the spatial and velocity radii
\begin{align}\label{E-position-velocity-radii}
\begin{aligned}R_x^\varepsilon(t)&:=\sup\{\vert x\vert:\,x\in \supp_x f_\varepsilon(t)\},\\
R_v^\varepsilon(t)&:=\sup\{\vert v\vert:\,v\in \supp_v f_\varepsilon(t)\},
\end{aligned}
\end{align}
for any $t\in [0,T]$. Then, we will assume the initial uniform confinement of supports
\begin{equation}\label{E-hypothesis-initial-data-uniform-compact-support}
R_x^\varepsilon(0)\leq r_M^0\quad \mbox{and}\quad R_v^\varepsilon(0)\leq v_M^0,
\end{equation}
for all $\varepsilon>0$ and some constants $r_M^0,\,v_M^0>0$ not depending on $\varepsilon>0$. On the other hand, some asymptotic control on the initial value of the internal variable is required to fully characterize its dynamics. Specifically, we will assume that there exists some value $\theta_0\in \mathbb{R}_+$ and a parameter $\gamma>0$ such that
\begin{equation}\label{E-hypothesis-initial-data-internal-variable}
\int_{\mathbb{R}^{2d}\times \mathbb{R}_+}\theta f_\varepsilon^0(z)\,dz=\theta_0+\mathcal{O}(\varepsilon^\gamma),\qquad (\mbox{as }\varepsilon\rightarrow 0).
\end{equation}

\subsection{Hierarchy of moments}
Again, associated with the hierarchy of moments $\rho_\varepsilon$, $j_\varepsilon$, $h_\varepsilon$, $A_\varepsilon$, $B_\varepsilon,\mathcal{S}_\varepsilon^v$ and $\mathcal{S}_\varepsilon^\theta$ in \eqref{E-hierarchy-moments}, we obtain the following hierarchy of equations
\begin{align}\label{E-hierarchy-moments-equations-slow}
\begin{aligned}
&\partial_t\rho_\e+\nabla\cdot j_\e=0,\\
&\e\partial_tj_\e+\e\nabla\cdot\mathcal{S}_\varepsilon^v+A_\varepsilon(\phi*\rho_\varepsilon)-\rho_\varepsilon(\phi*A_\varepsilon)+(\nabla W*\rho_\e)\rho_\e+A_\varepsilon-\rho_\varepsilon\frac{u^\infty(t)}{\theta^\infty(t)}=0,\\
&\e\partial_th_\e+\e\nabla\cdot\mathcal{S}_\varepsilon^\theta+\rho_\varepsilon(\zeta*B_\varepsilon)-B_\varepsilon(\zeta*\rho_\varepsilon)+\varepsilon\left(\frac{\rho_\varepsilon}{\theta^\infty(t)}-B_\varepsilon\right)=0,
\end{aligned}
\end{align}
where $\theta^\infty$ and $u^\infty$ are the background mean value of the internal variable and velocity respectively, see \eqref{E-background-mean-theta} and \eqref{E-background-mean-velocity}. Our final goal is to show rigorously that we can close the above hierarchy of equations \eqref{E-hierarchy-moments-equations-slow} when $\varepsilon\rightarrow 0$. Note the difference with the hierarchy \eqref{E-hierarchy-moments-equations} coming from the weak relaxation scale of the internal variable. Indeed, whilst the inertial terms $\e\partial j_\varepsilon+\e\nabla\cdot \mathcal{S}_\e^v$ in the velocity equation will disappear again in the limit $\varepsilon\rightarrow0$, we will see that a non-trivial dynamics for the internal variable is found due to the weaker relaxation of internal variable.

\subsection{Concentration of the internal variable support}

We recall the definitions of $D_x^\varepsilon(t)$, $D_v^\varepsilon(t)$ and $D_\theta^\varepsilon(t)$ in \eqref{E-diameters} for the diameters of the position, velocity and internal variable supports. Also, we recall that under the new scaling \eqref{A-6}, the characteristic system \eqref{E-characteristic-system-A-4} is modified and consists of trajectories $Z_\varepsilon(t;0,z_0)=(X_\varepsilon(t;0,z_0),V_\varepsilon(t;0,z_0),\Theta_\varepsilon(t;0,z_0))$ solving
\begin{align}\label{E-characteristic-system-A-6}
\begin{aligned}
&\frac{dX_\varepsilon}{dt}=V_\varepsilon,\\
&\frac{dV_\varepsilon}{dt}=\frac{1}{\varepsilon}F[f_\varepsilon](t,Z_\varepsilon)+\frac{1}{\e}H[f_\e](t,Z_\e)+\frac{1}{\varepsilon}F_c[\bar\rho,\bar u,\bar e](t,Z_\varepsilon),\\
&\frac{d\Theta_\varepsilon}{dt}=\frac{1}{\varepsilon}G[f_\varepsilon](t,Z_\varepsilon)+G_c[\bar \rho,\bar e](t,Z_\varepsilon),\\
&Z_\varepsilon(0;0,z_0)=z_0,
\end{aligned}
\end{align}
for $t\in [0,T]$ and any $z_0\in \mathbb{R}^{2d}\times \mathbb{R}_+$.

\begin{lemma}[Concentration of internal variable I]\label{lemma_temp_2}
	Let $f_\e$ be the solution of \eqref{A-6} subject to initial data $f^0_\e$. Assume that the initial data verify \eqref{E-hypothesis-initial-data} and $(\bar \rho,\bar u,\bar e)$ verifies \eqref{E-hypothesis-macro-species}. Consider the constants $0<\theta_m<\theta_M$ defined by \eqref{E-thetam-thetaM}. Then, the following estimates are verified:
	\begin{enumerate}
	\item {\bf (Confinement of support)}
	\[\textup{supp}_\theta f_\varepsilon(t)\subset (\theta_m,\theta_M),\quad \mbox{for all} \quad t\in [0,T].\]
	\item {\bf (Exponential concentration)}
	\begin{equation}\label{C-1}
	\frac{d^+ D_\theta^\varepsilon(t)}{dt}\le -\frac{1}{\theta_M^2}\left(\frac{\zeta(D_x^\varepsilon(t))}{\e} +1\right)D_\theta^\varepsilon(t)\le 0,\quad \mbox{for all} \quad t\in [0,T].
	\end{equation}
	\end{enumerate}
\end{lemma}

\begin{proof}
The proof is similar to that of Lemma \ref{lemma_temp}. In fact, a similar argument as in the derivation of \eqref{thetaM} and \eqref{thetam} yields the following estimates
\begin{align}
	\frac{d^+\theta_\varepsilon^M}{dt}&\leq  \frac{1}{\e}\int_{\bbr^{2d}\times\bbr_+}\zeta(D_x^\varepsilon(t))\left(\frac{1}{\theta^M_\e(t)}-\frac{1}{\theta_*}\right)f_\e(t,z_*)\,dz_*+\left(\frac{1}{\theta^M_\e(t)}-\frac{1}{\theta^\infty(t)}\right),\label{thetaM-slow}\\
	\frac{d^+\theta_\varepsilon^m}{dt}&\geq  \frac{1}{\e}\int_{\bbr^{2d}\times\bbr_+}\zeta(D_x^\varepsilon(t))\left(\frac{1}{\theta^m_\e(t)}-\frac{1}{\theta_*}\right)f_\e(t,z_*)\,dz_*+\left(\frac{1}{\theta^m_\e(t)}-\frac{1}{\theta^\infty(t)}\right),\label{thetam-slow}
	\end{align}
	for every $t\in [0,T]$, where we recall that $\theta_\varepsilon^M(t)=\max\supp_\theta f_\varepsilon(t)$ and $\theta_\varepsilon^m(t)=\min\supp_\theta f_\varepsilon(t)$. Then, we infer two different results.
	
	\medskip
	
	$\bullet$ {\sc Step 1}. On the one hand, since the first terms in the right-hand sides of \eqref{thetaM-slow} and \eqref{thetam-slow} have the appropriate sign, we neglect them and we recover
	\begin{align*}
	\frac{d^+\theta_\varepsilon^M}{dt}&\leq \frac{1}{\theta_\varepsilon^M(t)}-\frac{1}{\theta^\infty(t)}\leq \frac{1}{\theta_\varepsilon^M(t)}-\frac{1}{\bar\theta_M}\\
	\frac{d^+\theta_\varepsilon^m}{dt}&\geq \frac{1}{\theta_\varepsilon^m(t)}-\frac{1}{\theta^\infty(t)}\geq \frac{1}{\theta_\varepsilon^m(t)}-\frac{1}{\bar\theta_m},
	\end{align*}
	for every $t\in [0,T]$, where we have used the uniform control \eqref{E-hypothesis-macro-species-gain-theta-infty} of $\theta^\infty$ in Remark \ref{R-hypothesis-macro-species-gain} as a consequence of assumption \eqref{E-hypothesis-macro-species}. By the same continuity argument as in Lemma \ref{lemma_temp}, the preceding estimates ensure that $\theta_\varepsilon^M(t)\leq \theta_M:=\max\{\theta_M^0,\bar\theta_M\}$ and $\theta_\varepsilon^m(t)\geq \theta_m:=\min\{\theta_m^0,\bar\theta_m\}$, for any $t\in [0,T]$, thus proving the confinement of internal variable support.
	
	\medskip
	
	$\bullet$ {\sc Step 2}. Now, we take the difference between \eqref{thetaM-slow} and \eqref{thetam-slow} and recall that $D_\theta^\varepsilon(t)=\theta_\varepsilon^M(t)-\theta_\varepsilon^m(t)$. This implies the inequality
	$$\frac{d^+D_\theta^\varepsilon}{dt}\leq \left(\frac{1}{\varepsilon}\zeta(D_x^\varepsilon(t))+1\right)\left(\frac{1}{\theta_\varepsilon^M(t)}-\frac{1}{\theta_\varepsilon^m(t)}\right)\leq -\frac{1}{\theta_M^2}\left(\frac{1}{\varepsilon}\zeta(D_x^\varepsilon(t))+1\right)D_\theta^\varepsilon(t),$$
	for each $t\in [0,T]$, where in the last step we have used the previous uniform estimates of $\theta_\varepsilon^M$ and $\theta_\varepsilon^m$. This ends the proof.
\end{proof}

Therefore, we still have monotonicity of the internal variable support of $f_\e$. Indeed, it relaxes exponentially fast with rate of order $\mathcal{O}(1)$, as $\varepsilon\rightarrow 0$. Unfortunately, this is too weak in comparison with the results in Section \ref{sec:2} since the dynamics would take a large initial time layer in order for the internal variable support of $f_\varepsilon$ to shrink and concentrate at a single value. In order to circumvent this problem, we have to ensure an adequate behavior of the spatial support as $\varepsilon\rightarrow 0$. Otherwise, if $D_x^\varepsilon(t)$ diverged as $\varepsilon\rightarrow 0$, then the rate $\mathcal{O}(\varepsilon^{-1})$ could be lost. In the sequel, we obtain such a uniform-in-$\varepsilon$ control under the aforementioned confinement condition \eqref{E-hypothesis-initial-data-uniform-compact-support} of initial data.

\begin{lemma}[Concentration of internal variable II]\label{lem_vel_supp}
	Let $f_\e$ be the solution of \eqref{A-6} subject to initial data $f^0_\e$. Assume that the initial data verify \eqref{E-hypothesis-initial-data}, $(\bar \rho,\bar u,\bar e)$ verifies \eqref{E-hypothesis-macro-species} and the parameters fulfill the compatibility condition \eqref{E-hypothesis-thetaM-thetam}. In addition, assume the initial uniform control \eqref{E-hypothesis-initial-data-uniform-compact-support} of the spatial and velocity supports. Then, there exists $v_M\geq v_M^0$ such that
	\begin{align*}
	R_x^\varepsilon(t)&\leq r_M^0+v_M T,\\
	R_v^\varepsilon(t)&\leq v_M,
	\end{align*}
    for every $t\in [0,T]$ and $\varepsilon>0$. In addition, there exists $C_T\in \mathbb{R}_+$ depending on initial parameters and $T$ such that the following decay takes place
    $$D_\theta^\varepsilon(t)\leq D_\theta^\varepsilon(0)e^{-\frac{C_T}{\varepsilon}t},\quad \mbox{for all}\quad t\in [0,T].$$
\end{lemma}

\begin{proof}
	Let us define the critical sets
	\begin{align*}
	M_x^\varepsilon(t)&:=\{z_0\in \supp f_\varepsilon^0:\,\vert X_\varepsilon(t;0,z_0)\vert =R_x^\varepsilon(t)\},\\
	M_v^\varepsilon(t)&:=\{z_0\in \supp f_\varepsilon^0:\,\vert V_\varepsilon(t;0,z_0)\vert =R_v^\varepsilon(t)\},
	\end{align*}
	and notice that
	\begin{align*}
	R_x^\varepsilon(t)&=\max_{z_0\in M_x^\varepsilon(t)}\vert X_\varepsilon(t;0,z_0)\vert,\\
	R_v^\varepsilon(t)&=\max_{z_0\in M_v^\varepsilon(t)}\vert V_\varepsilon(t;0,z_0)\vert,
	\end{align*}
	for every $t\in [0,T]$. Again, we notice that it is not necessarily true that the maximum is propagated along fixed trajectories. However, since the flow $(t,z_0)\mapsto Z_\varepsilon(t;0,z_0)$ is continuous, differentiable with respect to $t$ with continuous derivative, we obtain an analogous identity as in \eqref{E-derivative-maximum} from the usual argument in \cite[Corollary 3.3, Chapter 1]{DR95}, namely,
	\begin{align}\label{E-derivative-maximum-position-velocity}
	\begin{aligned}
	\frac{d^+}{dt} \frac{R_x^\varepsilon(t)^2}{2}&=\max_{z_0\in M_x^\varepsilon(t)}V_\varepsilon(t;0,z_0)\cdot X_\varepsilon(t;0,z_0),\\
	\frac{d^+}{dt} \frac{R_v^\varepsilon(t)^2}{2}&=\max_{z_0\in M_v^\varepsilon(t)}\frac{\partial V_\varepsilon}{\partial t}(t;0,z_0)\cdot V_\varepsilon(t;0,z_0),
	\end{aligned}
	\end{align}
	for every $t\in [0,T]$, where again $d^+/dt$ denotes the right derivative.
	
	\medskip
	
	$\bullet$ {\sc Step 1}. Estimate of the velocity radius.\\
	Fix any time $t\in [0,T]$ and any $z_0\in M_v^\varepsilon(t)$. For simplicity of notation, let us denote 
	$$z_\varepsilon^M:=(x_\varepsilon^M,v_\varepsilon^M,\theta_\varepsilon^M):=(X_\varepsilon(t;0,z_0),V_\varepsilon(t;0,z_0),\Theta_\varepsilon(t;0,z_0)).$$
	Then, we obtain
	\begin{align*}\frac{\partial V_\varepsilon}{\partial t}&(t;0,z_0)\cdot V_\varepsilon(t;0,z_0)=v_\varepsilon^M\cdot\left(\frac{1}{\varepsilon}F[f_\varepsilon](t,z_\varepsilon^M)+\frac{1}{\e}H[f_\e](t,z_\e^M)+\frac{1}{\varepsilon}F_c[\bar\rho,\bar u,\bar e](t,z_\varepsilon^M)\right)\\
	&=\frac{1}{\varepsilon}v_\varepsilon^M\cdot\left(\int_{\mathbb{R}^{2d}\times \mathbb{R}_+}\phi(x_\varepsilon^M-x_*)\left(\frac{v_*}{\theta_*}-\frac{v_\varepsilon^M}{\theta_\varepsilon^M}\right)f_\varepsilon(t,z_*)\,dz_*\right)\\
	&\quad - \frac{1}{\e}v_\e^M\cdot\left(\int_{\bbr^{2d}\times\bbr_+}\nabla W(x_\e^M-x_*)f_\e(t,z_*)\,dz_*\right)+\frac{1}{\varepsilon}v_\varepsilon^M\cdot\left(\frac{u^\infty(t)}{\theta^\infty(t)}-\frac{v_\varepsilon^M}{\theta_\varepsilon^M}\right)\\
	&\leq -\frac{1}{\varepsilon}\frac{\vert v_\varepsilon^M\vert^2}{\theta_M}+\frac{1}{\varepsilon}\int_{\mathbb{R}^{2d}\times \mathbb{R}_+}\phi(x_\varepsilon^M-x_*)\frac{v_\varepsilon^M\cdot (v_*-v_\varepsilon^M)}{\theta_*}f_\varepsilon(t,z_*)\,dz_*\\
	&\quad +\frac{1}{\varepsilon}\vert v_\varepsilon^M\vert^2\int_{\mathbb{R}^{2d}\times \mathbb{R}_+}\phi(x_\varepsilon^M-x_*)\left(\frac{1 }{\theta_*}-\frac{1}{\theta_\varepsilon^M}\right)f_\varepsilon(t,z_*)\,dz_*\\
	&\quad +\frac{1}{\varepsilon}|v_\varepsilon^M| \left(\frac{|u^\infty(t)|}{\theta^\infty(t)}+\|\nabla W\|_{L^\infty}\right)\\
	&\leq -\frac{1}{\varepsilon}\left(\frac{1}{\theta_M}-\Vert \phi\Vert_{L^\infty}\frac{D_\theta^\varepsilon(t)}{\theta_m^2}\right)R_v^\varepsilon(t)^2+\frac{1}{\varepsilon}\left(\frac{\vert u^\infty(t)\vert}{\theta^\infty(t)} +\Vert \nabla W\Vert_{L^\infty}\right)R_v^\varepsilon(t).
	\end{align*}
	In the last inequality we have neglected the second term in the left-hand side, thanks to the property 
	$$v_\varepsilon^M\cdot(v_*-v_\varepsilon^M)\leq \vert v_\varepsilon^M\vert\left(\vert v_*\vert-\vert v_\varepsilon^M\vert\right)\leq 0,$$
	for each $v_*\in \supp_v f_\varepsilon(t)$, due to the fact that $\vert v_\varepsilon^M\vert=R_v^\varepsilon(t)$. Thanks to Lemma \ref{lemma_temp_2}, the internal variable diameter is non-increasing, i.e., $D_\theta^\varepsilon(t)\leq D_\theta^\varepsilon(0)\leq \theta_M^0-\theta_m^0$. Thus, putting everything together into \eqref{E-derivative-maximum-position-velocity} and by arbitrariness of $z_0\in M_v^\varepsilon(t)$, we have
	$$\frac{d^+}{dt}R_v^\varepsilon(t)\leq -\frac{C_1}{\varepsilon}R_v^\varepsilon(t)+\frac{C_2}{\varepsilon},$$
	for every $t\in [0,T]$, where we denote
	$$C_1:=\frac{1}{\theta_M}-\Vert \phi\Vert_{L^\infty}\frac{(\theta_M^0-\theta_m^0)}{\theta_m^2},\quad C_2:=\sup_{t\in [0,T]}\frac{\vert u^\infty(t)\vert}{\theta^\infty(t)}+\Vert \nabla W\Vert_{L^\infty}.$$
	Notice that $C_2$ is finite by Remark \ref{R-hypothesis-macro-species-gain} and hypothesis \eqref{H-hypothesis-phi-zeta-W}, and $C_1$ is positive by the compatibility condition \eqref{E-hypothesis-thetaM-thetam}. Thereby, Gr\"{o}nwall's lemma concludes that
	$$R_v^\varepsilon(t)\leq R_v^\varepsilon(0) e^{-\frac{C_2}{\varepsilon}t}+\frac{C_2}{C_1}(1-e^{-\frac{C_1}{\varepsilon}t})\leq v_M^0+\frac{C_2}{C_1}=:v_M,$$
	for every $t\in [0,T]$.
	
	\medskip
	
	$\bullet$ {\sc Step 2}. Estimate of the spacial radius.\\
	From $\eqref{E-derivative-maximum-position-velocity}_2$ and the Cauchy--Schwartz inequality, we infer
	$$\frac{d^+}{dt}R_x^\varepsilon(t)\leq R_v^\varepsilon(t),$$
	for ech $t\in [0,T]$. By the control on $R_v^\varepsilon(t)$ in {\sc Step 1}, we attain the claimed bound of $R_x^\varepsilon(t)$.
	
	\medskip
	
	$\bullet$ {\sc Step 3}. Strong concentration of the internal variable diameter.\\
	Notice that by the preceding estimate we obtain
	$$\zeta(D_x^\varepsilon(t))\geq \zeta(2R_v^\varepsilon(t))\geq \zeta(2(r_M^0+v_MT))=:C_T,$$
	for every $t\in [0,T]$ and each $\varepsilon>0$. Then, by Lemma \ref{lemma_temp_2} we conclude
	$$\frac{d^+}{dt}D_\theta^\varepsilon(t)\leq -\frac{1}{\theta_M^2}\left(\frac{C_T}{\varepsilon}+1\right)D_\theta^\varepsilon(t),$$
	for every $t\in [0,T]$ and this ends the proof.
\end{proof}

Therefore, as in the strong relaxation regime in Section \ref{sec:2}, we notice that relaxation of the internal variable diameter is still exponential with rate $\mathcal{O}(\varepsilon^{-1})$. However, there is a crucial difference between both scalings. In fact, notice that since the relaxation term of the  internal variable in \eqref{A-6} toward $\theta^\infty(t)$ is weak, we cannot expect the same deviation estimate around $\theta^\infty(t)$ as in Corollary \ref{cor_temp}. Instead, we show that, after a small initial time layer, the internal variable support is concentrated on a small neighborhood of a value $\theta(t)$, which relaxes towards $\theta^\infty(t)$ according to an appropriate ODE.

\begin{corollary}[Initial time layer]\label{C-initial-time-layer-slow-regular}
	Let $f_\e$ be the solution of \eqref{A-6} subject to the initial data $f^0_\e$. Assume that the initial data verify \eqref{E-hypothesis-initial-data}, $(\bar \rho,\bar u,\bar e)$ verifies \eqref{E-hypothesis-macro-species} and the parameters fulfill the compatibility condition \eqref{E-hypothesis-thetaM-thetam}. In addition, assume the initial uniform confinement \eqref{E-hypothesis-initial-data-uniform-compact-support} and the initial asymptotic control \eqref{E-hypothesis-initial-data-internal-variable}, for some value $\theta_0\in \mathbb{R}_+$ and some parameter $\gamma>0$. Then,  there exists a constant $C\in \mathbb{R}_+$ such that
	\begin{align*}
	&D_\theta^\varepsilon(t)\leq C\varepsilon^\alpha,\\
	&\supp f_\theta(t)\subset (\theta(t)-C\varepsilon^{\min\{\beta,\gamma\}},\theta(t)+C\varepsilon^{\min\{\beta,\gamma\}}),
	\end{align*}
	for any $\alpha>0$ and $0<\beta<1$, and for every $t\in [\varepsilon^\beta,T]$, where $\theta=\theta(t)$ is the solutions to the following relaxation ODE:
	\begin{align}\label{relax_ODE}
	\begin{aligned}
	&\dot{\theta}(t) = \frac{1}{\theta(t)}-\frac{1}{\theta^\infty(t)}, \quad t\in [0,T],\\
	&\theta(0) = \theta_0.
	\end{aligned}
	\end{align}
	Here, $C$ depends on the exponents $\alpha$ and $\beta$ along with the parameters $\theta_M^0$, $\theta_m^0$, $\bar\theta_m$, $\bar\theta_M$, $\Vert \zeta\Vert_{L^\infty}$ and $\Vert \theta^\infty(t)\Vert_{W^{1,\infty}(0,T)}$. 
\end{corollary}

\begin{proof}
    The control on $D_\theta^\varepsilon(t)$ after a small initial time layer $t\geq \varepsilon^\beta$ is obtained from Lemma \ref{lem_vel_supp} in the same way as in Corollary \ref{cor_temp}. Then, we just focus on the deviation estimate of $\supp_\theta f_\varepsilon$ around the solution $\theta(t)$ of the initial value problem \eqref{relax_ODE}. By an analogous argument like in {\sc Step 3} in the proof of Lemma \ref{lemma_temp}, we have the following estimates
    \begin{align*}
    \frac{d^+}{dt}(\theta_\varepsilon^M(t)-\theta(t))&\leq \frac{1}{\varepsilon}\int_{\mathbb{R}^{2d}\times \mathbb{R}_+}\zeta(D_x^\varepsilon(t))\left(\frac{1}{\theta_\varepsilon^M}-\frac{1}{\theta_*}\right) f_\varepsilon(t,z_*)\,dz_*+\left(\frac{1}{\theta_\varepsilon^M}-\frac{1}{\theta^\infty(t)}\right)-\frac{d\theta}{dt},\\
    \frac{d^+}{dt}(\theta_\varepsilon^m(t)-\theta(t))&\geq \frac{1}{\varepsilon}\int_{\mathbb{R}^{2d}\times \mathbb{R}_+}\zeta(D_x^\varepsilon(t))\left(\frac{1}{\theta_\varepsilon^m}-\frac{1}{\theta_*}\right) f_\varepsilon(t,z_*)\,dz_*+\left(\frac{1}{\theta_\varepsilon^m}-\frac{1}{\theta^\infty(t)}\right)-\frac{d\theta}{dt},
    \end{align*}
    for every $t\in [0,T]$, where we recall that
    $$\theta_\varepsilon^M(t)=\max\supp_\theta f_\varepsilon(t),\quad \theta_\varepsilon^m(t)=\min\supp_\theta f_\varepsilon(t).$$
    Now, since the first terms in the right-hand sides have the correct sign, we can neglect them. In addition, we substitute $\frac{d\theta}{dt}$ by the ODE \eqref{relax_ODE} to find
    \begin{align}\label{E-deviation-theta-slow-regular-pre}
    \begin{aligned}
    \frac{d^+}{dt}(\theta_\varepsilon^M(t)-\theta(t))&\leq \frac{1}{\theta_\varepsilon^M(t)}-\frac{1}{\theta(t)}\leq -C_\varepsilon^M(t)(\theta_\varepsilon^M(t)-\theta(t)),\\
    \frac{d^+}{dt}(\theta_\varepsilon^m(t)-\theta(t))&\geq \frac{1}{\theta_\varepsilon^m(t)}-\frac{1}{\theta(t)}\geq -C_\varepsilon^m(t)(\theta_\varepsilon^m(t)-\theta(t)),
    \end{aligned}
    \end{align} 
    for every $t\in [0,T]$. Here, we have used the uniform estimates for $\theta_\varepsilon^M$ and $\theta_\varepsilon^m$ in Lemma \ref{lemma_temp_2} and similar control $\theta_m\leq \theta(t)\leq \theta_M$ coming from the ODE \eqref{relax_ODE}, the information on $\theta_0\in [\theta_0^m,\theta_0^M]$ from \eqref{E-hypothesis-initial-data-internal-variable} and the uniform bounds \eqref{E-hypothesis-macro-species-gain-theta-infty} of $\theta^\infty(t)$ in Remark \ref{R-hypothesis-macro-species-gain}. Specifically, the time-dependent functions $C_\varepsilon^M(t)$ and $C_\varepsilon^m(t)$ are given by
    $$
    C_\varepsilon^M(t):=\left\{\begin{array}{ll}
    \frac{1}{\theta_M^2}, & \mbox{if }\theta_\varepsilon^M(t)\geq\theta(t),\\
    \frac{1}{\theta_m^2}, & \mbox{if }\theta_\varepsilon^M(t)<\theta(t),
    \end{array}\right.\quad 
    C_\varepsilon^m(t):=\left\{\begin{array}{ll}
    \frac{1}{\theta_m^2}, & \mbox{if }\theta_\varepsilon^m(t)\geq\theta(t),\\
    \frac{1}{\theta_M^2}, & \mbox{if }\theta_\varepsilon^m(t)<\theta(t).
    \end{array}\right.
    $$
    We now use Gronw{a}ll's lemma in \eqref{E-deviation-theta-slow-regular-pre} to deduce
    \begin{align}\label{E-deviation-theta-slow-regular-pre-2}
    \begin{aligned}
    \theta_\varepsilon^M(t)-\theta(t)
    &\leq \vert \theta_\varepsilon^M(t_*)-\theta(t_*)\vert e^{-\frac{t-t_*}{\theta_M^2}}\leq \vert \theta_\varepsilon^M(t_*)-\theta(t_*)\vert,\\ \theta_\varepsilon^m(t)-\theta(t)
    &\geq -\vert \theta_\varepsilon^m(t_*)-\theta(t_*)\vert e^{-\frac{t-t_*}{\theta_M^2}}\geq -\vert \theta_\varepsilon^m(t_*)-\theta(t_*)\vert,
    \end{aligned}
    \end{align}
    for any $0\leq t_*<t\leq T$. Let us define the average of internal variable
    $$\theta_\varepsilon(t):=\int_{\mathbb{R}^{2d}\times \mathbb{R}_+}\theta f_\varepsilon(t,z)\,dz,\quad t\in [0,T].$$
    Then, we have
    \begin{align*}
    \vert \theta_\varepsilon^M(t_*)-\theta(t_*)\vert&\leq \vert \theta_\varepsilon^M(t_*)-\theta_\varepsilon(t_*)\vert+\vert \theta_\varepsilon(t_*)-\theta(t_*)\vert\\
    &\leq D_\theta^\varepsilon(t_*)+\vert \theta_\varepsilon(0)-\theta(0)\vert+\left\vert\int_0^{t_*}\frac{d}{ds}(\theta_\varepsilon(s)-\theta(s))\,ds\right\vert\\
    &\leq D_\theta^\varepsilon(t_*)+\vert \theta_\varepsilon(0)-\theta(0)\vert+\frac{2}{\theta_m}t_*.
    \end{align*}
    Note that the last inequality has been obtained by computing
    $$\left\vert\frac{d}{ds}(\theta_\varepsilon(s)-\theta(s))\right\vert=\left\vert\int_{\mathbb{R}^{2d}\times \mathbb{R}_+}\frac{1}{\theta}f_\varepsilon(t,z)\,dz-\frac{1}{\theta(t)}\right\vert\leq \frac{2}{\theta_m},$$
    for any $s\in (0,T)$, where we have used the kinetic equation \eqref{A-6} and the relaxation ODE \eqref{relax_ODE} to derive above estimate. Setting $t_*=\varepsilon^\beta$, using the above polynomial control of $D_\theta^\varepsilon(t_*)$ and hypothesis \eqref{E-hypothesis-initial-data-internal-variable} and putting everything together into \eqref{E-deviation-theta-slow-regular-pre-2} yields
    $$\theta_\varepsilon^M(t)-\theta(t)\leq C(\varepsilon^\alpha+\varepsilon^\beta+\varepsilon^\gamma),$$
    for each $t\in [\varepsilon^\beta,T]$. Since $\alpha>0$ is arbitrary, this ends the upper bound and a similar lower estimates for $\theta_\varepsilon^m(t)-\theta(t)$ concludes the proof.
\end{proof}

To summarize, the previous results show that after a small initial time layer $\e_0=\e^\beta$, the diameter of $\textup{supp}_\theta f_\e(t)$ is smaller than $C\e^\alpha$ and its distance to $\theta(t)$ is smaller than $C\e^{\min\{\beta,\gamma\}}$. This justifies that $f_\e$ will concentrate around $\theta=\theta(t)$, as $\e\rightarrow 0$, although the relaxation to the background mean value $\theta^\infty(t)$ does not occurs instantaneously (as in Section \ref{sec:2}), but it actually occurs slowly according to the relaxation ODE \eqref{relax_ODE}. 

For clarity, we sketch in Figure \ref{fig1} the main difference in the dynamics of the internal variable support for both the strong and weak relaxation regimes. On the one hand, blue and red lines denote the relaxation of internal variable support in the strong and weak relaxation regimes respectively. Dashed lines are used for the background mean value $\theta^\infty(t)$, whilst dotted lines represent the solution $\theta(t)$ of the relaxation ODE \eqref{relax_ODE}. Finally, the vertical dotted line determines the initial time layer $\e_0=\varepsilon^\beta$, after which the internal variable support becomes very narrow with the size of order $\mathcal{O}(\e^\alpha)$ and $\mathcal{O}(\e^{\min\{\beta,\gamma\}})$ respectively.

\begin{figure}[h!]
	\includegraphics[width=0.6\textwidth]{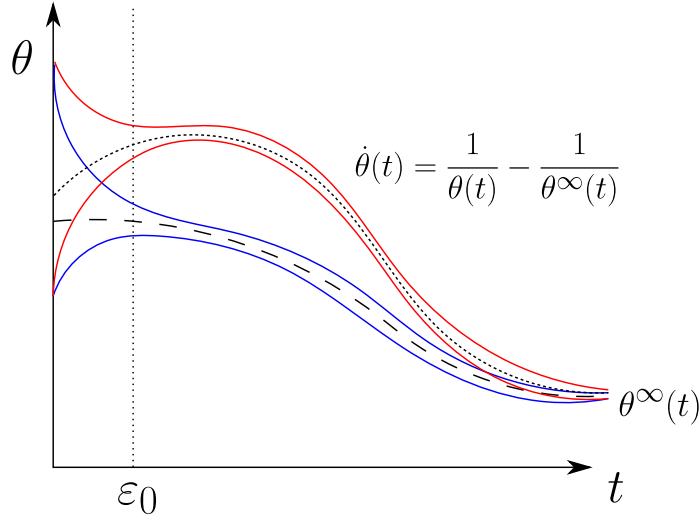}
	\caption{Relaxation of internal variable in the weak and strong regimes.}
	\label{fig1}
\end{figure}

\subsection{Moment estimates}

In this part, we derive again a priori estimates for the moments of $f_\varepsilon$. We notice that under the hypothesis in the preceding part guaranteeing the concentration of the internal variable support, the same arguments in Section \ref{sec:2} can be applied to derive the estimates. In particular, an analogue of Corollary \ref{C2.3} holds true.
	
	\begin{corollary}[Velocity and position moments]\label{C-a-priori-slow-regular}
		Let $f_\e$ be the solution of \eqref{A-4} subject to initial data $f^0_\e$. Suppose that initial data verify \eqref{E-hypothesis-initial-data}, $(\bar \rho,\bar u,\bar e)$ verifies \eqref{E-hypothesis-macro-species} and parameters fulfill the compatibility condition \eqref{E-hypothesis-thetaM-thetam}. 
		Then,
		\begin{align}\label{E-moment-estimates-slow}
		\begin{aligned}
		\Vert \vert x\vert f_\varepsilon\Vert_{L^\infty(0,T;L^1(\mathbb{R}^{2d}\times \mathbb{R}_+))}&\leq M_0+T^{1/2}\Vert \vert v\vert f_\varepsilon\Vert_{L^2(0,T;L^1(\mathbb{R}^{2d}\times\bbr_+))},\\ 
		\||v|f_\e\|_{L^2(0,T;L^1(\bbr^{2d}\times\bbr_+))}& \le\left(\||v|^2f_\e\|_{L^1(0,T;L^1(\bbr^{2d}\times\bbr_+))}\right)^{1/2},\\
		\||v|^2f_\e\|_{L^1(0,T;L^1(\bbr^{2d}\times\bbr_+))}&\le C\left(\e E_0+C G_0^2\right),
		\end{aligned}
		\end{align}
		for every $\varepsilon>0$. In addition, the  estimate 
		\begin{equation}\label{E-dissipation-estimate-slow}
		\frac{1}{\theta_M}\int_0^T\int_{\mathbb{R}^{4d}\times \mathbb{R}_+^2}\phi(x-x_*)\vert v-v_*\vert^2\,f_\varepsilon(t,z)f_\varepsilon(t,z_*)\,dz\,dz_*\,dt\leq \varepsilon E_0+C G_0^2
		\end{equation}
		holds.
	\end{corollary}
	
	Indeed, notice that by Lemma \ref{lem_vel_supp} both spatial and velocity supports have uniformly bounded radii. Then, it is apparent that \eqref{E-moment-estimates-slow} can be improved into
	\begin{align*}
	\Vert \vert x\vert f_\varepsilon\Vert_{L^\infty(0,T;L^1(\mathbb{R}^{2d}\times\mathbb{R}_+))}&\leq \sup_{0\leq t\leq T}R_x^\varepsilon(t)\leq r_M^0+v_MT,\\
	\Vert \vert v\vert f_\varepsilon\Vert_{L^\infty(0,T;L^1(\mathbb{R}^{2d}\times\mathbb{R}_+))}&\leq \sup_{0\leq t\leq T}R_v^\varepsilon(t)\leq v_M,\\
	\Vert \vert v\vert^2 f_\varepsilon\Vert_{L^\infty(0,T;L^1(\mathbb{R}^{2d}\times\mathbb{R}_+))}&\leq \sup_{0\leq t\leq T}R_v^\varepsilon(t)^2\leq v_M^2.
	\end{align*}
	Since they will not be necessary, we stick to the initial (weaker) estimates as in Section \ref{sec:2}.

\subsection{Hydrodynamic limit of \eqref{A-6}}

In this subsection, we present the rigorous hydrodynamic limit of \eqref{A-6}, whose explicit statement takes the following form:
	
	\begin{theorem}\label{T4.1}
		Let $f_\e$ be the solution of \eqref{A-6} subject to initial data $f^0_\e$. Assume that the initial data verify \eqref{E-hypothesis-initial-data}, $(\bar \rho,\bar u,\bar e)$ verifies \eqref{E-hypothesis-macro-species} and the parameters fulfill the compatibility condition \eqref{E-hypothesis-thetaM-thetam}. In addition, assume the initial uniform confinement \eqref{E-hypothesis-initial-data-uniform-compact-support} and the initial asymptotic control \eqref{E-hypothesis-initial-data-internal-variable}, for some value $\theta_0\in \mathbb{R}_+$ and some parameter $\gamma>0$. Then,
		\begin{align*}
		&\rho_\varepsilon\rightarrow\rho,\quad \mbox{in }C([0,T],\mathcal{M}(\mathbb{R}^{d})-\mbox{narrow}),\\
		&j_\varepsilon\overset{*}{\rightharpoonup} j,\quad \mbox{in }L^2_w(0,T;\mathcal{M}(\mathbb{R}^d)^d),
		\end{align*}
		when $\varepsilon\rightarrow 0$, for some probability measure $\rho$, some finite Radon measure $j$ and some subsequence of $\{\rho_\varepsilon\}_{\varepsilon>0}$ and $\{j_\varepsilon\}_{\varepsilon>0}$ that we denote in the same way. In addition, $(\rho,j,\theta(t))$ solves the following problem
		$$
		\begin{array}{ll}
		\displaystyle \partial_t \rho+\nabla\cdot j =0, & (t,x)\in [0,T)\times \mathbb{R}^d,\\
		\displaystyle j-\frac{\theta(t)u^\infty(t)}{\theta^\infty(t)}\rho+\theta(t)(\nabla W*\rho)\rho=\rho (\phi*j)-j(\phi*\rho), & (t,x)\in (0,T)\times \mathbb{R}^d,\\
		\displaystyle \frac{d\theta(t)}{dt}=\frac{1}{\theta(t)}-\frac{1}{\theta^\infty(t)}, & t\in [0,T],\\
		\rho(t=0)=\rho^0, & x\in \mathbb{R}^d,\\
		\theta(0)=\theta_0, & 
		\end{array}
		$$
		in distributional sense, where $\theta^\infty(t)$ and $u^\infty(t)$ are the background mean values of the internal variable and velocity, see \eqref{E-background-mean-theta} and \eqref{E-background-mean-velocity}.
	\end{theorem}
	
\begin{proof}
Thanks to the a priori estimates in Corollary \ref{C-a-priori-slow-regular}, we readily infer \eqref{E-a-priori-estimates-fast-regular}. Therefore, we use the same arguments as in the proof of Theorem \ref{T2.1} in Section \ref{sec:2} to find a probability measure $\rho\in C([0,T];\mathcal{M}(\mathbb{R}^d)-\mbox{narrow})$ and a finite Radon measure $j\in L^2_w(0,T;\mathcal{M}(\mathbb{R}^d)^d)$ so that we recover the convergence results in \eqref{E-compactness-fast-regular} and \eqref{E-compactness-fast-regular-gain}, namely,
        \begin{align*}
		&\rho_\varepsilon\rightarrow\rho,\quad \mbox{in }C([0,T],\mathcal{M}(\mathbb{R}^{d})-\mbox{narrow}),\\
		&j_\varepsilon\overset{*}{\rightharpoonup} j,\quad \mbox{in }L^2_w(0,T;\mathcal{M}(\mathbb{R}^d)^d).
		\end{align*}
Now, we use the deviation estimate of $\supp_\theta f_\varepsilon$ around $\theta(t)$ in \eqref{C-initial-time-layer-slow-regular} and we reproduce similar ideas as in the proofs of Lemmas \ref{L2.3} and \ref{L2.4} to find that, for every $\varepsilon_0>0$
        \begin{align*}
        &h_\e \rightarrow \theta(t)\rho,\quad \mbox{in}\quad C([\varepsilon_0,T],\mathcal{M}(\bbr^d)-\mbox{narrow}),\\
        &S_\e^\theta\xrightharpoonup{*} \theta(t)j,\quad \mbox{in}\quad L^2_w({\e_0},T;\mathcal{M}(\bbr^d)^d),\\
		&A_\e\xrightharpoonup{*} \frac{j}{\theta(t)},\quad \mbox{in}\quad L^2_w({\e_0},T;\mathcal{M}(\bbr^d)^d),\\
		&B_\e \rightarrow \frac{\rho}{\theta(t)},\quad \mbox{in}\quad C([\varepsilon_0,T],\mathcal{M}(\bbr^d)-\mbox{narrow}),\\
		&\rho_\e\otimes A_\e \xrightharpoonup{*} \frac{\rho\otimes j}{\theta(t)},\quad \mbox{in}\quad L^2_w(\e_0,T;\mathcal{M}(\bbr^{2d})),\\
		&\rho_\e\otimes\rho_\e\rightarrow \rho\otimes \rho,\quad \mbox{in}\quad C([0,T],\mathcal{M}(\bbr^{2d})-\mbox{narrow}),
		\end{align*}
		as $\varepsilon\rightarrow 0$. We emphasize that $\supp_\theta f_\varepsilon$ now concentrates asymptotically around $\theta(t)$, whence the dependence of limits on $\theta(t)$ instead of $\theta^\infty(t)$. We are now ready to pass to the limit in the hierarchy of equations \eqref{E-hierarchy-moments-equations-slow}. On the one hand, passing to the limit in the density and velocity equation follows the same train of thoughts as in the strong relaxation regime in Section \ref{sec:2} and yields the system of equations
		$$
		\begin{array}{ll}
		\displaystyle \partial_t \rho+\nabla\cdot j =0, & t\in [0,T)\times \mathbb{R}^d,\\
		\displaystyle j-\frac{\theta(t)u^\infty(t)}{\theta^\infty(t)}\rho+\theta(t)(\nabla W*\rho)\rho=\rho (\phi*j)-j(\phi*\rho), & t\in (0,T)\times \mathbb{R}^d,\\
		\rho(t=0)=\rho^0, & x\in \mathbb{R}^d.
		\end{array}
		$$
		Then, we focus on the limit of the internal variable equation in \eqref{E-hierarchy-moments-equations-slow}. In weak form it reads
		\begin{align*}
		\int_0^T\int_{\mathbb{R}^d}\partial_t\psi \,h_\varepsilon\,dx\,dt&+\int_0^T\int_{\mathbb{R}^d}\mathcal{S}_\varepsilon^\theta\,\nabla\psi\,dx\,dt\nonumber\\
    	&=\frac{1}{\varepsilon}\int_0^T\int_{\mathbb{R}^d}\psi(\zeta*\rho_\varepsilon)B_\varepsilon\,dx\,dt-\frac{1}{\varepsilon}\int_0^T\int_{\mathbb{R}^d}\psi(\zeta*B_\varepsilon)\rho_\varepsilon\,dx\,dt\\
    	&\qquad +\int_0^T\int_{\mathbb{R}^d}\psi B_\varepsilon\,dx\,dt-\int_0^T\int_{\mathbb{R}^d}\psi\frac{\rho_\varepsilon}{\theta^\infty(t)}\,dx\,dt,
    	\end{align*}
    	for any $\psi\in C^1_c((0,T)\times \mathbb{R}^d)$. The limit of the linear terms is clear. We show again that the nonlinear term vanishes, as $\varepsilon\rightarrow 0$, due to a similar argument supported by concentration of the internal variable support. Specifically, set any $\varepsilon_0>0$ so that $\supp\psi\subseteq [\varepsilon_0,T]\times \mathbb{R}^d$. Then, 
    	\begin{multline*}
    	\left\vert\frac{1}{\varepsilon}\int_0^T\int_{\mathbb{R}^d}\psi(\zeta*\rho_\varepsilon)B_\varepsilon\,dx\,dt-\frac{1}{\varepsilon}\int_0^T\int_{\mathbb{R}^d}\psi(\zeta*B_\varepsilon)\rho_\varepsilon\,dx\,dt\right\vert\\
    	\leq \frac{1}{\varepsilon}\Vert \psi\Vert_{C_c((0,T)\times \mathbb{R}^d)}\int_{\varepsilon_0}^T\int_{\mathbb{R}^{4d}\times \mathbb{R}^2_+}\zeta(x-x_*)\left\vert\frac{1}{\theta}-\frac{1}{\theta_*}\right\vert f_\varepsilon(t,z)f_\varepsilon(t,z_*)\,dz\,dz_*\,dt\\
    	\leq \frac{1}{\varepsilon}\frac{\Vert \zeta\Vert_{L^\infty}\Vert \psi\Vert_{C_c((0,T)\times \mathbb{R}^d)}}{\theta_m^2}\int_{\varepsilon_0}^T D_\theta^\varepsilon(t)\,dt\leq C\frac{\Vert \zeta\Vert_{L^\infty}\Vert \psi\Vert_{C_c((0,T)\times \mathbb{R}^d)}}{\theta_m^2}T\varepsilon^{\alpha-1},
    	\end{multline*}
    	for any $\varepsilon\leq \varepsilon_0^{1/\beta}$, where we have used the concentration estimate in Corollary \ref{C-initial-time-layer-slow-regular}. Since $\alpha$ can be taken larger than $1$, the right-hand side tends to $0$, as $\varepsilon\rightarrow 0$. Therefore, after passing to the limit, we obtain the weak equation
    	\begin{align*}
		\int_0^T\int_{\mathbb{R}^d}\partial_t\psi(t,x)\theta(t)\rho(t,dx)\,dt&+\int_0^T\int_{\mathbb{R}^d}\theta(t)j(t,dx)\cdot\nabla\psi(t,x)\,dt\\
    	&=\int_0^T\int_{\mathbb{R}^d} \frac{\psi(t,x)}{\theta(t)}\rho(t,dx)\,dt-\int_0^T\int_{\mathbb{R}^d}\frac{\psi(t,x)}{\theta^\infty(t)}\rho(t,dx)\,dt,
    	\end{align*}
    	for any $\psi\in C^1_c((0,T)\times \mathbb{R}^d)$, that is,
    	$$\partial_t(\theta(t)\rho)+\nabla\cdot(\theta(t)j)+\left(\frac{1}{\theta(t)}-\frac{1}{\theta^\infty}\right)\rho=0.$$
    	Using the continuity equation $\partial_t\rho+\nabla_x\cdot j=0$, such an equation amounts to
        \[\frac{d\theta(t)}{dt}+\left(\frac{1}{\theta(t)}-\frac{1}{\theta^\infty}\right)=0,\]
        which is exactly the same as the relaxation ODE \eqref{relax_ODE} for $\theta(t)$.
\end{proof}

\subsection{Weakly singular influence functions}\label{subsec:4.5}

	Again, we notice that the uniform bound of the influence functions $\phi$, $\zeta$ and the aggregation potential $\nabla W$ has strongly be used in some parts of the preceding proofs. Hence, the hydrodynamic limit of the system \eqref{E-kinetic-singular-slow}, where the influence functions $\phi$ and $\zeta$ are replaced by $\phi_\varepsilon$ and $\zeta_\varepsilon$ in \eqref{E-coupling-weights-epsilon}, requires appropriate modifications of the arguments in Theorem \ref{T4.1} leading to the following version.

		\begin{theorem}
		Let $f_\e$ be a solution to the equation \eqref{E-kinetic-singular-slow} subject to initial data $f^0_\e$ with parameters $\lambda_1\in(0,1)$, and $\lambda_2>0$. Assume that the initial data verify \eqref{E-hypothesis-initial-data}, $(\bar \rho,\bar u,\bar e)$ verifies \eqref{E-hypothesis-macro-species} and the  parameters fulfill the stronger compatibility condition \eqref{E-hypothesis-thetaM-thetam}. In addition, assume the initial uniform confinement \eqref{E-hypothesis-initial-data-uniform-compact-support} and the initial asymptotic control \eqref{E-hypothesis-initial-data-internal-variable}, for some value $\theta_0\in \mathbb{R}_+$ and some parameter $\gamma>0$. Then,
		\begin{align*}
		&\rho_\varepsilon\rightarrow\rho,\quad \mbox{in }C([0,T],\mathcal{M}(\mathbb{R}^{d})-\mbox{narrow}),\\
		&j_\varepsilon\overset{*}{\rightarrow} j,\quad \mbox{in }L^2_w(0,T;\mathcal{M}(\mathbb{R}^d)^d),
		\end{align*}
		as $\varepsilon\rightarrow 0$, for some probability measure $\rho$, some finite Radon measure $j$ and some subsequence of $\{\rho_\varepsilon\}_{\varepsilon>0}$ and $\{j_\varepsilon\}_{\varepsilon>0}$ that we denote in the same way. In addition, $(\rho,j,\theta(t))$ solves the following problem
		$$
		\begin{array}{ll}
		\displaystyle \partial_t \rho+\nabla\cdot j =0, & (t,x)\in [0,T)\times \mathbb{R}^d,\\
		\displaystyle j-\frac{\theta(t)u^\infty(t)}{\theta^\infty(t)}\rho+\theta(t)(\nabla W*\rho)\rho=\rho (\phi*j)-j(\phi*\rho), & (t,x)\in (0,T)\times \mathbb{R}^d,\\
		\displaystyle \frac{d\theta(t)}{dt}=\frac{1}{\theta(t)}-\frac{1}{\theta^\infty(t)}, & t\in [0,T],\\
		\rho(t=0)=\rho^0, & x\in \mathbb{R}^d,\\
		\theta(0)=\theta_0, & 
		\end{array}
		$$
		in distributional sense, for an appropriate meaning of the nonlinear term. Here, $\phi_0$ is the singular influence function in \eqref{E-singular-kernels} and, $\theta^\infty(t)$ and $u^\infty(t)$ are the background mean value of the internal variable and velocity, see \eqref{E-background-mean-theta} and \eqref{E-background-mean-velocity}.
	\end{theorem}
	
	Since the proof follows similar modifications like in Theorem \ref{T-hydro-fast-singular} exploiting the initial strong concentration assumption \eqref{E-hypothesis-thetaM-thetam-strong}, we omit details here and just comment on the main difficulties. On the one hand, note that the same arguments as in Subsection \ref{subsec:3.5} guarantee that the initial concentration assumption \eqref{E-hypothesis-thetaM-thetam-strong} in place of the weak compatibility condition \eqref{E-hypothesis-thetaM-thetam} solves the main delicate points. Indeed, such a condition allows compensating extra factors $\frac{1}{\varepsilon}$ coming from the singularity of the influence function $\phi_\e$, under the cost that a stronger initial concentration of the internal variable support is assumed. On the other hand, an analogous dissipation estimate like \eqref{E-dissipation-estimate} holds so that the non-concentration property \eqref{commut_A} persists and it allows extending the result to the limiting regime $\lambda_1=1$ like in Subsection \eqref{subsec:3.5}. Unfortunately, the uniform bound of $\nabla W$ in \eqref{H-hypothesis-phi-zeta-W} has been applied in an essential way in the proof if Lemma \ref{lem_vel_supp} in order to control (uniformly in $\e$) the growth of spacial and velocity supports. Hence, the case of singularly scaled potentials $W_\e$ like in \eqref{E-Cucker-Dong-potentials-scaled} cannot be addressed with our method and would require a significantly different approach and analysis than the present one.

    \section{Numerical simulations}
	In this section, we present the numerical simulations for the limiting macroscopic equations \eqref{A-5} and \eqref{A-7}, both for strong and weak relaxation regimes. For computational efficiency, we choose the 1-dimensional periodic spatial domain $\mathbb{T}$. However, similar results could be shown for compactly supported initial data in the free space $\bbr$. Along this part we shall use the following notation:
	\begin{enumerate}
	\item {\bf (Lie group)} We regard $\mathbb{T}$ as the quotient of the interval $[0,1]$ under identification of the endpoints $0$ and $1$. It has a Lie group structure under the group operation
	$$x+y:=[a+b+n],$$
	for any equivalent classes $x=[a]$, $y=[b]\in \mathbb{T}$ with $a,b\in [0,1]$, where $n\in \mathbb{Z}$ is the only integer such that $a-b+n\in (0,1]$. Hence, note that the identity element is $[0]$ and the inverse element of $x$ is given by 
	$$-x:=[1-a].$$
	\item {\bf (Geodesic distance)} Given two equivalence classes $x=[a]$, $y=[b]\in \mathbb{T}$ for some numbers $a,b\in [0,1]$, we define the distance over $\mathbb{T}$ by
	$$\vert x-y\vert_{\mathbb{T}}:=\vert a-b+n\vert,$$
	where $n\in \mathbb{Z}$ is the only integer such that $a-b+n\in \left(-\frac{1}{2},\frac{1}{2}\right]$. It is nothing than the usual geodesic distance associated to the Lie group. From here on, we will denote $\vert x-y\vert:=\vert x-y\vert_{\mathbb{T}}$ and $\vert x\vert=\vert x-[0]\vert_\mathbb{T}$ for simplicity when there is no confusion.
	\item {\bf (Convolution)} Given two periodic functions $f,g:\mathbb{T}\longrightarrow \mathbb{R}$, we define its convolution $f*g:\mathbb{T}\longrightarrow \mathbb{R}$, as the usual Lie group convolution
	$$(f*g)(x)=\int_{\mathbb{T}}f(x-y)g(y)\,dy=\int_0^1 f([a]-[b])g([b])\,db,\quad x=[a]\in \mathbb{T}.$$
	\end{enumerate}

	\subsection{Strong relaxation regime}\label{subsec:5.1}
	In this part, we present some numerical simulations of the limiting macroscopic system \eqref{A-5} under the strong relaxation regime. In periodic variables such a system is determined by the following coupled equations for $(\rho,u)$:
	\begin{align}\label{eq_rho_u_strong}
	\begin{aligned}
	&\partial_t\rho +\nabla\cdot (\rho u) = 0, & & (t,x)\in \mathbb{R}_+\times \mathbb{T}, \\
	& u -u^\infty(t)+\theta^\infty(t)(\nabla W*\rho)=\phi*(\rho u)-(\phi*\rho)u. & & (t,x)\in \mathbb{R}_+\times \mathbb{T}.
	\end{aligned}
	\end{align}
	Here, the relaxation velocity $u^\infty(t)\in \mathbb{R}^d$ and the internal variable $\theta^\infty(t)\in \mathbb{R}_+$ are
	$$u^\infty(t)=\theta^\infty(t)\int_{\mathbb{T}}\frac{\bar{\rho}(t,x)\bar{u}(t,x)}{\bar{e}(t,x)}\,dx,\quad \theta^\infty(t)=\left(\int_{\mathbb{T}}\frac{\bar{\rho}(t,x)}{\bar{e}(t,x)}\,dx\right)^{-1}.$$
	These are determined instantaneously by the background fluid $(\bar \rho,\bar u,\bar e)$, which is governed by the hydrodynamic equations:
	\begin{align}
	\begin{aligned}\label{eq_background}
	&\partial_t \bar{\rho} +\nabla\cdot(\bar{\rho}\bar{u})=0,\\
	&\partial_t (\bar{\rho}\bar{u}) +\nabla\cdot (\bar{\rho} \bar{u}\otimes \bar{u}) =\int_{\mathbb{T}}\phi(x-x_*)\left(\frac{\bar{u}(t,x_*)}{\bar{e}(t,x_*)}-\frac{\bar{u}(t,x)}{\bar{e}(t,x)}\right)\bar{\rho}(t,x)\bar{\rho}(t,x_*)\,dx_*,\\
	&\partial_t (\bar{\rho}\bar{e}) +\nabla\cdot (\bar{\rho} \bar{u}\bar{e}) =\int_{\mathbb{T}}\zeta(x-x_*)\left(\frac{1}{\bar{e}(t,x)}-\frac{1}{\bar{e}(t,x_*)}\right)\bar{\rho}(t,x)\bar{\rho}(t,x_*)\,dx_*.
	\end{aligned}
	\end{align}
	We choose $\lambda_1=\lambda_2=1$ so that the influence functions $\phi$ and $\zeta$ are given by 
	$$\phi(x)=\zeta(x)=\frac{1}{(1+3\vert x\vert^2)^\frac{1}{2}},\quad x\in \mathbb{T}.$$
	Finally, we choose the aggregation potential $W$ as
	$$
	W(x) = \eta(\vert x\vert)\left(\frac{1}{2}\log(1+|x|^2)-\frac{1}{2}\log\frac{5}{4}\right),\quad x\in\mathbb{T},
	$$
	where $\eta=\eta(r)$ is a smooth bump function with $\eta(r) = 1$ for $0\leq r\leq \frac{1}{6}$ and $\eta(r) = 0$ for $r\geq \frac{1}{3}$. We note that the above $W$ is a modification of the potential $\frac{1}{2}\log(1+|x|^2)$ introduced in \eqref{E-Cucker-Dong-potentials} with $\lambda_3=1$, so that it is smoothly extended to the periodic domain.

	\medskip
	
	$\bullet$ {\sc Initial data}. We choose the initial data for the background fluid $(\bar{\rho},\bar{u},\bar{e})$ as
	\begin{equation}\label{E-numerics-initial-data-macro-species}
	\bar{\rho}^0(x)\equiv 1,\quad \bar{u}^0(x) = 0.5+\sin(2\pi x),\quad \bar{e}^0(x) = 2+\cos(2\pi x),\quad x\in\mathbb{T},
	\end{equation}
	while the initial data for $\rho$ is
	\begin{equation}\label{E-numerics-initial-data-limit}
	\rho^0(x) = \frac{1}{Z}\exp\left(-50\left\vert x-[0.5]\right\vert^2\right),\quad x\in\mathbb{T},
	\end{equation}
	for some normalizing coefficient $Z$ so that $\int_{\mathbb{T}}\rho^0(x)\,dx=1$.
	
	\medskip
	
	$\bullet$ {\sc Discussion of numerical methods}. For the sake of completeness, we explain here the numerical methods that we have used to solve numerically both the hydrodynamic  equations \eqref{eq_background} and the limiting macroscopic system \eqref{eq_rho_u_strong}. Given spacial and temporal grid sizes $\Delta x=\frac{1}{M}$ and $\Delta t=\frac{T}{N}$ with $M,N\in \mathbb{N}$, we discretize the spatial variable $x\in \mathbb{T}$ by nodes $x_0,\ldots,x_M\in \mathbb{T}$ with $\vert x_{j+1}-x_j\vert =\frac{1}{M}$, for $j=0,\ldots,M-1$, and the temporal variable $t\in [0,T]$ by nodes $t_0,\ldots,t_N$ with $t_n:=n\Delta t$, for $n=0,\ldots,N$. For any function $w=w(t,x)$ for $t\in [0,T]$ and $x\in \mathbb{T}$, we denote its numerical approximation by
	$$w^n_j\approx w(t_n,x_j).$$
	Considering intermediate nodes, we can determine a staggered grid and we similarly denote 
	$$w^{n+\frac{1}{2}}_j\approx  w(t_{n+\frac{1}{2}},x_j),\quad w^n_{j+\frac{1}{2}}\approx  w(t_n,x_{j+\frac{1}{2}}),\quad w^{n+\frac{1}{2}}_{j+\frac{1}{2}}\approx w(t_{n+\frac{1}{2}},x_{j+\frac{1}{2}}).$$
	
	\medskip
	
	$\diamond$ On the one hand, the numerical solution $(\bar\rho^{n+1},\bar u^{n+1},\bar e^{n+1})$ of \eqref{eq_background} is computed from $(\bar \rho^n,\bar u^n,\bar e^n)$ by using the central difference method on the staggered grid for non-homogeneous balance law suggested by Nessyahu and Tadmor \cite{NT90} (see also \cite[Section 4]{LRR00}). 
	We briefly recall the scheme in the simpler case of a 1-dimensional scalar balance law
	\[\partial_t w +\partial_xf(w) =g(w).\]
    Assume that the numerical solution $w^n\approx w(t_n,\cdot)$ at the time step $n$ is given. Then, the (generalized) Nessyahu-Tadmor scheme computes the solution $w^{n+1}\approx w(t_{n+1},\cdot)$ at the next time step on the staggered grid as
	\begin{align*}
	w^{n+1}_{j+\frac{1}{2}}&=\frac{1}{2}(w^n_j+w^n_{j+1})+\frac{1}{8}({w^n_j}'-{w^n_{j+1}}')+\frac{\Delta t}{\Delta x}\left(f\left(w_j^{n+\frac{1}{2}}\right)-f\left(w_{j+1}^{n+\frac{1}{2}}\right)\right)\\
	&\quad+\frac{\Delta t}{2}\left(g\left(w_{j+1}^{n+\frac{1}{2}}\right)+g\left(w_{j}^{n+\frac{1}{2}}\right)\right),
	\end{align*}
	where the predictor $w^{n+\frac{1}{2}}_j$ is given by
	\[w^{n+\frac{1}{2}}_j = w^n_j+\frac{\Delta t}{2}\left(g(w^n_j)-\frac{{f^n_j}'}{\Delta x}\right).\]
	Here, ${w^n_j}'$ and ${f^n_j}'$ are first-order approximations of spatial derivative, computed as follows
	\begin{align*}
	{w^n_j}'&:=\textup{minmod}(w^n_{j+1}-w^n_j, w^n_j-w^n_{j-1}),\\ {f^n_j}'&:=\textup{minmod}(f(w^n_{j+1})-f(w^n_j),f(w^n_j)-f(w^n_{j-1})),
	\end{align*}
	where $\textup{minmod}$ is the minmod function:
	\[\textup{minmod}(x,y):=\begin{cases}\textup{sgn}(x)\min(|x|,|y|),\quad &\mbox{if}\quad  \textup{sgn}(x)=\textup{sgn}(y),\\
	0,\quad&\mbox{otherwise}.\end{cases}\]
	
	$\diamond$ On the other hand, the solution $(\rho^{n+1},u^{n+1})$ to \eqref{eq_rho_u_strong} is computed from $(\rho^n,u^n)$ as follows. First, from known $(\rho^n,u^n)$ we can compute $\rho^{n+1}$ by solving the initial value problem associated with the continuity equation
	\begin{equation}\label{eq_rho}
	\begin{array}{ll}
	\partial_t\rho+\nabla\cdot (\rho u)=0, & (t,x)\in \mathbb{R}_+\times \mathbb{T},\\
	\rho(0,\cdot)=\rho^0(x), & x\in \mathbb{T},
	\end{array}
	\end{equation}
	driven by $u=u^n$ and issued at $\rho^n$. To solve \eqref{eq_rho}, we use the well-known fifth-order weighted essentially non-oscillatory scheme (WENO5) \cite{S98} for discretizing the spatial variable, and use the third-order total variation diminishing Runge-Kutta scheme (TVD-RK3) \cite{S88} for the time integration. We refer to the recent survey paper \cite{S20} for further details. Second, once $\rho^{n+1}$ is known, $u^{n+1}$ can be recovered from the linear implicit equation
	\begin{equation}\label{eq_u}
	u-u^\infty(t)+\theta^\infty(t)(\nabla W*\rho)=\phi*(\rho u)-(\phi*\rho)u,
	\end{equation}
	for $\rho=\rho^{n+1}$ and $u^\infty(t_{n+1})$. Specifically, we approximate the convolution on the left-hand side and the integral on the right-hand side by using simple rectangle rule. Then, the equation \eqref{eq_u} for $u$ is approximated as
    \[u_i^{n+1}-u^\infty +\sum_{j=1}^M\nabla W(x_i-x_j)\rho_j^{n+1}\Delta x= \sum_{j=1}^M \phi(x_i-x_j)(u_j^{n+1}-u_i^{n+1})\rho_j^{n+1}\Delta x,\]
    or, in vector form,
    \begin{equation}\label{lineq}
    (I_M-\Delta x\,\Phi^{n+1})u^{n+1} = u^\infty(t_{n+1})\mathbbm{1}-\Delta x\mathcal{W} \rho^{n+1}.
    \end{equation}
    Here, $\rho^{n+1} = (\rho_1^{n+1},\ldots,\rho_M^{n+1})^\top $, $u^{n+1}=(u_1^{n+1},\ldots, u_M^{n+1})^\top$, $\mathbbm{1}:=(1,1,\ldots,1)^\top$ and the $(i,j)$-component of the matrices $\Phi^{n+1}\in \bbr^{M\times M}$ and $\mathcal{W}$ are defined as
    \begin{align*}
    &\Phi_{ij}^{n+1}:= \phi(x_i-x_j) \rho_j^{n+1}-\delta_{ij}\left(\sum_{k=1}^M \phi(x_i-x_k) \rho_k^{n+1}\right),\\
    &\mathcal{W}_{ij}:=\nabla W(x_i-x_j).
    \end{align*}
    Therefore, $u^{n+1}$ can be obtained by solving the linear algebraic system \eqref{lineq}. We use the MATLAB internal function.
    
    \medskip

    To summarize, we compute the numerical solution through the following protocol: given numerical solutions $(\rho^n,u^n)$ and $(\bar\rho^n,\bar u^n,\bar e^n)$ at time step $n$,
    
    \begin{enumerate}
    \item Compute $(\bar\rho^{n+1},\bar u^{n+1},\bar e^{n+1})$ and $u^\infty(t_{n+1})$ by solving \eqref{eq_background} from $(\bar\rho^n,\bar u^n,\bar e^n)$.
    \item Compute $\rho^{n+1}$ by solving \eqref{eq_rho} with $u=u^n$ from $\rho^n$. 
    \item Compute $u^{n+1}$ by solving \eqref{eq_u} with $\rho=\rho^{n+1}$ and $u^\infty(t_{n+1})$.
    \end{enumerate}

	\begin{figure}[h!]
		\centering
		\begin{subfigure}[b]{1\textwidth}
			\includegraphics[width=\textwidth]{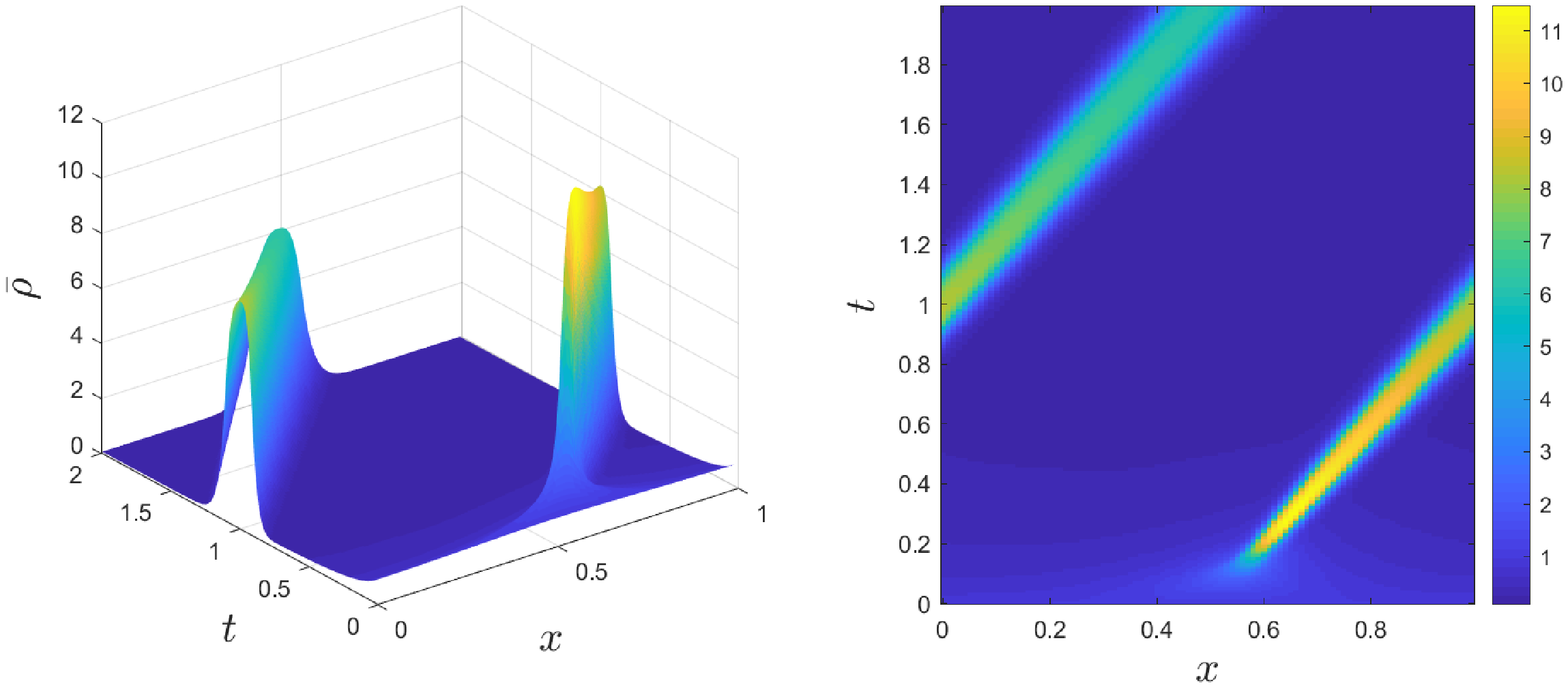}
			\caption{Profile of the density $\bar{\rho}$ of background fluid.}
		\end{subfigure}
		
		\begin{subfigure}[b]{0.45\textwidth}
			\includegraphics[width=\textwidth]{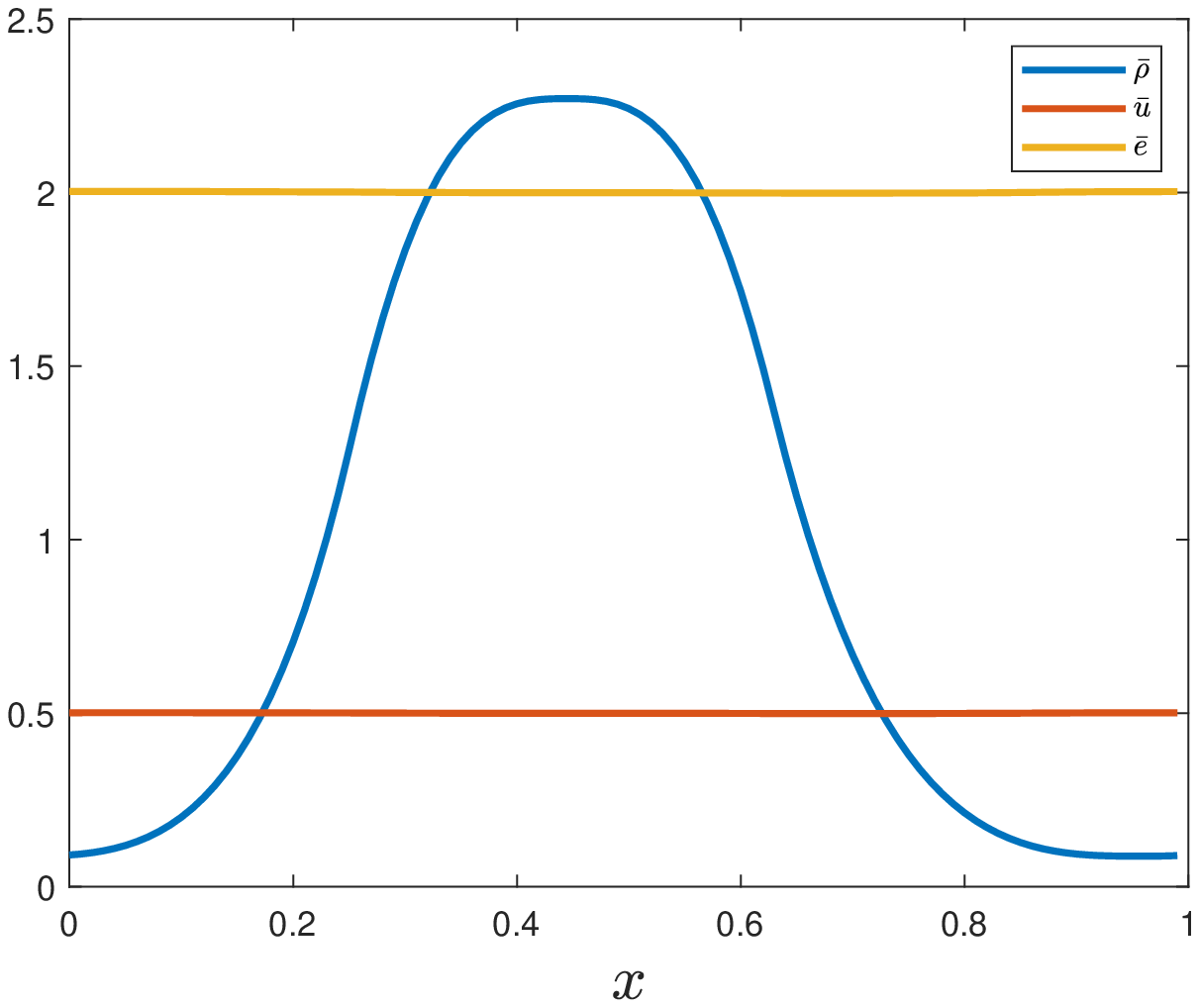}
			\caption{$(\bar{\rho},\bar{u},\bar{e})$ at $t=20$.}
		\end{subfigure}
		\begin{subfigure}[b]{0.45\textwidth}
			\includegraphics[width=\textwidth]{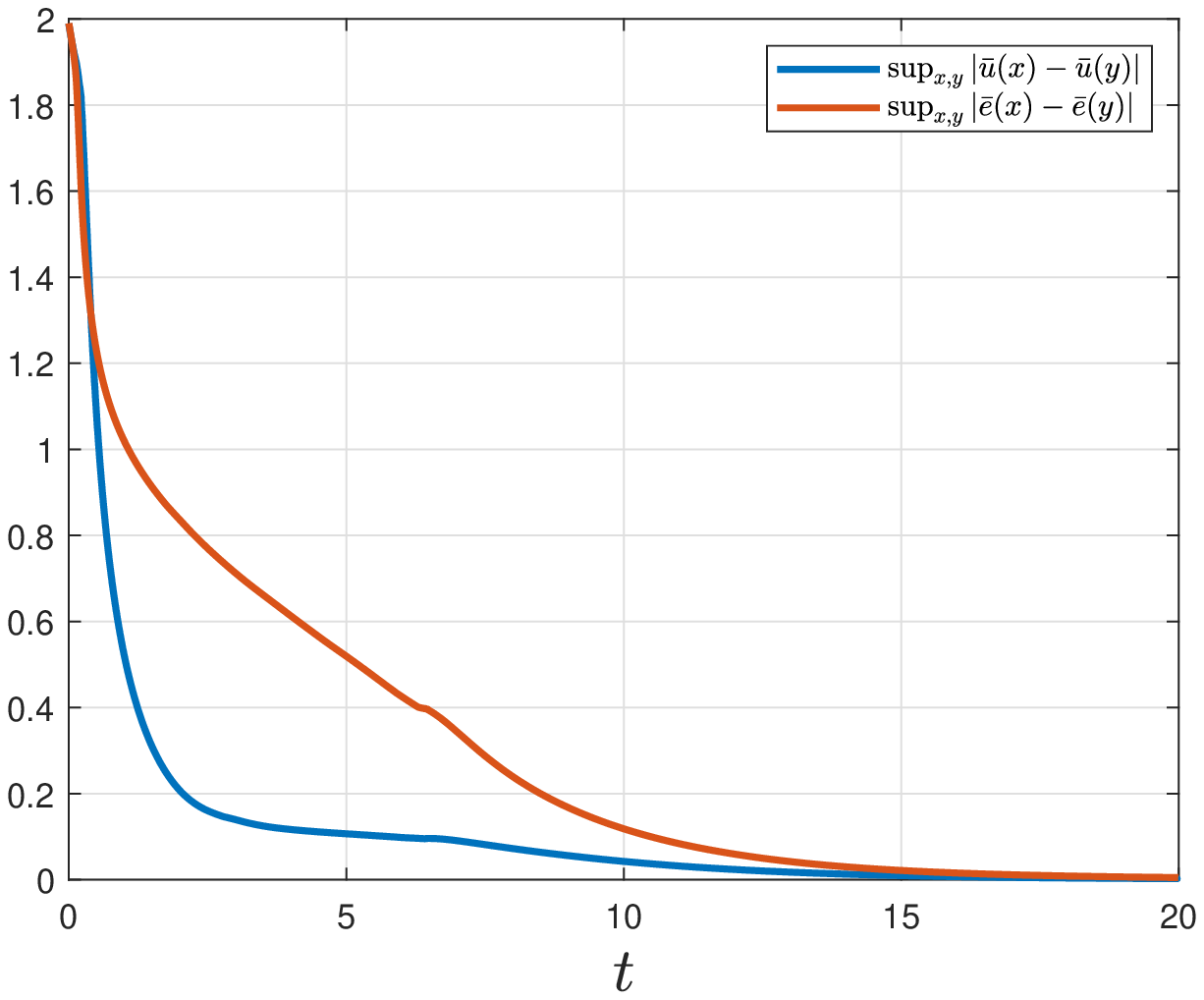}
			\caption{Dynamics of the fluctuation of $\bar{u}$ and $\bar{e}$.}
		\end{subfigure}
		\caption{Dynamics of the background fluid.}
		\label{fig:2}
	\end{figure}

\medskip

	$\bullet$ {\sc Numerical results}. First, in Figure \ref{fig:2} we present the results for the background fluid $(\bar{\rho},\bar{u},\bar{e})$. On the one hand, Figure \ref{fig:2}(A) shows the evolution of the profile of density $\bar{\rho}$. Initially, the density has been set so that it is accumulated around $x=0.5$. Since the initial velocity $\bar{u}^0(x)=0.5+\sin(2\pi x)$ is larger on $[0,0.5]$ than on $[0.5,1]$, that creates a compression for density. In addition, note that the averaged velocity takes the value
	\[\int_{\mathbb{T}}\bar{\rho}(0,x)\bar{u}(0,x)\,d x=\int_0^1 (0.5+\sin(2\pi x))\,d x=0.5>0.\]
	Hence, the density $\bar \rho$ moves slowly toward the positive direction. In Figure \ref{fig:2}(B), we present the asymptotic profile of $(\bar{\rho},\bar{u},\bar{e})$ at $t=20$. Since the background fluid satisfies the hydrodynamic  flocking equation, then the velocity $\bar u$ and the internal variable $\bar e$ are expected to become homogeneous over the spatial domain, for large enough time, see \cite{HKMRZ18}. Indeed, Figure \ref{fig:2}(B) shows exactly that the limit values of $\bar{u}$ and $\bar{e}$ are 0.5 and 2. Finally, Figure \ref{fig:2}(C) presents the relaxations of fluctuations of $\bar{u}$ and $\bar{e}$ along time $t$. Fluctuation of $\bar{u}$ and $\bar{e}$ are computed as the maximal oscillations of $\bar u$ and $\bar e$ over the spacial domain, i.e.,
	$$\sup_{x,y\in\mathbb{T}}|\bar{u}(t,x)-\bar{u}(t,y)|,\quad \sup_{x,y\in\mathbb{T}}|\bar{e}(t,x)-\bar{e}(t,y)|.$$
	As expected, fluctuations of $\bar{u}$ and $\bar{e}$ converge to $0$ and become negligible after $t\approx 15$.
	
	\begin{figure}[h!]
		\centering
		\begin{subfigure}[b]{0.45\textwidth}
			\includegraphics[width=\textwidth]{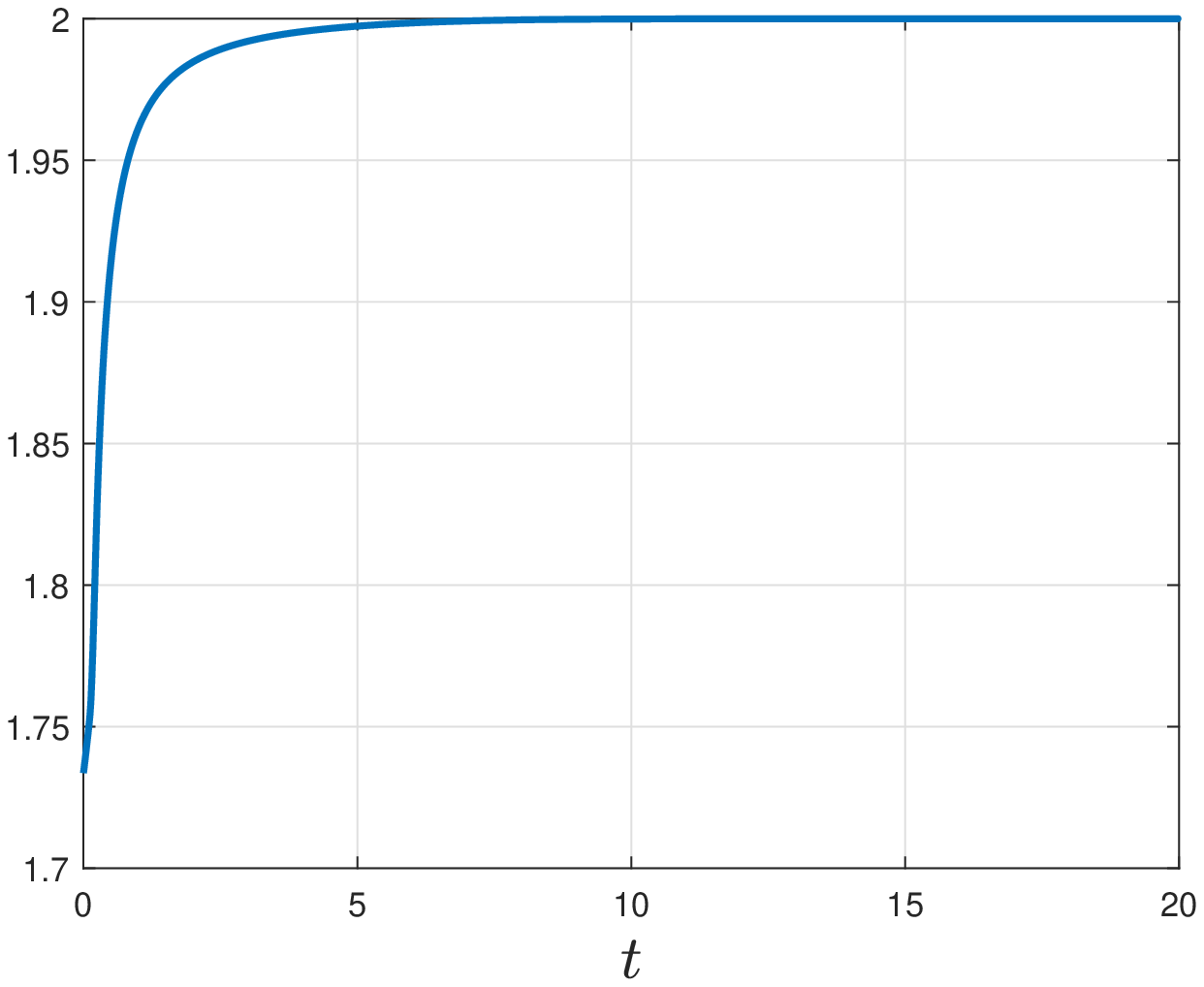}
			\caption{Dynamics of $\theta^\infty(t)$.}
		\end{subfigure}
		\begin{subfigure}[b]{0.45\textwidth}
			\includegraphics[width=\textwidth]{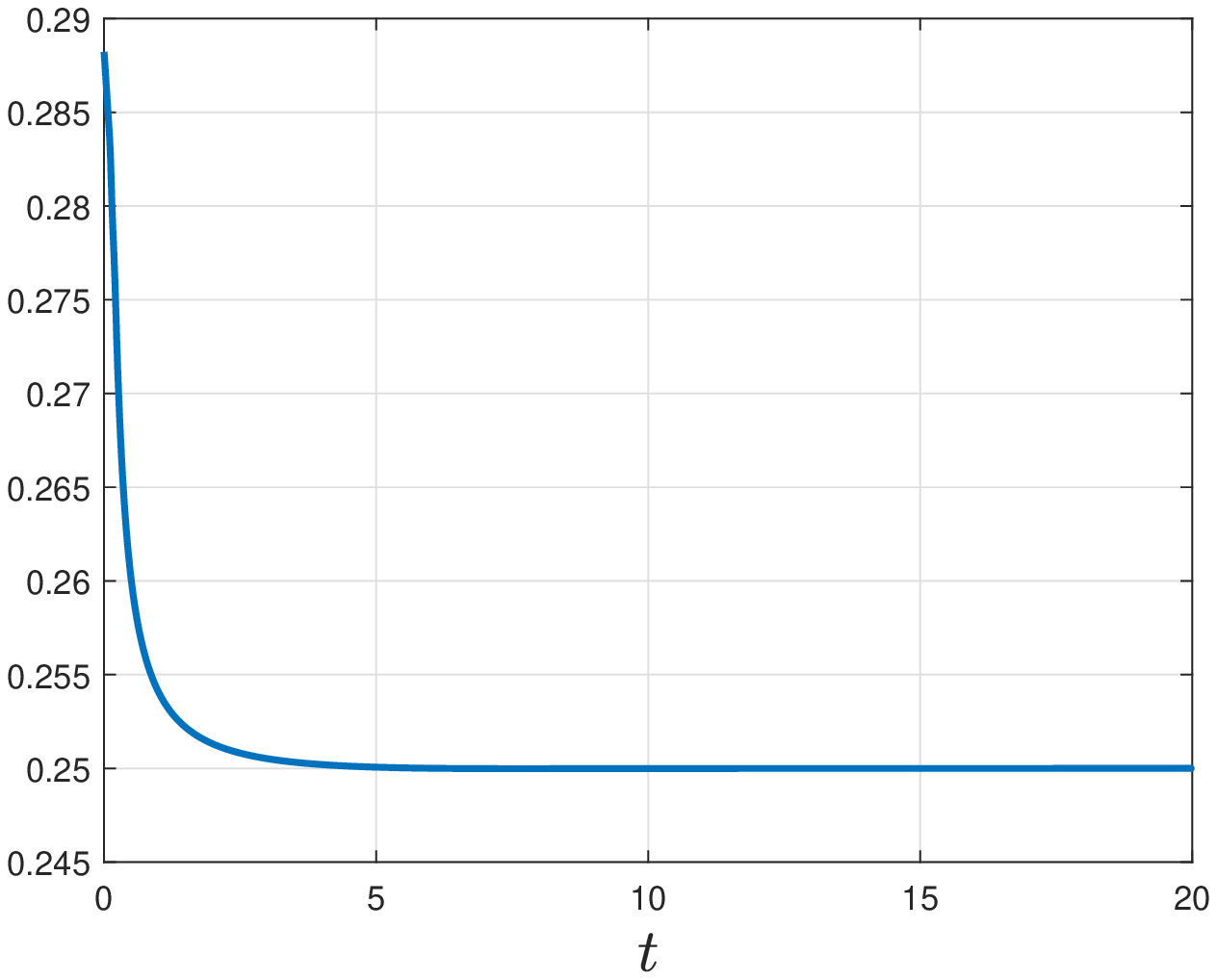}
			\caption{Dynamics of $\frac{u^\infty(t)}{\theta^\infty(t)}$.}
		\end{subfigure}
		\caption{Dynamics of two main quantities $\theta^\infty(t)$ and $\frac{u^\infty(t)}{\theta^\infty(t)}$.}
		\label{fig:3-1}
	\end{figure}
	
	Next, in Figure \ref{fig:3-1}(A) and (B) we observe the evolution of the average values $\theta^\infty(t)$ and  $\frac{u^\infty(t)}{\theta^\infty(t)}$ for the internal variable and the velocity of the background fluid. They appear in the right-hand side of \eqref{eq_u} and play a crucial role in the cross-interaction between the background fluid $(\bar\rho,\bar u,\bar e)$ and the limiting system $(\rho,u)$. Surprisingly, although $\bar{u}$ and $\bar{e}$ are not homogenized until $t\approx 15$, (recall Figure \ref{fig:2}(B),(C)), the quantities $\theta^\infty(t)$ and $\frac{u^\infty(t)}{\theta^\infty(t)}$ saturate earlier around $t\approx 5$. Moreover, the limit quantities are determined by the limit values of $\bar{u}$ and $\bar{e}$:
	\begin{align*}
	\lim_{t\to\infty} \theta^\infty(t)&=\lim_{t\to\infty} \left(\int_{\mathbb{T}}\frac{\bar{\rho}(t,x)}{\bar{e}(t,x)}\,dx\right)^{-1}=2,\\
	\lim_{t\to\infty} \frac{u^\infty(t)}{\theta^\infty(t)}&=\lim_{t\to\infty}\int_{\mathbb{T}}\frac{\bar{\rho}(t,x)\bar{u}(t,x)}{\bar{e}(t,x)}\,dx = 0.25.
	\end{align*}

	\begin{figure}[h!]
		\centering
		\begin{subfigure}[b]{1\textwidth}
			\includegraphics[width=\textwidth]{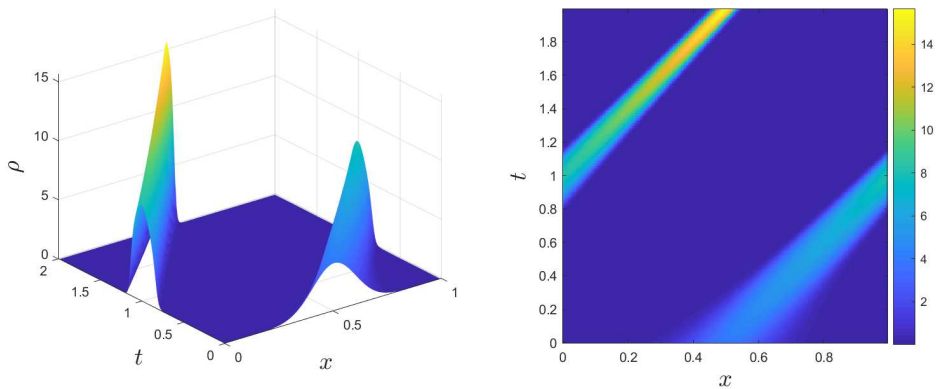}
			\caption{Profile of the density $\rho$.}
		\end{subfigure}
		
		\begin{subfigure}[b]{1\textwidth}
			\includegraphics[width=\textwidth]{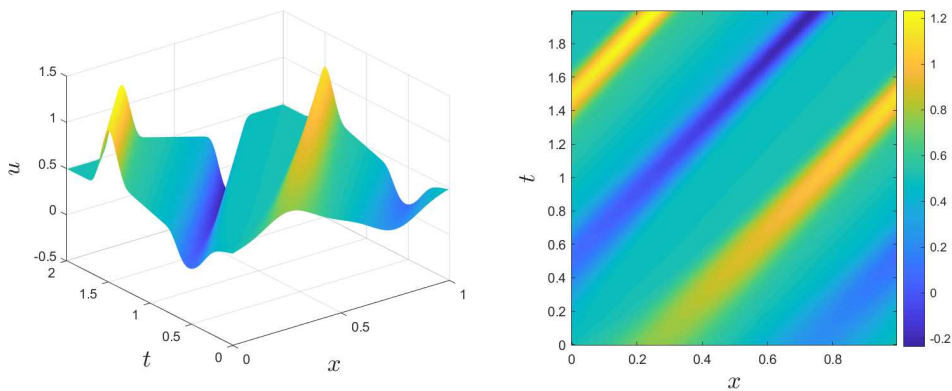}
			\caption{Profile of the velocity $u$.}
		\end{subfigure}
		
		\begin{subfigure}[b]{0.45\textwidth}
			\includegraphics[width=\textwidth]{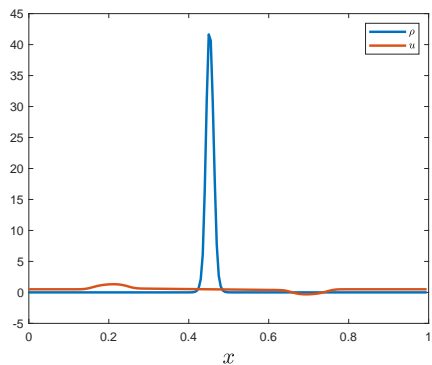}
			\caption{$(\rho,u)$ at $t=20$.}
		\end{subfigure}
		\begin{subfigure}[b]{0.45\textwidth}
			\includegraphics[width=\textwidth]{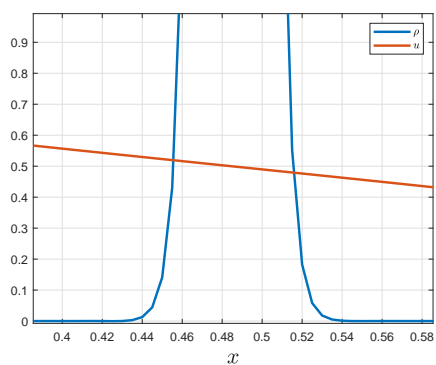}
			\caption{$(\rho,u)$ at $t=20$ (magnified).}
		\end{subfigure}
		\caption{Profiles of $\rho$ and $u$ for strong relaxation regime.}
		\label{fig:4}
	\end{figure}

	\begin{figure}[h!]
		\includegraphics[width=0.7\textwidth]{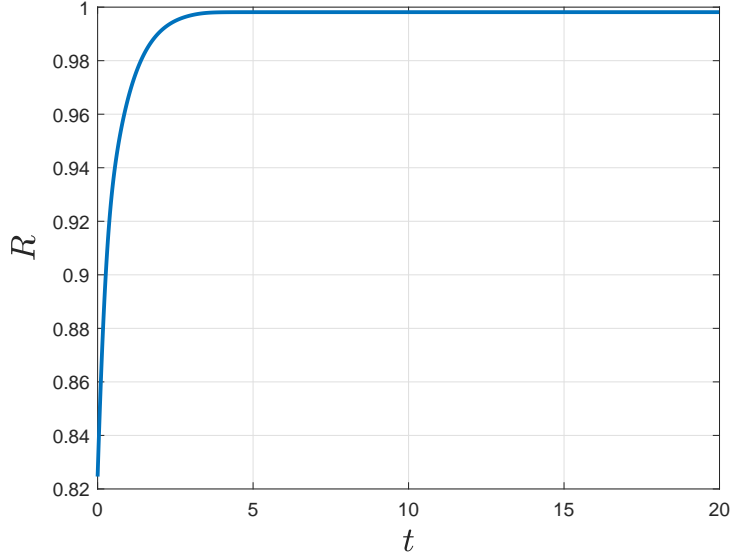}
		\caption{Dynamics of the order parameter in the strong relaxation regime.}
		\label{fig:order}
	\end{figure}
	
	In Figure \ref{fig:4} we present the results for the limiting system $(\rho,u)$. On the one hand, Figure \ref{fig:4}(A) and (B) shows the dynamics of $(\rho,u)$ along time $t$. We observe that the density $\rho$ starts to move toward the positive direction, as an effect of the background fluid. On the other hand, Figure \ref{fig:4}(C) shows the asymptotic profile of the density $\rho$ and velocity $u$ at time $t=20$. We observe that the density $\rho$ concentrates, due to the effect of the aggregation potential $W$. To quantify the degree of aggregation on the periodic domain, we introduce so-called the {\it order parameter} $R$, defined as
	\[R(t):=\left|\int_{\mathbb{T}} (\cos 2\pi x,\sin 2\pi x) \rho(t,x)\,d x\right|.\]
	Note that the range of $R$ is $[0,1]$ and $R=1$ when the density $\rho$ is concentrated on a single point. Therefore, the larger $R$, the more concentrated $\rho$ is. We present the dynamics of the order parameter in Figure \ref{fig:order}. The simulation result implies that the order parameter increases, which means that the density $\rho$ aggregates asymptotically. Finally, in Figure \ref{fig:4}(D), we magnify the same figure \ref{fig:4}(C) around the point at which $\rho$ concentrates. We observe that the velocity $u$ approximately takes the value $0.5$ at that point, which is the same value as the limit value of $u^\infty$.

	\subsection{Weak relaxation regime}
	
	In this part, we provide some numerical simulations of the limiting macroscopic system \eqref{A-7} under the weak relaxation regime. In periodic variables, such a system is determined by the following coupled equations for $(\rho,u,\theta(t))$
	\begin{align}\label{eq_rho_u_weak}
	\begin{aligned}
	&\partial_t\rho +\nabla\cdot (\rho u) = 0, & & (t,x)\in \mathbb{R}_+\times \mathbb{T},\\
	&u -\frac{\theta(t)}{\theta^\infty(t)}u^\infty(t)+\theta(t)(\nabla W*\rho)= \phi*(\rho u)-(\phi*\rho)u, & & (t,x)\in \mathbb{R}_+\times \mathbb{T},
	\end{aligned}
	\end{align}
	where $\theta(t)$ is defined by the solution of the relaxation ODE:
	\[\dot{\theta}(t) = \frac{1}{\theta(t)}-\frac{1}{\theta^\infty(t)},\quad t\in\mathbb{R}_+.\]
	We recover the same initial data $(\bar\rho^0,\bar u^0,\bar e^0)$ in \eqref{E-numerics-initial-data-macro-species} for the background fluid \eqref{eq_background} and $\rho^0$ in \eqref{E-numerics-initial-data-limit} for the limiting system \eqref{eq_rho_u_weak} like in Subsection \ref{subsec:5.1}. In addition, we set the choice $\theta(0)=5$ of initial internal variable. We skip the discussion about numerical methods, as we use similar schemes. In addition, since the dynamics of the background fluid is identical to the previous strong relaxation regime, we only focus on the dynamics of $(\rho,u)$ and we compare the new results with the strong relaxation regime in the previous Subsection \ref{subsec:5.1}. 
	
	\begin{figure}[h!]
		\centering
		\begin{subfigure}[b]{1\textwidth}
			\includegraphics[width=\textwidth]{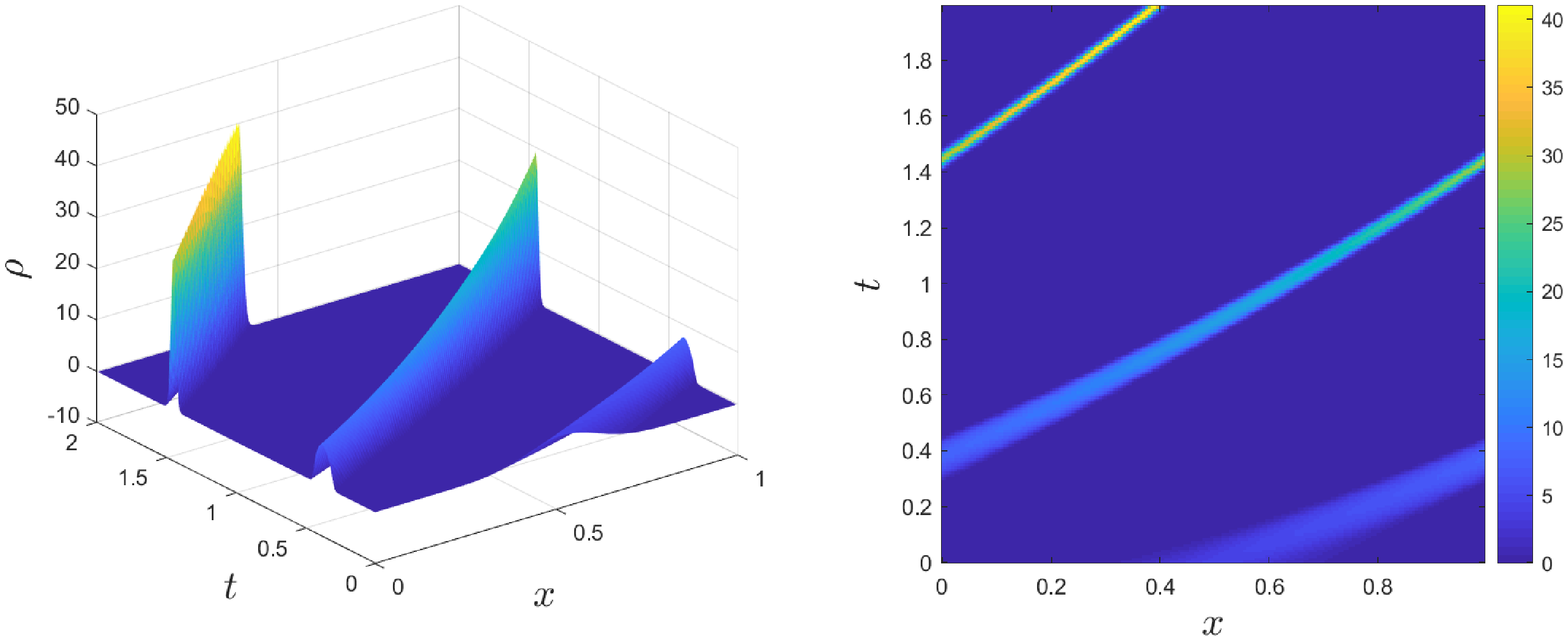}
			\caption{Profile of the density $\rho$.}
		\end{subfigure}
		
		\begin{subfigure}[b]{1\textwidth}
			\includegraphics[width=\textwidth]{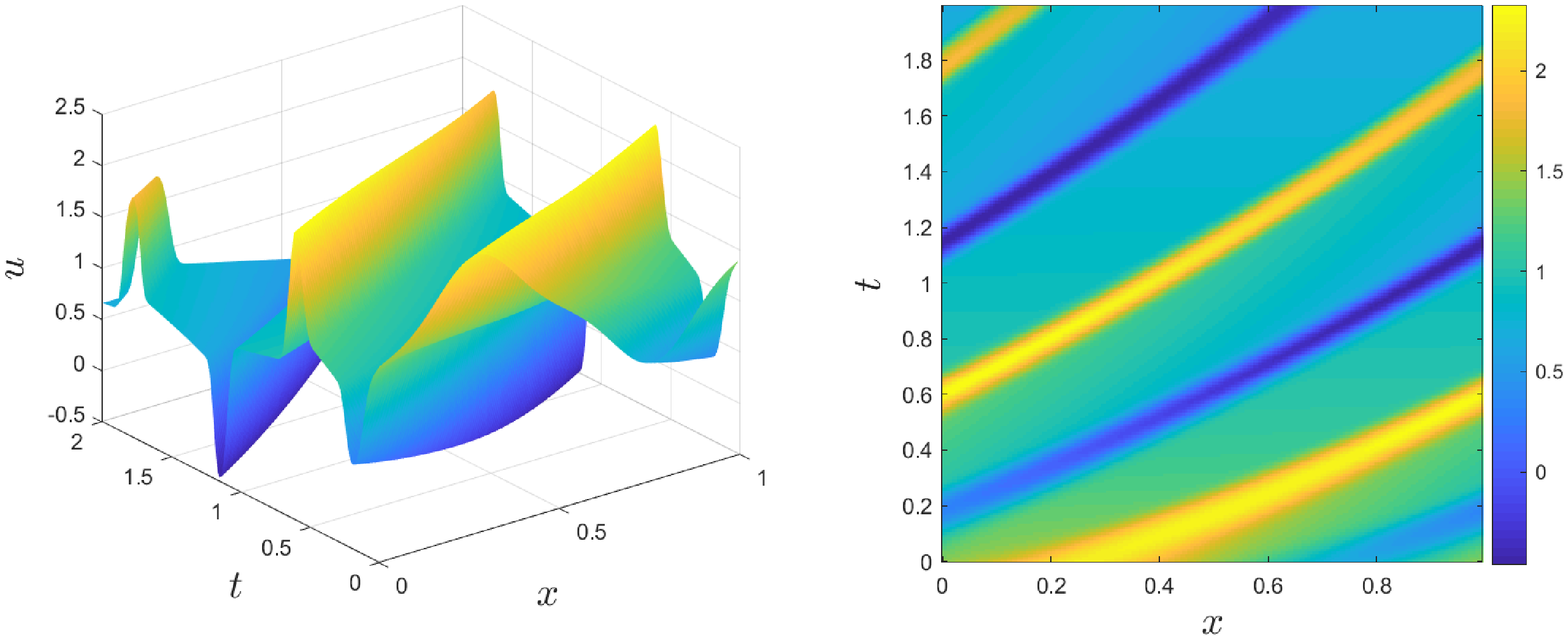}
			\caption{Profile of the velocity $u$.}
		\end{subfigure}
		
		\begin{subfigure}[b]{0.45\textwidth}
			\includegraphics[width=\textwidth]{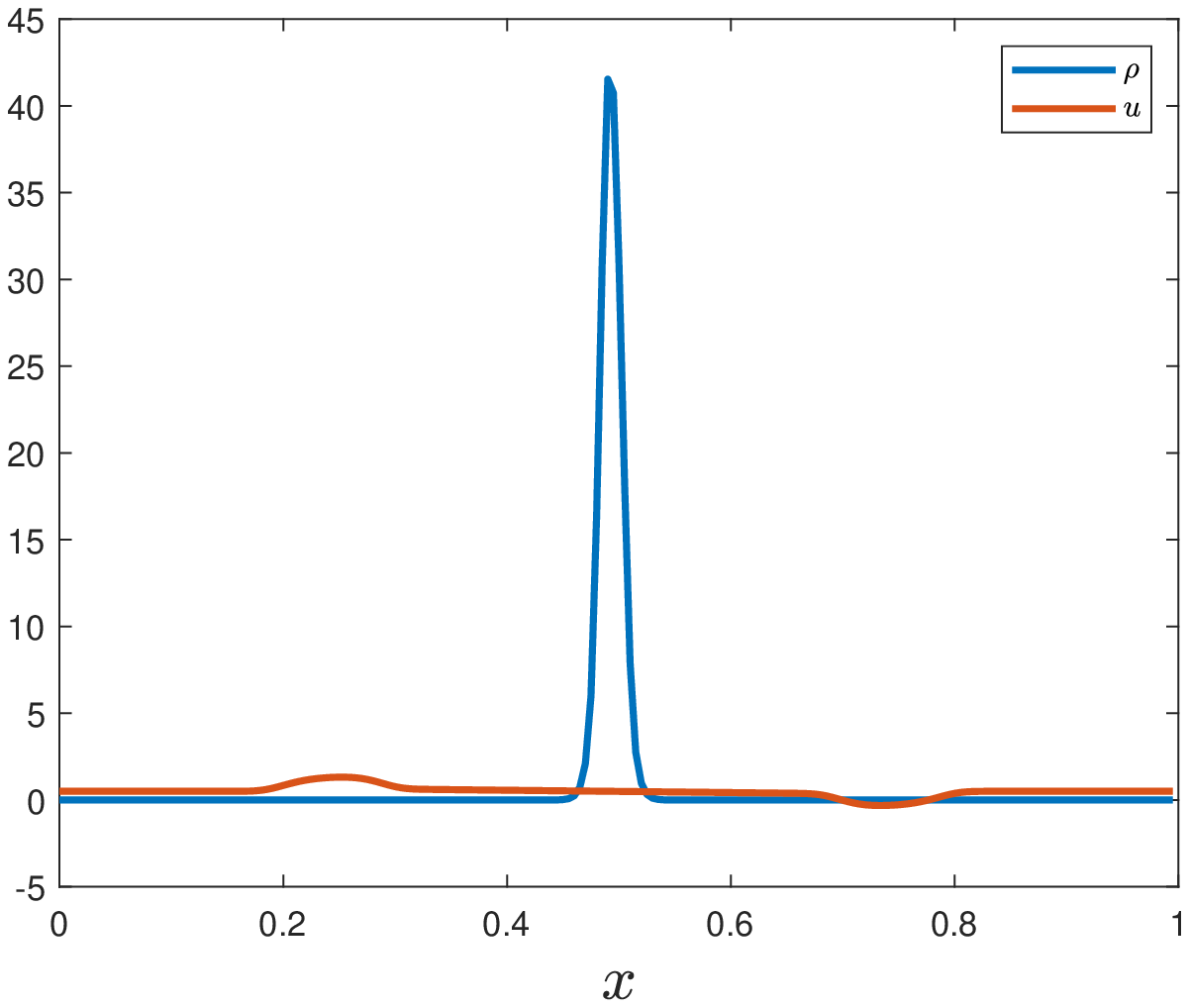}
			\caption{$(\rho,u)$ at $t=20$.}
		\end{subfigure}
		\begin{subfigure}[b]{0.45\textwidth}
			\includegraphics[width=\textwidth]{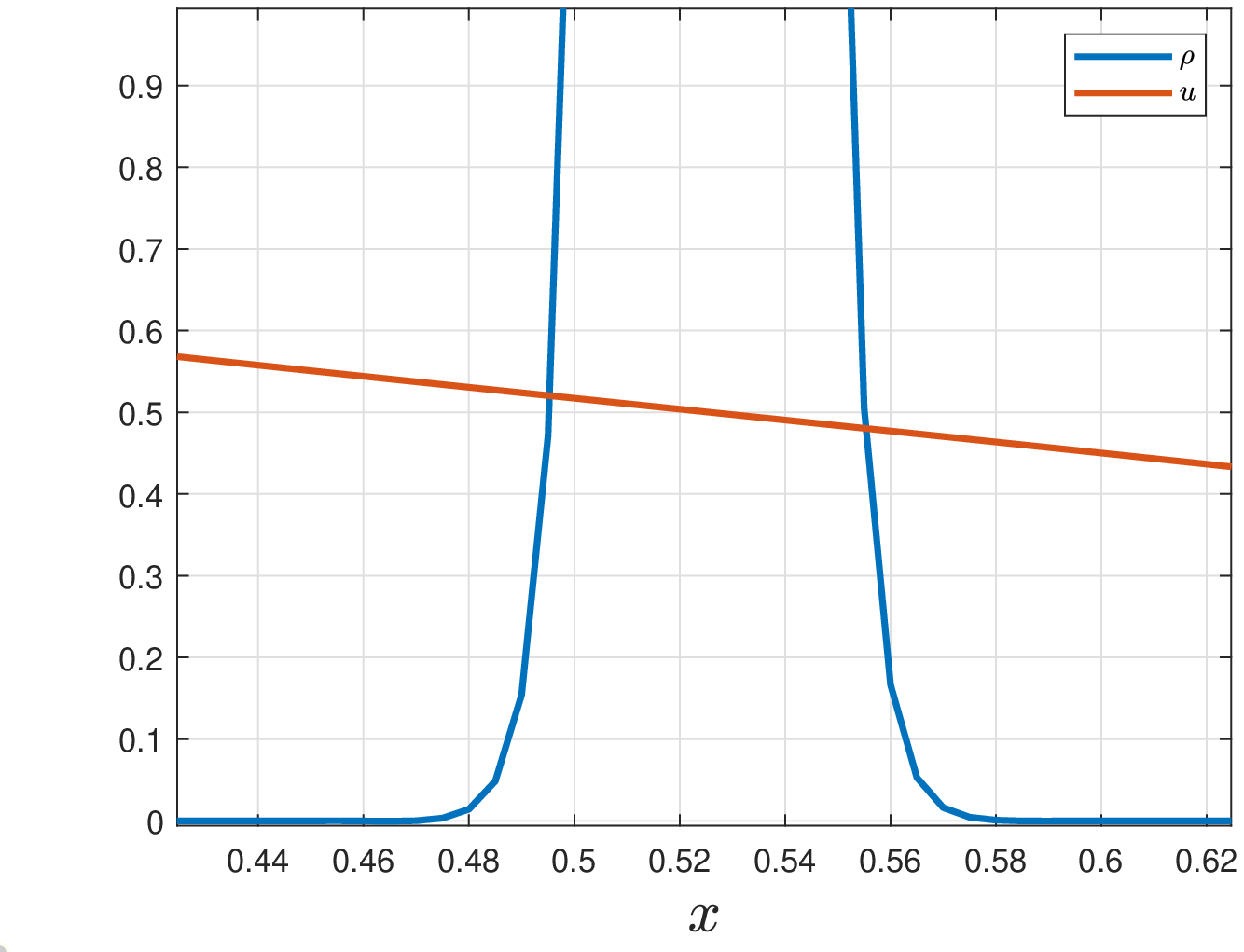}
			\caption{$(\rho,u)$ at $t=20$ (magnified).}
		\end{subfigure}
		
		\caption{Profiles of the density $\rho$ and $u$ for weak relaxation regime.}
		\label{fig:5}
	\end{figure}

	\begin{figure}[h!]
		\includegraphics[width=0.7\textwidth]{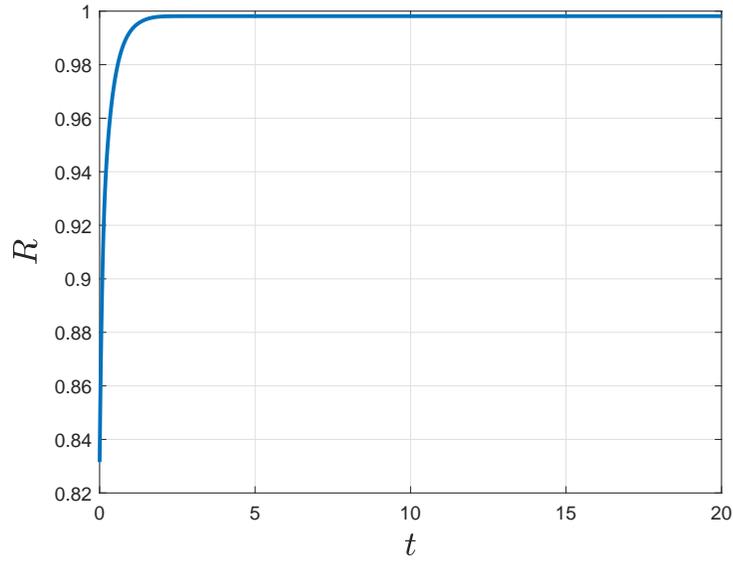}
		\caption{Dynamics of the order parameter in the weak relaxation regime.}
		\label{fig:order_weak}
	\end{figure}
	
	\begin{figure}[h!]
		\includegraphics[width=0.7\textwidth]{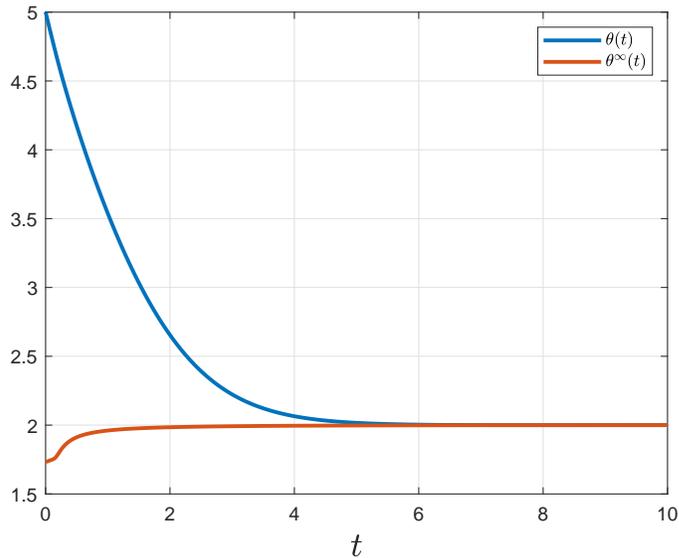}	
		\caption{Dynamics of $\theta(t)$ and $\theta^\infty(t)$.}
		\label{fig:6}
	\end{figure}
	
	In Figure \ref{fig:5}, we present the profiles of the solution $(\rho,u)$ of \eqref{eq_rho_u_weak}. We observe that the velocity $u$ is much larger than in the strong relaxation regime, and this is due to the higher initial value of the internal variable $\theta(t)$. Indeed, in Figure \ref{fig:6} we compare the dynamics of $\theta(t)$ and $\theta^\infty(t)$, and the result shows that the internal variable $\theta(t)$ is initially higher than $\theta^\infty(t)$ (due to the choice $\theta(0)=5$), and it relaxes toward $\theta^\infty(t)$ for larger time. The fact that higher internal variable $\theta(t)$ implies higher velocity $u$ can also be expected from the velocity equation $\eqref{eq_rho_u_weak}_2$.
	Namely, when the internal variable $\theta(t)$ is larger than $\theta^\infty(t)$, then the forcing term in the left-hand side of the implicit equation $\eqref{eq_rho_u_weak}_2$ is also larger, which leads to a larger value of the solution $u$. Consequently, the dynamics of density $\rho$ is also faster than in the strong relaxation regime, as shown in Figure \ref{fig:5}(A) and (B). Of course, when $\theta(0)$ is smaller than $\theta^\infty(0)$, then the situation is reversed, and the dynamics becomes slower than in the strong relaxation regime. 
	
	Moreover, we present the asymptotic profile of $(\rho,u)$ at time $t=20$ in Figure \ref{fig:5}(C). We point out that, after sufficiently large time, say $t\approx 20$, the solution of \eqref{eq_rho_u_weak} in the weak relaxation regime has almost the same shape as the solution of \eqref{eq_rho_u_weak} in the strong relaxation regime (see Figure \ref{fig:4}(C)), up to a position shift arising from the initial difference of velocities.  Therefore, we conclude that although the initial speed of the dynamics depends upon the initial value of the internal variable, the solution under the weak relaxation regime eventually converges to the solution of the strong relaxation regime. We also quantify aggregation by providing the dynamics of the order parameter $R$ in Figure \ref{fig:order_weak}. Compared to Figure \ref{fig:order}, we observe that the order parameter saturates faster than in the case of the strong relaxation regime, which also supports the fact that the weak relaxation regime leads to faster aggregation when the initial value $\theta(0)$ of the internal variable is higher than the background value $\theta^\infty(0)$. Finally, in Figure \ref{fig:5}(D) we present again the magnified Figure \ref{fig:5}(C) and we find that the asymptotic value of $u$ over the support of $\rho$ is also near $u^\infty=0.5$.
	
\appendix

\section{An extension of Riesz representation theorem}\label{Appendix-LB}
In this appendix, we recall some notation and basic concepts that are used systematically regarding Banach-valued $L^p$-type spaces and their dual representability. Specifically, we shall first the usual \textit{Lebesgue--Bochner spaces} $L^p(0,T;X)$ for any Banach space $X$ and we recall the corresponding representability of their topological dual spaces $L^p(0,T;X)^*$. This is an intriguing problem usually characterized by the \textit{Radon--Nikodym property (RNP)} of the topological dual $X^*$, that is known to fulfill if $X$ is reflexive or $X^*$ is separable. For $X^*$ verifying the RNP, we recall the classical \textit{Riesz representation theorem}, stating that 
$$L^p(0,T;X)^*\equiv L^{p'}(0,T;X^*).$$
Unfortunately, if $X^*$ fails RNP, the preceding representation is not valid. For those cases, we shall recall the \textit{Diculeanu-Foias theorem} stating that
$$L^p(0,T;X)^*\equiv L^{p'}_w(0,T;X^*),$$
where $L^{p'}_w(0,T;X)$ is the \textit{weak-* Bochner-Lebesgue spaces}.

\begin{definition}[Lebesgue--Bochner spaces \cite{DU77}]\label{Appendix-LB-D-LB}
Consider a Banach space $X$. We define
$$
L^p(0,T;X):=\left\{f:[0,T]\longrightarrow X:\,f\ \mbox{is measurable and}\ \Vert f\Vert_X\in L^p(0,T)\right\},
$$
for any $1\leq p\leq\infty$. So defined, $L^p(0,T;X)$ becomes a Banach with norm
$$\Vert f\Vert_{L^p(0,T;X)}=\Vert \Vert f\Vert_X\Vert_{L^p(0,T)}.$$
As usual, we shall identify functions that agree almost everywhere in $[0,T]$ by taking the quotient by the relation
$$f\sim g\Longleftrightarrow f(t)=g(t)\ \mbox{ for almost every }t\in [0,T].$$
\end{definition}

See \cite{DU77} for further insight about measurability and weak measurability (c.f., Pettis' measurability theorem) of Banach-valued functions.

\begin{theorem}[Riesz representation \cite{BT38,DU77}]\label{Appendix-LB-Riesz-representation-RNP}
Let $X$ be a Banach space, consider any exponent $1\leq p<\infty$ and define the mapping
$$\begin{array}{cccc}
\Phi_p: & L^{p'}(0,T;X^*) & \longrightarrow & L^p(0,T;X)^*,\\
 & f & \longmapsto & \Phi_p[f],
\end{array}\ \hspace{0.5cm}\ \left<\Phi_p[f],g\right>:=\int_{[0,T]}\left<f(t),g(t)\right>\,dt,$$
for any $g\in L^p(0,T;X)$. Then $\Phi_p$ is a linear isometry. In addition, $\Phi_p$ is surjective if, and only if, $X^*$ verifies RNP with respect to Lebesgue measure in $[0,T]$.
\end{theorem}

The typical criteria to test RNP are due to Philips, Dunford and Pettis and we summarize them in the following result, see \cite[Corollary III.2.13 and Theorem III.3.1]{DU77}.

\begin{proposition}
Let $X$ be a Banach space:
\begin{enumerate}
\item (Philips) If $X$ is reflexive, then $X$ has the RNP.
\item (Dunford-Pettis) If $X=Y^*$ is a separable dual, then $X$ has the RNP.
\end{enumerate}
\end{proposition}

In this paper, we are interested in applying the preceding duality result to several situations in order to endow the corresponding Lebesgue-Bochner space with a weak-* topology so that weak-* compactness can be derived from the Alaouglu-Bourbaki theorem. We illustrate the two typical examples:

\begin{enumerate}
\item If $X=L^q(\mathbb{R}^d)$ with $1<q<\infty$, then reflexivity guarantees
$$L^{p'}(0,T;L^{q'}(\mathbb{R}^d))\equiv L^p(0,T;L^q(\mathbb{R}^d))^*,$$
for any $1\leq p<\infty$.
\item If $X=C_0(\mathbb{R}^d)$, then $X^*=\mathcal{M}(\mathbb{R}^d)$. Of course, reflexivity is not true, so that the first criterion by Philips fails. Also, the map $x\in \mathbb{R}^d \longmapsto \delta_x\in\mathcal{M}(\mathbb{R}^d)$ embeds $\mathbb{R}^d$ into $\mathcal{M}(\mathbb{R}^d)$ as an uncountable and discrete subset. Therefore, $\mathcal{M}(\mathbb{R}^d)$ is not separable neither, so that Dunford-Pettis criterion fails too. Indeed, $\mathcal{M}(\mathbb{R}^d)$ fails RNP because $L^1(\mathbb{R}^d)$ is a subspace failing RNP, see \cite[Example 2.1.2]{B83}.
\end{enumerate}

For those cases, representation is achieved in terms of \textit{weak-* Lebesgue-Bochner spaces}.

\begin{definition}[$w^*$ Lebesgue--Boschner spaces \cite{F99,II62,II69,PK09}]\label{Appendix-LB-D-weak-LB}
Consider a Banach space $X$. We will define
$$
L^p_w(0,T;X^*):=\left\{
f:[0,T]\longrightarrow X^*: \begin{array}{c} 
\displaystyle\left<f,x\right>\in L^p(0,T)\ \mbox{ for all }\ x\in X,\\
\displaystyle\mbox{and }\sup_{\Vert x\Vert_X\leq 1}\Vert \left<f,x\right>\Vert_{L^p(0,T)}<\infty,
\end{array}\right\}
$$
for any $1\leq p\leq\infty$. So defined, $L^p_w(0,T;X^*)$ becomes a Banach space with norm
$$\Vert f\Vert_{L^p_w(0,T;X^*)}=\sup_{\Vert x\Vert_X\leq 1}\Vert\left<f,x\right>\Vert_{L^p(0,T)}.$$
Again, we identify it with its quotient by another (different) relation
$$f\approx g\Longleftrightarrow \left<f(t),x\right>=\left<g(t),x\right> \mbox{ a.e. } t\in [0,T],\ \mbox{ for any }\ x\in X.$$
Notice that for $\approx$, the negligible subset of $[0,T]$ depends on $x\in X$, as opposed to $\sim$ in Definition \ref{Appendix-LB-D-LB}.
\end{definition}

The most delicate point of Definition \ref{Appendix-LB-D-weak-LB} in contrast with Definition  \ref{Appendix-LB-D-LB}, is that functions $f\in L^p_w(0,T;X^*)$ are not necessarily measurable, but only weak-* measurable. In particular, $\sim$ and $\approx$ do not agree. However, they do agree if $X$ is separable. If indeed $X^*$ (thus $X$) is separable, measurability and weak-* separability agree and $L^p(0,T;X^*)=L^p_w(0,T;X^*)$. We end this section by recalling the \textit{Dinculeanu-Foias theorem}, see \cite{II62}, \cite[p. 95 and 99]{II69} and also \cite[Theorem 10.1.16]{PK09}, \cite[Theorems 12.2.11 and 12.9.2]{F99}.

\begin{theorem}[Dinculeanu-Foias]\label{Appendix-LB-Riesz-representation-no-RNP}
Let $X$ be a Banach space, set any exponent $1\leq p<\infty$ and define the mapping
$$
\begin{array}{cccc}
\widetilde{\Phi}_p: & L^{p'}_w(0,T;X^*) & \longrightarrow & L^p(0,T;X)^*,\\
 & f & \longmapsto & \widetilde{\Phi}_p[f],
\end{array}\hspace{0.3cm}\left<\widetilde{\Phi}_p[f],g\right>=\int_{[0,T]}\left<f(t),g(t)\right>\,dt,
$$
for any $g\in L^p(0,T;X)$. Then, $\widetilde{\Phi}_p$ is a surjective isometry.
\end{theorem}

In particular, we conclude that $L^p(0,T;C_0(\mathbb{R}^d))^*\equiv L^{p'}_w(0,T;\mathcal{M}(\mathbb{R}^d))$, for any $1\leq p<\infty$.

\end{document}